\def\version{last updated 08/11/2022 -- version 4
\hfill\href{https://arxiv.org/abs/2003.12003}{arXiv:2003.12003}
}

\documentclass[11pt,a4paper,centertags]{amsart}

\usepackage{fullpage}

\usepackage{amssymb,amscd,amsxtra}
\usepackage{autobreak}

\usepackage[alphabetic,lite]{amsrefs}

\usepackage[autobold]{mathfixs}

\usepackage[scr,scaled=1.1]{rsfso}

\usepackage{leftidx}

\usepackage[anticlockwise]{rotating}

\usepackage[hypertexnames=false]{hyperref}

\usepackage{xkvltxp}
\usepackage[nomargin,inline]{fixme}
\usepackage[]{todonotes}

\usepackage{mathtools}
\usepackage{relsize}

\usepackage{tikz}
\usetikzlibrary{arrows,calc,cd}

\usepackage{tikz-network}

\usepackage{xcolor}
\usepackage[all,poly,web,color,matrix,arrow]{xy}
\newdir{ >}{{}*!/-8pt/@{>}}
\usepackage{braket}

\usepackage{pigpen}

\renewcommand{\thefootnote}{\fnsymbol{footnote}}
\long\def\symbolfootnote[#1]#2{\begingroup%
\def\thefootnote{\fnsymbol{footnote}}\footnote[#1]{#2}\endgroup}

\newtheorem{thm}{Theorem}[section]
\newtheorem{lem}[thm]{Lemma}
\newtheorem{prop}[thm]{Proposition}
\newtheorem{cor}[thm]{Corollary}

\theoremstyle{definition}
\newtheorem{rem}[thm]{Remark}
\newtheorem*{rem*}{Remark}

\newtheorem{examp}[thm]{Example}

\numberwithin{equation}{section}

\def\ds{\displaystyle}
\def\:{\colon}
\def\.{\cdot}

\def\<{\left\langle}
\def\>{\right\rangle}
\def\({\left(}
\def\){\right)}
\def\ph#1{\phantom{#1}}
\def\epsilon{\varepsilon}
\def\phi{\varphi}

\def\leq{\leqslant}
\def\geq{\geqslant}

\def\la{\leftarrow}

\def\Lra{\Longrightarrow}

\def\bar#1{\overline{#1}}

\def\tilde#1{\widetilde{#1}}
\def\iso{\cong}

\DeclareMathOperator{\Id}{Id}

\DeclareMathOperator{\im}{im}

\def\F{\mathbb{F}}

\def\k{\Bbbk}

\def\RP{\mathbb{R}\mathrm{P}}

\def\Z{\mathbb{Z}}

\DeclareMathOperator{\diag}{diag}

\DeclareMathOperator{\Ext}{Ext}

\DeclareMathOperator{\Hom}{Hom}

\DeclareMathOperator{\Pic}{Pic}

\DeclareMathOperator{\Tor}{Tor}

\def\SO{\mathrm{SO}}
\def\Sp{\mathrm{Sp}}
\def\Spin{\mathrm{Spin}}

\def\SU{\mathrm{SU}}
\def\String{\mathrm{String}}

\def\bo{\mathrm{bo}}

\def\kO{{k\mathrm{O}}}

\def\kU{{k\mathrm{U}}}
\def\tmf{{\mathrm{tmf}}}
\DeclareMathOperator{\ann}{ann}

\DeclareMathOperator{\Sq}{Sq}

\def\dlQ{\mathrm{Q}}

\def\QS0{\dlQ S^0}
\def\QSo0{\dlQ_0S^0}
\def\StA{\mathcal{A}}
\def\StE{\mathcal{E}}
\def\J{\mathcal{J}}
\def\StMod{\mathbf{StMod}}

\def\op{\circ}
\def\sint{{\smallint}}
\DeclareMathOperator{\Adj}{Adj}

\def\Sage{\texttt{Sage}}

\title
[Homotopy theory over a commutative $S$-algebra:
some tools]
{Homotopy theory of modules over a commutative
$S$-algebra: some tools and examples}
\author{Andrew Baker}
\date{\version}
\address{
School of Mathematics \& Statistics,
University of Glasgow, Glasgow G12~8QQ, Scotland.}
\email{a.baker@maths.gla.ac.uk}
\urladdr{http://www.maths.gla.ac.uk/$\sim$ajb}
\thanks{
I would like to thank the following for
helpful comments and insights: Bob Bruner
and John Rognes from whom I learnt an
enormous amount about working with the
Steenrod algebra; Peter Eccles who taught
me how to use Toda brackets; Ken Brown
who helped me fine-tune an algebraic
result; Mike Hill who drew my attention
to~\cite{CLD-AGH-MAH:HomObstrStrOtns}. \\
The mathematics described in this paper
is based in part on work carried out while
the author was supported by the following
organisations:
the National Science Foundation under
Grant No.~0932078~000 while the author
was in residence at the Mathematical
Sciences Research Institute in Berkeley
California during the Spring~2014 semester;
Kungliga Tekniska H\"ogskolan and Stockholms
Universitet in Spring of 2018;
the Isaac Newton Institute for Mathematical
Sciences for support and hospitality during
the programme \emph{Homotopy Harnessing
Higher Structures} (supported by EPSRC
grant number EP/R014604/1);
the Max Planck Institute for Mathematics
in Bonn during a visit in January~2020.
}

\begin{document}

\begin{abstract}
Modern categories of spectra such as that
of Elmendorf et al equipped with strictly
symmetric monoidal smash products allow
the introduction of symmetric monoids
providing a new way to study highly
coherent commutative ring spectra. These
have categories of modules which are
generalisations of the classical categories
of spectra that correspond to modules
over the sphere spectrum; passing to
their derived or homotopy categories
leads to new contexts in which homotopy
theory can be explored.

In this paper we describe some of the
tools available for studying these
`brave new homotopy theories' and
demonstrate them by considering
modules over the connective
$K$-theory spectrum, closely related
to Mahowald's theory of $bo$-resolutions.

In a planned sequel we will apply
these techniques to the much less
familiar context of modules over
the $2$-local connective spectrum
of topological modular forms.
\end{abstract}

\maketitle

\tableofcontents

\section*{Introduction}

Modern categories of spectra such as that
of Elmendorf et al~\cite{EKMM} equipped
with strictly symmetric monoidal smash
products allow for the introduction of
symmetric monoids giving a new way to
study highly coherent commutative ring
spectra. In turn these have categories
of modules which are generalisations
of the classical categories of spectra
(corresponding to modules over the sphere
spectrum). For example, such categories
have Quillen model structures and so homotopy
(or derived) categories, thus allowing
the study of `brave new homotopy theories'.
This paper provides an introduction to
some of the machinery available for
engaging in this version of homotopy
theory and Sections~\ref{sec:HtpyRModules},
\ref{sec:CellModules} and \ref{sec:duality}
provide an overview on homotopy theory
for $R$-modules over a commutative
$S$-algebra~$R$ which should be sufficient
for reading the present work. Although
we only discuss connective spectra,
many aspects also apply to non-connective
settings with suitable modifications.

As an example we consider the important
case of~$\kO$ (the $2$-local connective
real $K$-theory spectrum) which is related
to Mahowald's theory of $\bo$-resolutions
and we review some aspects of this from
the present perspective. The case of~$\tmf$
(the $2$-local connective spectrum of
topological modular forms) is largely
waiting to be developed in the spirit
of Mahowald's work and in the planned
sequel we will discuss this, focusing
especially on examples associated
with~$\kO$ considered as a $\tmf$-module.
Throughout our aim will be to exhibit
interesting phenomena with connections
to classical homotopy theory.

Since we use make use of modules over
the mod~$2$ Steenrod algebra $\StA^*$
and its finite subHopf algebras $\StA(n)^*$,
we give some algebraic background in
Section~\ref{app:Steenrodmodules}.

\subsection*{Conventions \& notations}
In this paper we will mainly work locally
at the prime~$2$, so in that context~$H$
will denote the mod~$2$ Eilenberg-Mac~Lane
spectrum $H\F_2$ and~$\StA^*$ the mod~$2$
Steenrod algebra.

To avoid excessive display of gradings we
will usually suppress cohomological degrees
and write~$V$ for a cohomologically graded
vector space~$V^*$; in particular we will
often write~$\StA$ for the Steenrod algebra.
The linear dual of~$V$ is $\mathrm{D}V$
where
$(\mathrm{D}V)^k = \Hom_{\F_2}(V^{-k},\F_2)$,
and we write $V[m]$ for graded vector space
with $(V[m])^k=V^{k-m}$, so for the cohomology
of a spectrum~$X$, $H^*(\Sigma^m X)=H^*(X)[m]$.
For a connected graded algebra $\mathcal{B}^*$
we will often just write $\mathcal{B}$, and
denote its positive degree part by $\mathcal{B}^+$.

When working with left modules over a Hopf
algebra~$\mathcal{B}$ over a field we will
write $M\odot N$ to denote the vector space
tensor product of two~$\mathcal{B}$-modules~$M$
and~$N$ equipped with the diagonal action
defined using the coproduct in $\mathcal{B}$;
for a vector space~$V$, $M\otimes V$ will
denote the left module with action obtained
from the action on~$M$. The product~$\odot$
defines the monoidal structure on the
category of left~$\mathcal{B}$-modules.
It is well known that
$\mathcal{B}\odot N\iso\mathcal{B}\otimes N$
as left $\mathcal{B}$-modules, so~$\odot$
descends to the stable module category
$\StMod_{\mathcal{B}}$.

When discussing modules we will often follow
well established precedent and use diagrams
such as those in the figure below
%Figure~\ref{fig:A1diagrams}
which both represent~$\StA(1)$ as a left
$\StA(1)$-module. We will usually interpret
degrees as cohomological so that Steenrod
actions are displayed pointing upwards and
we usually suppress arrow heads; labels on
vertices denote degrees but are often omitted.
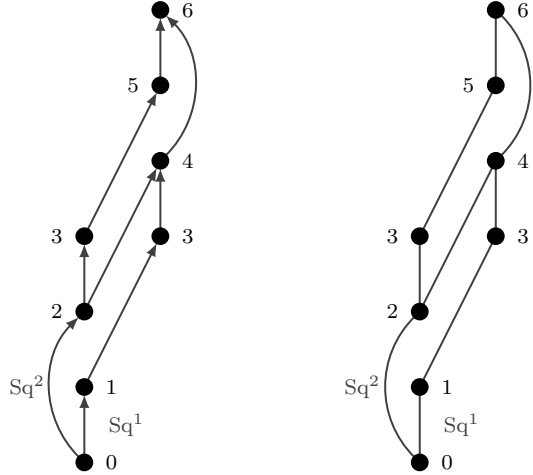
\begin{equation*}
\begin{tikzpicture}
\Vertex[x=1,y=6,color=black,size=.05,label={6},position=right,distance=0.5mm]{A6}
\Vertex[x=1,y=5,color=black,size=.1,label={5},position=left,distance=0.5mm]{A5}
\Vertex[x=1,y=4,size=.1,color=black,label={4},position=right,distance=0.5mm]{A4}
\Vertex[y=3,size=.1,color=black,label={3},position=left,distance=0.5mm]{A31}
\Vertex[x=1,y=3,size=.1,color=black,label={3},position=right,distance=0.5mm]{A32}
\Vertex[y=2,size=.1,color=black,label={2},position=left,distance=0.5mm]{A2}
\Vertex[y=1,size=.1,color=black,label={1},position=right,distance=0.5mm]{A1}
\Vertex[y=0,size=.1,color=black,label={0},position=right,distance=0.5mm]{A0}
\Edge[lw=0.75pt,Direct](A5)(A6)
\Edge[lw=0.75pt,bend=-45,Direct](A4)(A6)
\Edge[lw=0.75pt,Direct](A32)(A4)
\Edge[lw=0.75pt,Direct](A31)(A5)
\Edge[lw=0.75pt,Direct](A2)(A31)
\Edge[lw=0.75pt,Direct](A2)(A4)
\Edge[lw=0.75pt,Direct](A1)(A32)
\Edge[lw=0.75pt,label={$\Sq^1$},position=right,Direct](A0)(A1)
\Edge[lw=0.75pt,label={$\Sq^2$},position=left,bend=45,Direct](A0)(A2)
\end{tikzpicture}
\qquad\qquad
\begin{tikzpicture}
\Vertex[x=1,y=6,color=black,size=.05,label={6},position=right,distance=0.5mm]{A6}
\Vertex[x=1,y=5,color=black,size=.1,label={5},position=left,distance=0.5mm]{A5}
\Vertex[x=1,y=4,size=.1,color=black,label={4},position=right,distance=0.5mm]{A4}
\Vertex[y=3,size=.1,color=black,label={3},position=left,distance=0.5mm]{A31}
\Vertex[x=1,y=3,size=.1,color=black,label={3},position=right,distance=0.5mm]{A32}
\Vertex[y=2,size=.1,color=black,label={2},position=left,distance=0.5mm]{A2}
\Vertex[y=1,size=.1,color=black,label={1},position=right,distance=0.5mm]{A1}
\Vertex[y=0,size=.1,color=black,label={0},position=right,distance=0.5mm]{A0}
\Edge[lw=0.75pt](A5)(A6)
\Edge[lw=0.75pt,bend=-45](A4)(A6)
\Edge[lw=0.75pt](A32)(A4)
\Edge[lw=0.75pt](A31)(A5)
\Edge[lw=0.75pt](A2)(A31)
\Edge[lw=0.75pt](A2)(A4)
\Edge[lw=0.75pt](A1)(A32)
\Edge[lw=0.75pt,label={$\Sq^1$},position=right](A0)(A1)
\Edge[lw=0.75pt,label={$\Sq^2$},position=left,bend=45](A0)(A2)
\end{tikzpicture}
\end{equation*}

\section{Homotopy theory of $R$-modules}\label{sec:HtpyRModules}

We adopt the terminology and notation of~\cite{EKMM}.
Initially we do not necessarily assume spectra
are localised (or completed) at some prime
although this possibility would not affect
the generalities described here. Later we
do focus on some specifically local aspects
in order to incorporate notions from~\cite{AJB&JPM}.

Let $R$ be a connective commutative $S$-algebra.
In the category of $R$-modules $\mathscr{M}_R$
let $S^n_R$ denote the functorial cofibrant
replacement of the suspension $\Sigma^nR$.
When $R=S$ we will often suppress~$S$ from
notation and for an $R$-module~$M$ also
write~$M$ for the underlying $S$-module.

A morphism $f\colon X\to Y$ in
$\mathscr{M}=\mathscr{M}_S$ gives rise
to a morphism $R\wedge X\to R\wedge Y$
in $\mathscr{M}_R$, namely $1_r\wedge f$.
If~$M$ is an $R$-module, a morphism
$g\colon X\to M$ in $\mathscr{M}$ gives
rise to a morphism $R\wedge X\to M$.
\[
\xymatrix{
R\wedge X\ar[r]^{1_r\wedge g}\ar@/_15pt/[rr]
& R\wedge M\ar[r] & M \\
&&
}
\]

For an $R$-module~$M$,
\[
\pi_n(M) \iso \mathscr{D}_S(S^n,M)
         \iso \mathscr{D}_R(S^n_R,M).
\]

Now suppose that $R$ admits a morphism
of commutative $S$-algebras $R\to H=H\F_p$;
for example, we might have $\pi_0R=\Z$
or $\pi_0R=\Z_{(p)}$. Then we can define
a homology theory $H^R_*(-)$ on $\mathscr{D}_R$
by setting
\[
H^R_*M = \pi_*(H\wedge_RM)
\]
when $M$ is cofibrant. This theory has
a dual cohomology theory $H_R^*(-)$
defined by
\[
H_R^*M = \mathscr{D}_R(M,H).
\]
and which satisfies strict duality
\[
H_R^nM \iso \Hom_{\F_p}(H^R_nM,\F_p).
\]

One approach to calculating is by using the
K\"unneth spectral sequence of~\cite{EKMM},
\begin{equation}\label{eq:KunnethSS}
\mathrm{E}^2_{s,t} = \Tor^{H_*R}_{s,t}(\F_p,H_*M)
\Lra H^R_*M
\end{equation}
which results from the isomorphism
of $H$-modules
\[
H\wedge_R M \iso H\wedge_{H\wedge R}(H\wedge M).
\]
Taking $M=H$ we obtain a spectral
sequence
\[
\Tor^{H_*R}_{s,t}(\F_p,H_*H)
 \Lra H^R_*H
\]
which is known to be multiplicative.
In a case where the induced homomorphism
$H_*R\to H_*H=\StA_*$ is a monomorphism,
this is especially useful. For $p=2$,
each of the cases
$R=H\Z,\kO,\kU,\tmf,\tmf_1(3)$ has this
property. In such cases the dual Steenrod
algebra $\StA_*$ is a free module over
$H_*R$ and we obtain
\[
H^R_*H \iso \StA_*\otimes_{H_*R}\F_p
= \StA_*/\!/H_*R.
\]
Dually we have
\[
H_R^*H = \mathscr{D}_R(H,H)
\]
and a spectral sequence
\[
\mathrm{E}_2^{s,t} =
\Ext_{H_*R}^{s,t}(H_*H,\F_p)
 \Lra H_R^*H.
\]
In the situation where $\StA(p)_*$
is a free module over $H_*R$, this
gives
\[
H_R^nH \iso \Hom_{\F_p}(H^R_nH,\F_p).
\]
We also have
\[
H^*R \iso \StA(p)^*\otimes_{H_R^*H}\F_p.
\]

As in the case $R=S$, $H_R^*H$ is a
cocommutative Hopf algebra and its
dual is a commutative Hopf algebra.
Furthermore, $H_R^*H$ has a natural
left coaction on $H^R_*M$ which
induces a right action of $H^R_*H$
on $H_R^*M$ and a left action on
$H_R^*M$. Using the algebra
isomorphism
\[
\StA\xrightarrow{\;\iso\;}\StA^\op;
\quad
\theta\leftrightarrow(\chi\theta)^\op
\]
we can convert $H^R_*M$ into a left
$H_R^*H$-module with the grading
convention that $H^R_nM$ is in
cohomological degree~$-n$ (so positive
degree elements of $H_R^*H$ act on
$H^R_*M$ by lowering degrees).

By using a geometric resolution of
$\F_p$ over $H_*R$ to compute the
$\mathrm{E}^2$-term, it can be shown
that~\eqref{eq:KunnethSS} is a spectral
sequence of left $H_R^*H$-comodules
and right $H_R^*H$-modules.

In~\cite{AB&AL:ASS} we showed that
there is an Adams spectral sequence
for computing $\mathscr{D}_R(L,M)$
for $R$-modules $L,M$ with $L$
strongly dualisable. Dualising from
$H_R^*H$-comodules to $H_R^*H$-modules
this has the form
\begin{equation}\label{eq:ASS-homology}
\mathrm{E}_2^{s,t}(L,M)
= \Ext_{H_R^*H}^{s,t}(H^R_*L,H^R_*M)
\Lra \mathscr{D}_R(\Sigma^{t-s}L,M)\sphat,
\end{equation}
where $(-)\sphat$ denotes $p$-adic completion.
When $L=S^0_R$ we will set
$\mathrm{E}_r^{*,*}(M)=\mathrm{E}_r^{*,*}(L,M)$.
Assuming appropriate finiteness conditions
this $\mathrm{E}_2$-term can be rewritten
to give
\begin{equation}\label{eq:ASS-cohomology}
\mathrm{E}_2^{s,t} =
\Ext_{H_R^*H}^{s,t}(H_R^*M,H_R^*L)
\Lra \mathscr{D}_R(\Sigma^{t-s}L,M)\sphat.
\end{equation}
When $M=R\wedge Z$, this gives
\[
\mathrm{E}_2^{s,t}(L,R\wedge Z)
= \Ext_{H_R^*H}^{s,t}(H^R_*L,H_*Z)
\iso
\Ext_{H_R^*H}^{s,t}(H^*Z,H_R^*L).
\]
Finally, if $L=S^0_R$,
\[
\mathrm{E}_2^{s,t}(R\wedge Z)
= \Ext_{H_R^*H}^{s,t}(\F_p,H_*Z)
\iso
\Ext_{H_R^*H}^{s,t}(H^*Z,\F_p).
\]
This agrees with the classical Adams
$\mathrm{E}_2$-term
\[
\Ext_{\StA^*}^{s,t}(H^*(R\wedge Z),\F_p)
\Lra \pi_{t-s}(R\wedge Z)
\iso \mathscr{D}_R(S^{t-s}_R,R\wedge Z).
\]

Notice that $\mathrm{E}_r^{*,*}(R)$ is a bi-graded
commutative algebra and for any $R$-module~$M$,
$\mathrm{E}_r^{*,*}(M)$ is a spectral sequence
over it. When $A$ is an $R$-ring spectrum,
$\mathrm{E}_r^{*,*}(A)$ is a spectral sequence
of $\mathrm{E}_r^{*,*}(R)$-algebras.

For examples such as $R=\kO,\kU,\tmf,\tmf_1(3)$
with $p=2$, we know that $H_*R$ is isomorphic
to a left $\StA_*=H_*H$-comodule algebra and
then $H_R^*H$ is a subHopf algebra of the
Steenrod algebra $\StA=H^*H$. In these cases
$\StA$ is a free $H_R^*H$-module.

Now recall that for any ring homomorphism
$A\to B$, left $A$-module $U$ and left
$B$-module~$V$, there is a Cartan-Eilenberg
change of rings spectral sequence of the
form
\[
\mathrm{E}_2^{p,q} =
\Ext_{B}^{q}(\Tor^A_p(B,U),V)
\Lra \Ext_A^{q-p}(U,V).
\]
If $B$ is $A$-flat this collapses to give
\[
\Ext_A^*(U,V) \iso \Ext_B^*(B\otimes_A U,V).
\]
When $A=H_R^*H$ and $B=\StA$ we obtain
\[
\Ext_{H_R^*H}^*(\F_2,\F_2) \iso
\Ext_{\StA}^*(\StA\otimes_{H_R^*H}\F_2,\F_2)
\iso \Ext_{\StA}^*(H^*R,\F_2)
\]
which is the $\mathrm{E}_2$-term of the Adams
spectral sequence for computing~$\pi_*R$; of
course, this is a standard change of rings
isomorphism.

\section{Cell and CW $R$-modules}\label{sec:CellModules}
Now we assume that $R$ is $p$-local for some
prime~$p$ and also $(-1)$-connected with
$R_0=\pi_0R$ a acyclic $\Z_{(p)}$-module. This
means that is a local graded ring whose maximal
ideal $\mathfrak{m}\lhd\pi_*R$ consists of the
ideal $(p)\lhd R_0$ together with all positive
degree elements.

The notions of cell and CW $R$-modules have
the usual forms described in~\cite{EKMM}.
The $n$-skeleton $X^{[n]}$ of such a cell
or CW $R$-module~$X$ is obtained from the
$(n-1)$-skeleton $X^{[n-1]}$ by forming
a cofibre sequence of form
\begin{equation}\label{eq:Cell-nskeleton}
\bigvee_iS_R^{n-1}\xrightarrow{j^{n-1}}
X^{[n-1]}\to X^{[n]}.
\end{equation}
Here we take $S_R^k$ to be the functorial
cofibrant replacement of $\Sigma^k R$ (recall
that~$R$ is not cofibrant in the model category
$\mathcal{M}_R$ of~\cite{EKMM}). Associated
with such a cell $R$-module~$X$ there is a
cellular chain complex of $R_0$-modules
$\mathrm{C}^{\mathrm{CW}}_*(X)$ satisfying
\begin{equation}\label{eq:Cell-homology}
\mathrm{H}_*(\mathrm{C}^{\mathrm{CW}}_*(X)\otimes_{R_0}\F_p)
\iso H^R_*X.
\end{equation}

It is often useful to work with a minimal
cell structure; we adapt the notion of
minimal cell structure from~\cite{AJB&JPM}
to the present context. A CW $R$-module~$X$
is \emph{minimal} if for every~$n$, the
attaching map
$j^{n-1}$ of~\eqref{eq:Cell-nskeleton}
satisfies
\[
\im[j^{n-1}_*\colon\pi_*(\bigvee_iS_R^{n-1})
\to\pi_*X^{[n-1]}]
\subseteq \mathfrak{m}\pi_*X^{[n-1]}.
\]
It is easy to see that every connective
$R$-module with finite type homotopy can
be realised by a finite type connective
minimal CW $R$-module. Furthermore,
by~\eqref{eq:Cell-homology} the cells
of this give an $\F_p$-basis for $H^R_*X$,
and the inclusion map induces a monomorphism
\[
H^R_{n-1}X^{[n-1]} \to H^R_{n-1}X^{[n]}.
\]

In order to describe CW modules or their
(co)homology, we will often use cell
diagrams or diagrams showing bases
with action of $H_R^*H$. Usually we will
assume a minimal cell structure has been
used so cells will correspond to basis
elements. For example, when $R=S$, the
mapping cone of the Hopf map
$\eta\colon S^1\to S^0$, $X=S^0\cup_\eta e^2$,
and its mod~$2$ cohomology, $H^*X$, can
be represented by diagrams such as the
following.
%\[
%\xymatrix{
%*+[o][F]{2} \\
%*+[o][F]{0}\ar@{-}[u]^\eta
%}
%\qquad\qquad\qquad
%\xymatrix{
%\bullet \\
%\bullet\ar@/_5pt/[u]_{\Sq^2}
%}
%\]
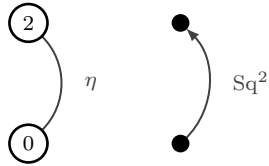
\begin{center}
\begin{tikzpicture}
\Vertex[x=-1,y=1.6,size=.5,label={$2$},color=white]{A2}
\Vertex[x=-1,y=0,size=.5,label={$0$},color=white]{A0}
\Edge[lw=0.75pt,bend=-45,label={$\eta$},position=right](A0)(A2)
\Vertex[x=1,y=1.6,size=.05,color=black]{B2}
\Vertex[x=1,y=0,size=.05,color=black]{B0}
\Edge[lw=0.75pt,Direct,bend=-45,label={$\Sq^2$},position=right](B0)(B2)
\end{tikzpicture}
\end{center}
Such diagrams are standard in homotopy
theory, for example they are discussed
by Barratt, Jones \&
Mahowald~\cite{MGB-JDSJ-MEM:Relns}.

Occasionally we will require more
complicated diagrams which involve
maps between more general objects
than just spheres. For example,
given maps $g\colon X\to Y$ and
$h\colon Y\to Z$ with $hg$ null
homotopic, we can factor
$\Sigma g\colon\Sigma X\to\Sigma Y$
through the mapping cone of~$h$ and
so define an object which is the
mapping cone of a factor
$\Sigma X \to Z\cup_h CY$ represented
by the following diagram.
\[
\xymatrix{
\Sigma^2X \\
\Sigma Y\ar@{-}[u]^{\Sigma g} \\
Z\ar@{-}[u]^h
}
\]

Here is a useful result on building
objects realising such diagrams; it
generalises well known results for
maps between spheres. Part (a) features
prominently in~\cite{MGB-JDSJ-MEM:Relns}
while in~\cite{AB:Joker} we used~(b)
to construct realisations of Joker-like
modules.
\begin{lem}\label{lem:TodaBracket}
Suppose that $f\colon W\to X$, $g\colon X\to Y$
and $h\colon Y\to Z$ where~$gf$ and
$hg$ are null homotopic; therefore
the Toda bracket $\langle h,g,f\rangle_R
\subseteq\mathscr{D}_R(\Sigma W,Z)$
is defined. \\
\emph{(a)} It is possible to realise
the diagram
\[
\xymatrix{
\Sigma^3W \\
\Sigma^2X\ar@{-}[u]^{\Sigma^2 f} \\
\Sigma Y\ar@{-}[u]^{\Sigma g} \\
Z\ar@{-}[u]^h
}
\]
if and only if the Toda bracket
$\langle h,g,f\rangle_R$ contains zero. \\
\emph{(b)} If there are maps $k\:\Sigma W\to U$
and $\ell\:U\to Z$ so that
$\ell k\in\langle h,g,f\rangle_R$ then
we can realise the diagram
\[
\xymatrix{
& \Sigma^3W & \\
\Sigma^2X\ar@{-}[ur]^{\Sigma^2 f} && \\
&& \Sigma U\ar@{-}[uul]_{\Sigma g} \\
\Sigma Y\ar@{-}[uu]^{\Sigma g} &&  \\
& Z\ar@{-}[ul]^h\ar@{-}[uur]_{\ell} &
}
\]
\end{lem}

Here $\langle-,-,-\rangle_R$ denotes the Toda
bracket calculated in the homotopy category
of $R$-modules.

\begin{rem}\label{rem:TodaBrackets-Whitehead}
There is a different kind of Toda bracket that
we might consider for an $R$ ring spectrum~$A$
and a left $A$-module~$X$ spectrum (both being
$R$-modules of course). For example we can consider
$\langle u,v,w\rangle_{R,X}$ for $u\in\pi_r(A)$,
$v\in\pi_s(A)$ and $w\in\pi_t(X)$ where $uv=0$
in $\pi_*(A)$ and $vw=0$ in $\pi_*(X)$ where
these products are defined using the evident
ring structure on $\pi_*(A)$ and the
$\pi_*(A)$-module structure of $\pi_*(X)$.
The resulting bracket is a subset of
$\pi_{r+s+t+1}(X)$ with indeterminacy
$u\pi_{s+t+1}(X)+\pi_{r+s+1}(A)w$. A version
of this theory was described by
G.~Whitehead~\cite{GWW:RecentAdvHtpyThy} well
before modern categories of spectra were
developed but the essential ideas can be found
in his work. Some applications of these brackets
can be found in~\cite{AJB&JPM} using examples
given in~\cite{GWW:RecentAdvHtpyThy}; in these
we have $R=A=S$ and $X=\kO$ or $\kU$.
\end{rem}

\section{Duality}\label{sec:duality}

Various sorts of duality occur in the categories
$\mathscr{M}_R$ and $\mathscr{D}_R$, generalising
classical cases.

\subsection*{Spanier-Whitehead duality}
For an account of duality from a categorical
viewpoint, we recommend the article of Dold
\& Puppe~\cite{Dold&Puppe:Duality}.

Following~\cites{Dold&Puppe:Duality,EKMM},
the symmetric monoidal category $\mathscr{D}_R$
has strongly dualisable objects and so there
is a version of Spanier-Whitehead duality;
we will denote the Spanier-Whitehead of
an $R$-module~$X$ by $D_RX=F_R(X,S_R^0)$.
As usual, when~$X$ is a finite CW~module
we can replace it by a weakly equivalent
CW~module; it is well known that an $R$-module
is strongly dualisable if it is equivalent
to a retract of a finite CW~module. Of course,
if $Z$ is a strongly dualisable $S$-module
(i.e., a spectrum) then $R\wedge Z$ is a
strongly dualisable $R$-module; more
generally, if $Z$ is a strongly dualisable
$R'$-module where~$R$ is a commutative
$R'$-algebra, then $R\wedge_{R'}Z$ is
strongly dualisable.

When~$L$ is strongly dualisable, the Adams
spectral sequence of~\eqref{eq:ASS-cohomology}
can be expressed as
\begin{equation}\label{eq:ASS-cohomology-SWD}
\mathrm{E}_2^{s,t} =
\Ext_{H_R^*H}^{s,t}(H_R^*(D_RL\wedge_RM),\F_p)
\Lra
\mathscr{D}_R(\Sigma^{t-s},D_RL\wedge_RM)\sphat\;
\iso
\mathscr{D}_R(S_R^{t-s}L,M)\sphat
\end{equation}
since there are natural isomorphisms of
functors
\[
\Hom_{H_R^*H}(H_R^*(D_RL)\otimes(-),\F_p)
\iso
\Hom_{H_R^*H}(\mathrm{D}(H_R^*L)\otimes(-),\F_p)
\iso
\Hom_{H_R^*H}(-,H_R^*L)
\]
on left $H_R^*H$-modules extending to right
derived functors. This is useful for
computational purposes as it allows us to
work consistently with projective resolutions
and calculations with right derived functors
of form $\Hom_{H_R^*H}(-,\F_p)$.

\subsection*{Poincar\'e duality and Spanier-Whitehead
duality}%\label{sec:PD&SWD}

We begin with some algebra. Let $\k$ be a field
and~$K^*$ a connected graded cocommutative Hopf
algebra of finite type. We will indicate the
coproduct~$\psi$ using the notation
\[
\psi\theta
=
\theta\otimes1 + 1\otimes\theta
           + \sum_i \theta'_i\otimes\theta_i''
=
\theta\otimes1 + 1\otimes\theta
           + \sum_i \theta''_i\otimes\theta_i'
\]
where the degrees of $\theta',\theta_i''$ are
positive and smaller than the degree of
$\theta$. The action of the antipode~$\chi$
will often be indicated by writing
$\chi\theta=\bar{\theta}$, so
\[
\bar{\theta} = -\theta - \sum_i \bar{\theta'_i}\theta_i''
= -\theta - \sum_i \theta_i'\bar{\theta''_i}.
\]

Now let $P_*$ be a local Poincar\'e duality
algebra of degree~$d$ and let its graded dual
be $P^*$ where $P^n=\Hom_\k(P_n,\k)$. This
means that $P_n=0$ except when $0\leq n\leq d$,
$P_0=\k$ and there is a $\k$-linear isomorphism
$P_0\xrightarrow{\iso}P^d$ with
$1\leftrightarrow\lambda$ which induces
isomorphisms
\[
P_n \xrightarrow{\iso} P^{d-n};
\quad
x \mapsto x\lambda
\]
where
\[
(x\lambda)(y) = \lambda(yx)
\]
for all $y\in P_{d-n}$. The pairing
\[
P_* \otimes P^* \to P^*;
\quad
x \otimes \gamma \mapsto x\gamma
\]
makes $P^*$ a left $P_*$-module and the above
duality isomorphism can be interpreted as
defining an isomorphism of left $P_*$-modules
\[
P_* \xrightarrow{\iso} P^*[-d];
\quad
x \mapsto x\lambda.
\]

Now suppose that $P_*$ is a left $K^*$-module
algebra. This means that there are pairings
\[
K^r\otimes P_s \xrightarrow{\iso} P_{s-r};
\quad
\theta\otimes x \mapsto \theta x
\]
so that the Cartan formula holds for all
$x,y\in P_*$:
\[
\theta(xy) =
(\theta x)y + x(\theta y)
     + \sum_i (\theta'_ix)(\theta_i''y).
\]
There is also an action of $K^*$ on $P^*$
given by pairings
\[
K^r\otimes P^s \xrightarrow{\iso} P^{r+s};
\quad
\theta\otimes\gamma \mapsto \theta\gamma
\]
where
\[
(\theta\gamma)(z) = \gamma(\bar{\theta}z)
\]
and this makes $P^*$ a left $H^*$-module.

\begin{lem}\label{lem:PD-HA}
The duality isomorphism of left\/ $P_*$-modules
$P_*\xrightarrow{\iso}P^*[-d]$ is also
an isomorphism of left $H^*$-modules.
\end{lem}
\begin{proof}
We have to show that for all homogeneous
elements $x,y\in P_*$ and $\theta\in K^*$,
\[
\lambda(y(\theta x))
= ((\theta x)\lambda)(y)
= \lambda((\bar{\theta}y)x).
\]
We will prove this by induction on the degree
of~$\theta$. It is clearly true when~$\theta$
has degree $0$. So assume it holds whenever
$\theta$ has degree less than $n>0$.

Suppose that $\theta$ has degree~$n$. Consider
$\theta(yx)$; in order for this element to
have degree~$d$, $yx$ has to be of degree
$n+d>0$, hence~$yx=0$. So
\begin{align*}
0 = ((\theta(yx))\lambda)(1)  &= \lambda(\theta(yx))   \\
&= \lambda((\theta y)x)  +  \lambda(y\theta x)
+ \sum_i  \lambda((\theta'_iy)(\theta_i''x))   \\
&= \lambda((\theta y)x)  +  \lambda(y\theta x)
+ \sum_i  \lambda((\bar{\theta''_i}\theta_i'y)x)  \\
&= \lambda(y\theta x )  -  \lambda((\bar{\theta} y)x)
+  \lambda \biggl((\theta y)x
+ \sum_i  \lambda((\bar{\theta'_i}\theta_i''y)x
+ (\bar{\theta} y)x\biggr)  \\
&= \lambda(y\theta x )  -  \lambda((\bar{\theta} y)x).
\end{align*}
Thus $\lambda(y\theta x) = \lambda((\bar{\theta}y)x)$,
so the result holds for all~$n$.
\end{proof}

Now let $R$ be a commutative $S$-algebra satisfying
the conditions assumed earlier. In particular,
suppose that~$R_0$ is a cyclic $\Z_{(p)}$-module
for some prime~$p$.

Suppose that $E$ is an $R$ ring spectrum
for which $P_*=H^R_*E$ is a local Poincar\'e
duality algebra over~$\F_p$. Taking $K^*=H_R^*H$,
the Spanier-Whitehead dual of~$E$ satisfies
\[
H^R_*D_RE \iso H_R^*E\iso H^R_*E[d]
\]
as $H_R^*H$-modules.

The next result is very useful for identifying
Spanier-Whitehead stably self dual objects.
\begin{prop}\label{prop:PD-HA}
There is a morphism of $R$-modules $E\to\Sigma^dR$
inducing a non-trivial homomorphism
\[
H^R_*E \to H^R_*R[-d] = \F_p[-d].
\]
The multiplication map $E\wedge_R E\to E$
composed with this map define a duality
pairing $E\wedge_R E\to\Sigma^dR$. Hence
$E$ is a Spanier-Whitehead stably self
dual $R$-module with $D_RE\sim\Sigma^{-d}E$.
\end{prop}
\begin{proof}
Choose a minimal CW $R$-module realisation
of~$E$; this will have a single cell in
each of the degrees~$0$ and~$d$. Inclusion
of the bottom cell induces the unit
$\F_p\to H^R_0E$, while collapse onto the
top cell induces a non-trivial linear
mapping $H^R_dE\to\F_p$ and this gives a
basis element of $H_R^dE$; by composing
with a self map of $S_R^d\sim\Sigma^dR$
we can assume this corresponds to
$1\in H^R_0E$ under the Poincar\'e duality
isomorphism $H^R_0E\xrightarrow{\iso}H_R^dE$.
The product $E\wedge_R E\to E$ composed
with the projection $E\to\Sigma^dR$ gives
rise to a morphism $f\:E\to\Sigma^dD_RE$
and induces a non-degenerate pairing
$H^R_*E\otimes H^R_*E\to\F_p[d]$. It follows
that
$f_*\:H^R_*E\xrightarrow{\iso}H^R_*D_RE[d]$
and so $f\:E\xrightarrow{\sim}\Sigma^{d}D_RE$.
Notice that all the algebraic maps here
are compatible with the actions of~$H_R^*H$.
\end{proof}

For any compact Lie group~$G$, $E=R\wedge(G_+)$
provides an example of such a stably self dual
$R$ ring spectrum. A generalisation to finite
$H$-spaces was proved by Browder \&
Spanier~\cite{WB&ES:Hspacesduality}, and some
exotic examples can be found in the papers of
Bauer, Pedersen and Rognes
\cites{TB:pCmpctGps,TB&EKP:LoopSpcesMfds,JR:MAMS192};
in all these classical cases, $R=S$ is the base
spectrum but the ideas work more generally.
More recently, Beaudry et al \cite{5A:QuoRingsHH}
have also considered Spanier-Whitehead
self duality over commutative $S$-algebras.

\subsection*{Cyclic modules}
It is common to encounter cyclic modules
of the form $\StA(n)\otimes_{\mathcal{B}}\F_2$
arising as cohomology of $R$-modules for
examples such as $R=\kO$ and $R=\tmf$.
Proposition~\ref{prop:PD-HA} sometimes
allows us to identify the underlying
$R$-module as stably Spanier-Whitehead
self dual, but this seems not to be
a purely algebraic result. For example,
consider the $\StA(1)$-module
\[
\StA(1)\otimes_{\F_2(\Sq^2)}\F_2
= \StA(1)/\StA(1)\{\Sq^2\}.
\]
Viewing this as the quotient of $\StA(1)$
obtained by killing the white circles in
the diagram below, we see that this is
the question mark module which is clearly
not self dual.
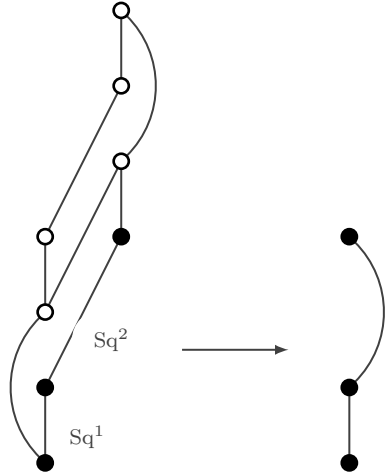
\[
\begin{tikzpicture}
\Vertex[x=-1,y=6,color=white,size=.05]{A6}
\Vertex[x=-1,y=5,color=white,size=.1]{A5}
\Vertex[x=-1,y=4,size=.1,color=white]{A4}
\Vertex[x=-2,y=3,size=.1,color=white]{A31}
\Vertex[x=-1,y=3,size=.1,color=black]{A32}
\Vertex[x=-2,y=2,size=.1,color=white]{A2}
\Vertex[x=-2,y=1,size=.1,color=black]{A1}
\Vertex[x=-2,y=0,size=.1,color=black]{A0}
\Edge[lw=0.75pt](A5)(A6)
\Edge[lw=0.75pt,bend=-45](A4)(A6)
\Edge[lw=0.75pt](A32)(A4)
\Edge[lw=0.75pt](A31)(A5)
\Edge[lw=0.75pt](A2)(A31)
\Edge[lw=0.75pt](A2)(A4)
\Edge[lw=0.75pt,label={$\Sq^2$},position=right,distance=.1mm](A1)(A32)
\Edge[lw=0.75pt,label={$\Sq^1$},position=right,distance=.1mm](A0)(A1)
\Edge[lw=0.75pt,bend=45](A0)(A2)
\Vertex[x=2,y=3,size=.1,color=black]{B32}
\Vertex[x=2,y=1,size=.1,color=black]{B1}
\Vertex[x=2,x=2,y=0,size=.1,color=black]{B0}
\Edge[lw=0.75pt,bend=-45](B1)(B32)
\Edge[lw=0.75pt](B0)(B1)
\Vertex[x=-0.5,y=1.5,Pseudo]{P}
\Vertex[x=1.5,y=1.5,Pseudo]{Q}
\Edge[lw=0.75pt,Direct](P)(Q)
\end{tikzpicture}
\]

The situation when $\mathcal{B}$ is a subHopf
algebra is more interesting and we will discuss
it from an algebraic perspective below. We will
encounter many examples of this kind arising
as cohomology of~$\kO$ or~$\tmf$-modules. The
situation described in our next result is
encountered in many examples.
\begin{prop}\label{prop:SubAlg}
Let $X$ be an $R$-module and let
$\mathcal{B}\subseteq H_R^*H$ be a subalgebra.
Suppose that as a $\mathcal{B}$-module,
$H_R^*X\iso\mathcal{B}$. Then we have the
following. \\
\emph{(a)}
$H_R^*X$ is a cyclic $H_R^*H$-module. \\
\emph{(b)}
The element $x_0\in H_R^0X\iso\mathcal{B}^0=H_R^0H$
corresponding to\/~$1$ induces a cofibre sequence
\[
X\xrightarrow{x_0} H \to H/X
\]
and in turn a split short exact sequence of
$\mathcal{B}$-modules
\[
0\leftarrow H_R^*X\leftarrow H_R^*H \leftarrow H_R^*(H/X)\leftarrow0.
\]
\emph{(c)}
If $H_R^*H$ is self-injective and flat as
a $\mathcal{B}$-module, then $\mathcal{B}$
is self-injective.
\end{prop}
\begin{proof}
(a) This is obvious. \\
(b) We have a commutative diagram with exact
row
\[
\xymatrix{
H_R^*X & \ar[l]_{x_0^*} H_R^*H & \ar[l]H_R^*(H/X) \\
& \mathcal{B}\ar@{^{(}->}[u]\ar@{<->}[ul]^{\iso} &
}
\]
showing that the row is split exact as a diagram
of $\mathcal{B}$-modules.  \\
(c) This follows from the following standard result:
For an injective $H_R^*H$-module~$J$, the adjunction
isomorphism
\[
\Hom_{\mathcal{B}}(-,J)
\iso
\Hom_{H_R^*H}(H_R^*H\otimes_{\mathcal{B}}(-),J)
\]
is exact.
\end{proof}

Of course (c) applies whenever $\mathcal{B}$
is a subHopf algebra of~$H_R^*H$. we will
see some examples of this when $R=\kO$ and
$R=\tmf$.

\subsection*{Some algebraic results}
Next we give some purely algebraic results
that we will use. The reader is referred 
to the following for general results on 
Hopf algebras: 
Larson \& Sweedler~\cite{LarsonSweedlerThm},
Humphreys~\cite{JEH:SymmFinDimlHA}*{theorem~1}
and Montgomery~\cite{SM:HopfAlgActions};
the book of Lorenz~\cite{ML:TourRepThy}
also contains much useful material. 

We will assume the reader is aware of basic 
results such as the freeness of a Hopf algebra 
over a subHopf algebra (Milnor-Moore or 
Nichols-Zoeller for the ungraded case), and 
Poincar\'e duality/Frobenius property and 
self injectivity for finite dimensional Hopf 
algebras (Browder-Spanier or Larson-Sweedler 
for the ungraded case).

First we give a statement and proof of a
result for non-graded Hopf algebras which
ought to be standard; the closest reference
appears to be in Fischman et
al~\cite{DF&SM&HJS:FrobeniusExtns}*{theorem~4.8
and corollary~4.9}. This version incorporates
suggestions of Ken Brown which led to
a substantial improvement of our original
attempt. Of course, if~$K$ is a \emph{normal}
subHopf algebra (i.e., if $HK^+=K^+H$) then
$H/\!/K$ is a Hopf algebra and this result
is immediate. 

Topologists are often interested in the case 
where~$H$ is commutative or cocommutative when 
its antipode is self inverse making it is 
involutive, i.e., $\chi\circ\chi=\Id$; the 
same is of course also true for any subHopf 
algebra. Finally we remark that when~$H$ is 
a local algebra its socle is $1$-dimensional 
so it is modular.

\begin{prop}\label{prop:Cyclic-Selfdual}
Let $H$ be a finite dimensional unimodular 
Hopf algebra over a field\/~$\k$ and let~$K$ 
be a subHopf algebra which is also unimodular. 
Then
\[
H/\!/K = H\otimes_{K}\k \iso H/HK^+
\]
is a self dual left $H$-module.
\end{prop}
\begin{proof}
Let $\lambda$ be a Frobenius form for~$H$
which we take to be a left and right integral
in the dual Hopf algebra $\mathrm{D}H=\Hom_\k(H,\k)$.
The \emph{Nakayama algebra automorphism}
$\nu\:H\to H$ is characterised by the
identity
\begin{equation}\label{eq:SymmetryCond}
\lambda(xy) = \lambda(y\nu(x))
\end{equation}
for all $x,y\in H$. If the associated bilinear 
pairing $H\otimes_\k H\to\k$ is symmetric 
(which is true when~$H$ is unimodular and 
involutive) then~$\nu=\Id_H$.

As~$K$ is also unimodular, its ($1$-dimensional)
vector spaces of left and right integrals coincide,
i.e., $\sint^{\mathrm{l}}_K = \sint^{\mathrm{r}}_K$,
and we will just write $\sint_K$ for this subspace.
By definition, the left and right annihilators
of $\sint_K$ and any non-zero element $s\in\sint_K$
satisfy
\[
\ann^{\mathrm{l}}_K\sint_K = \ann^{\mathrm{l}}_K(s)
= \ann^{\mathrm{r}}_K\sint_K = \ann^{\mathrm{r}}_K(s)
= K^+,
\]
the kernel of the counit $K\to\k$ which is a maximal
ideal. By the Nichols-Zoeller theorem~\cite{WDN&MBZ:HAfreeness},
$H$ is free as a left or right $K$-module, so the
left annihilators in~$H$ satisfy
\begin{equation}\label{eq:leftannsocK}
\ann^{\mathrm{l}}_H\sint_K = \ann^{\mathrm{l}}_H(s)
= HK^+,
\end{equation}
the left ideal of~$H$ generated by~$K^+$. Similarly
the right annihilators satisfy
\begin{equation}\label{eq:rightannsocK}
\ann^{\mathrm{r}}_H\sint_K = \ann^{\mathrm{r}}_H(s)
= K^+H = \chi(HK^+).
\end{equation}

Now choose a non-zero element $s_0\in\int_K$.
The $\k$-linear mapping
\[
\lambda'\: H\to k;
\quad
x \mapsto \lambda(xs_0)
\]
satisfies
\[
\lambda'(xz) = 0
\]
whenever $x\in H$ and $z\in K^+$, so it factors
through a linear mapping $\lambda''\:H/\!/K\to\k$.
It follows that the left $H$-module homomorphism
\[
H \to H;
\quad
x\mapsto xs_0
\]
has kernel $HK^+$ and so induces an isomorphism
$H/\!/K\xrightarrow{\iso}Hs_0$.

Define a left $H$-module structure on
$\Hom_\k(Hs_0,\k)$ by setting
\[
(x\cdot\alpha)(z) = \alpha(\chi(x)z)
\]
for $\alpha\in\Hom_\k(Hs_0,\k)$ and $x\in H$.
Now define a left $H$-module homomorphism
\[
H\to \Hom_\k(Hs_0,\k);
\quad
x\mapsto x\cdot\lambda'.
\]
Using the Nakayama automorphism characterised
in~\eqref{eq:SymmetryCond} we obtain
\[
(x\cdot\lambda')(z) = \lambda(\chi(x)z)
                    = \lambda(z\nu(\chi(x))).
\]
So when $x\in HK^+=\chi(K^+H)$ we have $(x\cdot\lambda')(z) = 0$,
and therefore we can factor our homomorphism
through a left $H$-module homomorphism
$H/\!/K\to\Hom_\k(Hs_0,\k)$. Also, when
$x\notin HK^+$, since the kernel of the
functional $\lambda((-)\nu(\chi(x)))\:H\to\k$
cannot contain a non-trivial left ideal,
the functional $x\cdot\lambda\:Hs_0\to\k$
must be non-trivial. This shows that there
is an injection
\[
H/\!/K\to\Hom_\k(Hs_0,\k)\iso\Hom_\k(H/\!/K,\k)
\]
which must be an isomorphism since the
dimensions of the domain and codomain
agree.
\end{proof}

%%%%%%%%%%%%%

The proof of the following involves a modification
of that for Proposition~\ref{prop:Cyclic-Selfdual}
with suitable allowances for gradings.
\begin{prop}\label{prop:Cyclic-Selfdual-Graded}
Let $H$ be a finite dimensional commutative
or cocommutative unimodular connected graded
Hopf algebra over a field\/~$\k$ and let\/~$K$
be a subHopf algebra. Then
\[
H/\!/K=H\otimes_{K}\k \iso H/HK^+
\]
and there is an isomorphism of left $H$-modules
\[
\Hom_\k(H/\!/K,\k) \xrightarrow{\iso} H/\!/K[-d],
\]
hence $H/\!/K$ is a stably self dual left
$H$-module.
\end{prop}

Recall that a connected cohomologically graded
$\k$-algebra~$A$ has $A^i=0$ when $i<0$ and
$A^0\iso\k$. A finite dimensional connected
graded Hopf algebra over a field~$\k$ is a
Poincar\'e duality algebra of some dimension~$d$,
and a basis element of the $1$-dimensional
$\k$-vector space $\Hom_\k(A^d,\k)$ provides
a `Frobenius form' with similar properties
to the ungraded case. The antipode~$\chi$
acts as an isomorphism on the $1$-dimensional
~$A^d$, hence the Hopf algebra~$A$ and its
graded dual $\mathrm{D}A$ are unimodular.
This means we can apply
Proposition~\ref{prop:Cyclic-Selfdual-Graded}
to pairs of finite dimensional subHopf algebras
of the Steenrod algebra~$\StA$.

In the next result we give some useful
consequences of Propositions~\ref{prop:Cyclic-Selfdual}
and~\ref{prop:Cyclic-Selfdual-Graded}
(in the latter case we need to interpret
modules and morphisms as being suitably
graded). First we need to set up some
notation.
\begin{itemize}
\item
We will set $\otimes=\otimes_\k$ and
$\Hom=\Hom_\k$.
\item
For a right $K$-module $N$ (i.e., a left
$K^\op$-module), make $\Hom_K(H,N)$ a left
$H$-module with action given by
\[
(h\cdot\phi)(x) = \phi(\chi(h)x)
\]
for $h,x\in H$ and $\phi\in\Hom_{K^\op}(H,N)$
where $H$ is regarded as a left $K^\op$-module
through right multiplication of~$K$. Also make
$H/\!/K\otimes N$ and $\Hom_\k(H/\!/K,N)$ left
$K^\op$-modules by letting~$K^\op$ act through
its action on~$N$; this makes them both
$H\otimes K^\op$-modules.
\item
When $L$ and $M$ are left $H$-modules make
$\Hom(L,M)$ a left $H$-module with action
given by
\[
(h\cdot\theta)(x) = \sum_i h''_i\theta(\chi(h'_ix))
\]
for $h\in H$, $x\in L$ and $\theta\in\Hom(L,M)$
and the coproduct on $h$ being
\[
\psi(h) = \sum_i h'_i\otimes h''_i.
\]
It is well known that $\theta$ satisfies
$h\cdot\theta = \epsilon(h)\theta$ for
all $h\in H$ if and only if
$\theta\in\Hom_H(L,M)\subseteq\Hom(L,M)$.
\item
Viewing $H$ and $M$ as left $K$-modules,
make $\Hom_K(H,M)$ a left $H$-module by
setting
\[
(h\cdot\rho)(x) = \rho(xh).
\]
\end{itemize}

\begin{prop}\label{prop:Cyclic-Selfdual-consequences}
Let $H$ and $K$ be as in \emph{Proposition~\ref{prop:Cyclic-Selfdual}}
or~\emph{\ref{prop:Cyclic-Selfdual-Graded}}.
Let $L$ and $M$ be left $H$-modules and let
$N$ be a right $K$-module. \\
\emph{(a)} There are natural isomorphisms
of left $H\otimes K^\op$-modules
\[
H/\!/K\otimes N
\xrightarrow{\iso} \Hom(H/\!/K,\k)\otimes N
\xrightarrow{\iso} \Hom(H/\!/K,N).
\]
\emph{(b)} There are natural isomorphisms
of left $H$-modules
\[
H/\!/K\odot M
\xrightarrow{\iso} \Hom(H/\!/K,M)
\xrightarrow{\iso} \Hom_K(H,M).
\]
\emph{(c)}
There is a natural isomorphism of\/
$\k$-vector spaces
\[
\Hom_H(L,H/\!/K\odot M)
\xrightarrow{\iso} \Hom_K(L,M).
\]
\end{prop}
\begin{proof}
(a) The first isomorphism of vector spaces
uses Proposition~\ref{prop:Cyclic-Selfdual},
the second is standard; these clearly respect
the $H\otimes K^\op$-module structures.  \\
(b) As for (a), there are isomorphisms
of $\k$-vector spaces
\[
H/\!/K\otimes M
\xrightarrow{\iso} \Hom(H/\!/K,\k)\otimes M
\xrightarrow{\iso} \Hom(H/\!/K,M)
\]
which are both $H$-linear.

Define a map
\[
\Hom(H/\!/K,M) \to \Hom_K(H,M);
\quad
\phi \mapsto \tilde\phi
\]
where
\[
\tilde\phi(x) = \sum_i x'_i\phi(\chi(x''_i)+HK^+).
\]
This has inverse
\[
\Hom_K(H,M) \to \Hom(H/\!/K,M);
\quad
\theta \mapsto \overset{\approx}{\theta}
\]
where
\[
\overset{\approx}{\theta}(x+HK^+)
      = \sum_i\chi(x''_i)\theta(x'_i).
\]
These are both $H$-linear. \\
(c) Using~(b) and a standard adjunction
result we obtain
\begin{align*}
\Hom_H(L,H/\!/K\odot M)
&\xrightarrow{\iso} \Hom_H(L,\Hom_K(H,M)) \\
&\xrightarrow{\iso} \Hom_K(H\otimes_HL,M) \\
&\xrightarrow{\iso} \Hom_K(L,M).
\qedhere
\end{align*}
\end{proof}

These identifications can be used to deduce
homological results. Here is one that is
useful in our work.
\begin{prop}\label{prop:ExtH//K}
Let $H$, $K$, $L$ and $M$ be as in
\emph{Proposition~\ref{prop:Cyclic-Selfdual-consequences}}.
Then
\[
\Ext^*_H(L,H/\!/K\odot M) \iso \Ext^*_K(L,M).
\]
\end{prop}
\begin{proof}
Take a projective resolution $P_*\to L$ of the
$H$-module~$L$. Then for each $s\geq0$,
\[
P_s \iso H\otimes_HP_s
\]
is both a projective $H$-module and a projective
$K$-module since~$H$ is a free left $K$-module.
Therefore $P_*\to L$ is also a projective
resolution for~$L$ as a $K$-module. Also
\begin{align*}
\Hom_H(P_*,\Hom_K(H,M))
&\iso
\Hom_K(H\otimes_HP_*,M)   \\
&\iso
\Hom_K(P_*,M),
\end{align*}
so on taking cohomology we obtain
\[
\Ext_H^*(L,\Hom_K(H,M)) \iso \Ext_K^*(L,M).
\]
The result now follows using
Proposition~\ref{prop:Cyclic-Selfdual-consequences}(b).
\end{proof}

In particular, the case where $L=H/\!/K$
and $ M=\k$ is useful for some of the
topological examples we will see later.

In the graded case, suppose that the top
degrees of~$H$ and~$K$ are~$d$ and~$e$
respectively. Then the top degree of
$H/\!/K$ is~$d-e$ and
\[
\mathrm{D}(H/\!/K) \iso H/\!/K[e-d].
\]

\subsection*{Poincar\'e duality for manifolds}

For a commutative ring spectrum~$E$, classical
Poincar\'e duality in $E$-theory is defined
using the slant product for a space~$X_+$
with disjoint base point
\[
E^*(X_+)\otimes_{E_*} E_*(X_+)\to E_*(X_+)
\]
and the augmentation induced by collapse on
the base point $E_*(X_+)\to E_*(S^0)=E_*$.
Underlying this are compositions of the
form
\[
\xymatrix{
%S^n\ar[r] &
\Sigma^nE\wedge X_+ \ar[r]^(.47){\Id\wedge\Delta} & E\wedge X_+\wedge X_+ \ar[r]
& \Sigma^nE\wedge\Sigma^rE\wedge X_+\ar[r]^(.53){\mathrm{mult}\wedge\Id}
& \Sigma^{n+r}E\wedge X_+ \ar[r]
& \Sigma^{n+r}E
}
\]
and when $E$ is a commutative $S$-algebra,
the composition is a morphism of left
$E$-modules. Of course this can also be
formulated in terms of Spanier-Whitehead
duality using Atiyah duality. This leads
to the following result which gives a
rich source of stably self dual $R$-modules;
similar ideas were explored by Douglas,
Henriques and Hill~\cite{CLD-AGH-MAH:HomObstrStrOtns}.
\begin{prop}\label{prop:PD-Rmods}
Let $R$ be a commutative $S$-algebra.
Suppose that~$M$ is a compact closed
smooth $n$-manifold whose normal bundle
admits an orientation in $R$-cohomology.
Then $R\wedge M_+$ is a stably
Spanier-Whitehead self $R$-module and
\[
D_R(R\wedge M_+) \sim R\wedge\Sigma^{-n}(M_+).
\]
\end{prop}
\begin{proof}
On choosing an $R$-orientation, Poincar\'e
duality for $R_*(M_+)$ gives an explicit
isomorphism
\[
R_*(M_+) \xrightarrow{\iso} R^{n-*}(M_+).
\]
By the above remarks, this is induced by
a weak equivalence of $R$-modules
\[
D_R(R\wedge M_+) \sim R\wedge\Sigma^{-n}(M_+).
\qedhere
\]
\end{proof}

If $G$ is a compact Lie group, whenever
$A\leq G$ is a closed abelian subgroup (such
as a maximal torus), $G/A$ is stably frameable
and so stably self dual, hence $R\wedge(G/A_+)$
is a stably self dual $R$-module.

Since $\Spin$ manifolds are $\kO$-orientable
they satisfy the conditions; similarly when
$R=\tmf$, $\String$ manifolds are $\tmf$-orientable.
As examples of these, recall that $\RP^n$
is a $\Spin$ manifold if and only if
$n\equiv3\bmod{4}$, and it is a $\String$
manifold if $n\equiv7\bmod{8}$. So
$\kO\wedge(\RP^{4k-1}_+)$ is a stably self
dual $\kO$-module with
\[
D_\kO(\kO\wedge(\RP^{4k-1}_+))
\sim \kO\wedge\Sigma^{-4k+1}(\RP^{4k-1}_+)
\]
and $\tmf\wedge(\RP^{8\ell-1}_+)$ is a stably
self dual $\tmf$-module with
\[
D_\tmf(\tmf\wedge(\RP^{8\ell-1}_+))
\sim
\tmf\wedge\Sigma^{-8\ell+1}(\RP^{8\ell-1}_+).
\]
We will discuss homogeneous spaces in more
greater detail in Section~\ref{sec:HomogenSpces}.

%\subsection*{A duality isomorphism}
%\todo[inline]{Redundant}
%
%Suppose that $H$ is a finite dimensional Hopf
%algebra over a field~$\k$ and that $K\subseteq H$
%is sub-Hopf algebra. For simplicity we also
%assume that~$K$ is a local ring, hence the
%kernel of its counit $K^+=\ker K\to\k$ is
%its unique maximal ideal and its socle is
%$1$-dimensional.
%
%Then we know that $H$ is Frobenius so we will
%choose a Frobenius form $\lambda\in\mathrm{D}H$.
%Associated to this is a Frobenius isomorphism
%\[
%H \xrightarrow{\iso}\mathrm{D}H;
%\quad
%x\leftrightarrow x\lambda = \lambda((-)x).
%\]
%We may regard $H$ and $\mathrm{D}H$ as left
%$H$-bimodules with the latter having its
%multiplication by $H^{\mathrm{e}}=H\otimes H^{\mathrm{o}}$
%given by
%\[
%(a\otimes b^{\mathrm{o}})\theta = a\theta b
%= \theta(\chi b(-)\chi a).
%\]
%Then the Frobenius isomorphism is an isomorphism
%of $H$-bimodules.
%
%Now choose a non-trivial element $z\in\soc K$.
%Then the map
%\[
%\phi\:H \to \mathrm{D}H;
%\quad
%x \mapsto \lambda((-)xz)
%\]
%is a left $H$-module homomorphism. As~$H$
%is a free right $K$-module, it follows that
%\begin{align*}
%\ker \phi
%&= \{ x\in H : xz=0 \} \\
%&= H\tod K = HK^+,
%\end{align*}
%since the Jacobson radical $\rad K$ is also
%the left annihilator of~$\soc K$. So $\phi$
%factors through $H/\!/K=H/HK^+$.
%\[
%\xymatrix{
%H\ar[rr]^\phi\ar[dr]&&\mathrm{D}H \\
%&H/\!/K\ar@{ >->}[ur]_{\bar\phi}&
%}
%\]
%Notice that for every $x\in H$,
%\todo[inline]{finish!}

\section{Recollections on the Steenrod algebra
and finite subHopf algebras}\label{app:Steenrodmodules}

The book of Margolis~\cite{HRM:Book} provides
a thorough treatment of the Steenrod algebra~$\StA$
and its finite subHopf algebras such as the~$\StA(n)$
family. For more on bases of~$\StA$ and its
subalgebras see the survey article of
Wood~\cite{RMWW:Problems}.

In this appendix we highlight some important
ideas that are useful for the present work.

\subsection*{The Wall relations for $\StA(n)$}
%\label{app:Wallrelns}
When working with $\StA$ or $\StA(n)$ we
often describe elements by dualising the
monomial basis in the dual. Recall that
\begin{align*}
\StA_* &= \F_2[\xi_i:i\geq1], \\
\StA(n)_* &= \F_2[\xi_i:i\geq1]
/(\xi_1^{2^{n+1}},\xi_2^{2^n},\xi_3^{2^{n-1}},
\ldots,\xi_{n+1}^{2},\xi_{n+2},\ldots).
\end{align*}
Then the basis of (residue classes of) monomials
$\xi_1^{r_1}\xi_2^{r_2}\cdots\xi_\ell^{r_\ell}$
defines a basis consisting of the elements
$\Sq(r_1,\ldots,r_\ell)$ in~$\StA$ or~$\StA(n)$.
We remark that in the computer algebra system
\Sage, only this basis is available for working
in~$\StA(n)$ making it the natural basis for
expressing results found with its aid. Of course
these basis elements can also be expressed
in terms of monomials in~$\Sq^r$ or indeed
the indecomposables~$\Sq^{2^s}$.

The usual Adem relations in $\StA$ do not
always restrict to a finite subalgebra~$\StA(n)$.
For example, the following consequences
of Adem relations
\[
\Sq^2\Sq^3 = \Sq^4\Sq^1 + \Sq^5
           = \Sq^4\Sq^1 + \Sq^1\Sq^4
\]
are not meaningful in $\StA(1)$ since
$\Sq^4\notin\StA(1)$.

A minimal set of relations between indecomposable
generators~$\Sq^{2^s}$ of~$\StA$ was determined
by Wall~\cite{CTCW:A-gens} and these do restrict
to defining relations for each~$\StA(n)$.
Incidentally, these relations can be interpreted
in either the sense of algebra relations or
module relations for the augmentation ideal
considered as a left or right module.

Consider the following elements of $\StA^*$:
for $0\leq s\leq r-2$ and $1\leq t$,
\begin{align*}
\tag*{(A)}\Theta(r,s) &=
\Sq^{2^r}\Sq^{2^s} + \Sq^{2^s}\Sq^{2^r}, \\
\tag*{(B)}\Phi(t) &=
\Sq^{2^t}\Sq^{2^t}
   + \Sq^{2^{t-1}}\Sq^{2^t}\Sq^{2^{t-1}}
   + \Sq^{2^{t-1}}\Sq^{2^{t-1}}\Sq^{2^t}.
\end{align*}
Then $\Theta(r,s)\in\StA(r-1)$ and
$\Phi(r)\in\StA(r-1)^*$, so these can
be expressed as polynomial expressions
in the $\Sq^{2^k}$ for $0\leq k\leq r-1$.
The elements of form
\[
\Sq^{2^r}\Sq^{2^s} + \Sq^{2^s}\Sq^{2^r}
+\; \Theta(r,s),
\quad
\Sq^{2^t}\Sq^{2^t}
+ \Sq^{2^{t-1}}\Sq^{2^t}\Sq^{2^{t-1}}
+ \Sq^{2^{t-1}}\Sq^{2^{t-1}}\Sq^{2^t}
+\; \Phi(t)
\]
give a minimal set of relations for~$\StA$.
In particular, such elements with $r,t\leq n$
form a minimal set of relations for~$\StA(n)$.
In the first few cases the Wall relations are
\begin{align*}
\StA(0):\quad \Sq^1\Sq^1& =0, \\
\StA(1):\quad \Sq^1\Sq^1
&= \Sq^2\Sq^2+\Sq^1\Sq^2\Sq^1 = 0 \\
\StA(2):\quad
\Sq^1\Sq^1
&= \Sq^2\Sq^2+\Sq^1\Sq^2\Sq^1 \\
&= \Sq^4\Sq^4+\Sq^2\Sq^4\Sq^2+\Sq^2\Sq^2\Sq^4 \\
&= \Sq^1\Sq^4+\Sq^4\Sq^1+\Sq^2\Sq^1\Sq^2 = 0.
\end{align*}

We will sometimes use the Milnor primitives
$\mathrm{P}_1^s$ ($s\geq0$) defined recursively
by
\[
\mathrm{P}_1^0=\Sq(1),
\qquad
\mathrm{P}_1^s =
\Sq^{2^s}\mathrm{P}_1^{s-1}+\mathrm{P}_1^{s-1}\Sq^{2^s}
\quad (s\geq1).
\]
More generally, Margolis~\cite{HRM:Book} uses
the notation $\Sq(r_1,\ldots,r_\ell)$ for the
element dual to the monomial
$\xi_1^{r_1}\cdots\xi_\ell^{r_\ell}$ and sets
$\mathrm{P}_t^s = \Sq(0,\ldots,0,2^s)$ where
$2^s$ occurs in the $t$-th place.

\subsection*{Doubling}
Doubling is discussed by Margolis~\cite{HRM:Book}*{section~15.3}
and also by the present author~\cite{AB:Joker}*{section~2}.
The main thing to observe is that for any $n\geq1$
there is a grade halving homomorphism of Hopf algebras
\[
\StA(n) \xrightarrow{\iso} \StA(n-1)
\]
which induces a grade halving isomorphism
\[
\StA(n)/\!/\StE(n) \xrightarrow{\iso} \StA(n-1)
\]
under which the residue class of $\Sq^k$ maps
to $\Sq^{k/2}$ if $k$ is even and $0$ otherwise.

If $M$ is a left $\StA(n-1)$-module then it
becomes $\StA(n)/\!/\StE(n)$-module $\leftidx{^{(1)}}M$
with all gradings doubled, hence it is also
a left $\StA(n)$-module. For example, when $n=1$
the $\StA(1)$-module
$\leftidx{^{(1)}}\StA(0)\iso\StA(1)/\!/\StE(1)$
is realised by the cohomology of the $\kO$-module
$\kU$ shown in~\eqref{eq:H*kU}, and when $n=2$,
the $\StA(2)$-module
$\leftidx{^{(1)}}\StA(1)\iso\StA(2)/\!/\StE(2)$
(the double of $\StA(1)$) is realised by the
cohomology of the $\tmf$-module
$\tmf_1(3)\sim BP\langle2\rangle$.

\section{$\kO$-modules and $\StA(1)$-modules}
\label{sec:kO+StA(1)}

In this section we review some well known
results on the (co)homology of $\kO$-modules
localised at the prime~$2$. Such results
were originally due to Adams \& Priddy,
Mahowald and Milgram,
see~\cites{JFA&SBP:BSO,MM:bo-res,MM&RJM:Sq4};
the book of Margolis~\cite{HRM:Book} also
provides a thorough treatment of stable
module categories for finite subHopf
algebras of~$\StA$.

We begin by summarising some relationships
between relative Steenrod algebras and
their duals that involve connective
$K$-theory; these are obtained using the
ideas discussed in Section~\ref{sec:HtpyRModules}.
These will be used heavily in what follows.
\begin{thm}\label{thm:kO&kU-StA}
For $H=H\F_2$ we have the following
identifications of Hopf algebras:
\begin{align*}
H_{\kO}^*H &\iso \StA(1)^* = \StA(1),
& H^{\kO}_*H &\iso \StA(1)_*, & \ph{\StA(1)_*}& \\
H_{\kU}^*H &\iso \StE(1)^* = \StE(1),
& H^{\kU}_*H &\iso \StE(1)_*. & \ph{\StA(1)_*}&
\end{align*}
\end{thm}

Of course this means that the homology and
cohomology of a $\kO$-module are naturally
$\StA(1)$-modules and those of a $\kU$-module
are $\StE(1)$-modules.

The $\kO$-module $H=H\F_2$ can be given
a minimal cell structure with~$8$ cells
corresonding to the obvious basis of
$\StA(1)$-module $\StA(1)$.
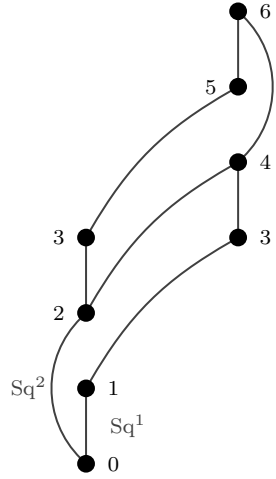
\begin{center}
\begin{tikzpicture}
%\Text[x=1,y=5.5]{$\StA(1)$}
\Vertex[x=2,y=6,color=black,size=.05,label={6},position=right,distance=0.5mm]{A6}
\Vertex[x=2,y=5,color=black,size=.1,label={5},position=left,distance=0.5mm]{A5}
\Vertex[x=2,y=4,size=.1,color=black,label={4},position=right,distance=0.5mm]{A4}
\Vertex[y=3,size=.1,color=black,label={3},position=left,distance=0.5mm]{A31}
\Vertex[x=2,y=3,size=.1,color=black,label={3},position=right,distance=0.5mm]{A32}
\Vertex[y=2,size=.1,color=black,label={2},position=left,distance=0.5mm]{A2}
\Vertex[y=1,size=.1,color=black,label={1},position=right,distance=0.5mm]{A1}
\Vertex[y=0,size=.1,color=black,label={0},position=right,distance=0.5mm]{A0}
\Edge[lw=0.75pt](A5)(A6)
\Edge[lw=0.75pt,bend=-45](A4)(A6)
\Edge[lw=0.75pt](A32)(A4)
\Edge[lw=0.75pt,bend=15](A31)(A5)
\Edge[lw=0.75pt](A2)(A31)
\Edge[lw=0.75pt,bend=15](A2)(A4)
\Edge[lw=0.75pt,bend=15](A1)(A32)
\Edge[lw=0.75pt,label={$\Sq^1$},position=right](A0)(A1)
\Edge[lw=0.75pt,label={$\Sq^2$},position=left,bend=45](A0)(A2)
\end{tikzpicture}
\end{center}
Since $H$ is a $\kO$ ring spectrum, by
Proposition~\ref{prop:PD-HA} it is stably
self dual and $D_\kO H\sim\Sigma^{-6}H$.

When working in the stable module category
$\StMod_{\StA(1)}$ we will often blur the
distinction between a module and its stable
equivalence class. We denote the kernel of
the counit $\epsilon\colon\StA(1)\to\F_2$
by $\mathcal{I}(1)$; this module is stably
invertible with stable inverse
$\mathcal{I}(1)^{-1}=\mathrm{D}\mathcal{I}(1)$,
represented by the linear dual of~$\mathcal{I}(1)$.

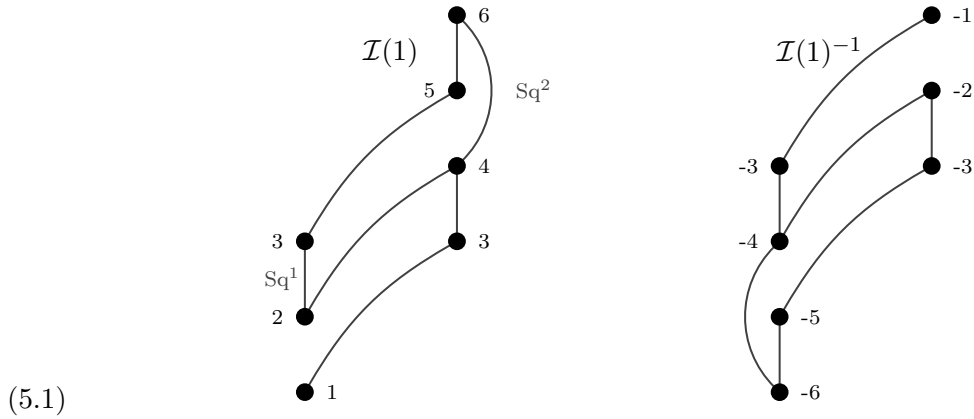
\begin{equation}\label{eq:I(1)&DI(1)}
\begin{tikzpicture}
\Text[x=1,y=5.5]{$\mathcal{I}(1)$}
\Vertex[x=2,y=6,color=black,size=.05,label={6},position=right,distance=0.5mm]{A6}
\Vertex[x=2,y=5,color=black,size=.1,label={5},position=left,distance=0.5mm]{A5}
\Vertex[x=2,y=4,size=.1,color=black,label={4},position=right,distance=0.5mm]{A4}
\Vertex[y=3,size=.1,color=black,label={3},position=left,distance=0.5mm]{A31}
\Vertex[x=2,y=3,size=.1,color=black,label={3},position=right,distance=0.5mm]{A32}
\Vertex[y=2,size=.1,color=black,label={2},position=left,distance=0.5mm]{A2}
\Vertex[y=1,size=.1,color=black,label={1},position=right,distance=0.5mm]{A1}
%\Vertex[y=0,size=.1,color=black,label={0},position=right,distance=0.5mm]{A0}
\Edge[lw=0.75pt](A5)(A6)
\Edge[lw=0.75pt,label={$\Sq^2$},position=right,bend=-45](A4)(A6)
\Edge[lw=0.75pt](A32)(A4)
\Edge[lw=0.75pt,bend=15](A31)(A5)
\Edge[lw=0.75pt,label={$\Sq^1$},position=left](A2)(A31)
\Edge[lw=0.75pt,bend=15](A2)(A4)
\Edge[lw=0.75pt,bend=15](A1)(A32)
%\Edge[lw=0.75pt,label={$\Sq^1$},position=right](A0)(A1)
%\Edge[lw=0.75pt,label={$\Sq^2$},position=left,bend=45](A0)(A2)
\end{tikzpicture}
\qquad\qquad\qquad
\begin{tikzpicture}
\Text[x=0.4,y=4.5]{$\mathcal{I}(1)^{-1}$}
%\Vertex[x=2,y=6,color=black,size=.05,label={6},position=right,distance=0.5mm]{A6}
\Vertex[x=2,y=5,color=black,size=.1,label={-1},position=right,distance=0.5mm]{A5}
\Vertex[x=2,y=4,size=.1,color=black,label={-2},position=right,distance=0.5mm]{A4}
\Vertex[y=3,size=.1,color=black,label={-3},position=left,distance=0.5mm]{A31}
\Vertex[x=2,y=3,size=.1,color=black,label={-3},position=right,distance=0.5mm]{A32}
\Vertex[y=2,size=.1,color=black,label={-4},position=left,distance=0.5mm]{A2}
\Vertex[y=1,size=.1,color=black,label={-5},position=right,distance=0.5mm]{A1}
\Vertex[y=0,size=.1,color=black,label={-6},position=right,distance=0.5mm]{A0}
%\Edge[lw=0.75pt](A5)(A6)
%\Edge[lw=0.75pt,bend=-45](A4)(A6)
\Edge[lw=0.75pt](A32)(A4)
\Edge[lw=0.75pt,bend=15](A31)(A5)
\Edge[lw=0.75pt](A2)(A31)
\Edge[lw=0.75pt,bend=15](A2)(A4)
\Edge[lw=0.75pt,bend=15](A1)(A32)
\Edge[lw=0.75pt](A0)(A1)
\Edge[lw=0.75pt,bend=45](A0)(A2)
\end{tikzpicture}
\end{equation}

Recall that there are two mutually inverse
endofunctors $\Omega,\Omega^{-1}$ of
$\StMod_{\StA(1)}$, where
\[
\Omega M = \mathcal{I}(1)\odot M,
\quad
\Omega^{-1} M = \mathcal{I}(1)^{-1}\odot M.
\]
These commute with products and preserve duals, i.e.,
\begin{align*}
(\Omega M)\odot N &= \Omega(M\odot N),
& (\Omega^{-1} M)\odot N &= \Omega^{-1}(M\odot N),  \\
\Omega \mathrm{D}M &= \mathrm{D}\Omega M,
& \Omega^{-1} \mathrm{D}M &= \mathrm{D}\Omega^{-1} M.
\end{align*}

Amongst the many cyclic quotient $\StA(1)$-modules
is the Joker module
%\[
%\J = \StA(1)/\StA(1)\{\Sq^3\}[-2]
%\quad
%\xymatrix@C=0.5cm@R=0.5cm{
%& {}_{2}\bullet_{\ph{2}}\ar@{-}[d]\ar@/^15pt/@{-}[dd] & \\
%& \bullet\ar@/_15pt/@{-}[dd] & \\
%& {}_0\bullet_{\ph{0}}\ar@/^15pt/@{-}[dd]^{\Sq^2} & \\
%& \bullet\ar@{-}[d]_{\Sq^1} & \\
%& {}_{-2}\bullet_{\ph{-2}} & \\
%}
%\]
\begin{equation}\label{eq:Joker}
\begin{tikzpicture}[scale=0.9]
\Text[x=-4,y=0]{$\J = \StA(1)/\StA(1)\{\Sq^3\}[-2]$}
\Vertex[x=1,y=3,size=.05,color=black,label={\tiny$2$},position=right,distance=1mm]{B2}
\Vertex[x=1,y=1.5,size=.05,color=black]{B1}
\Vertex[x=1,y=0,size=.05,color=black,label={\tiny$0$},position=right,distance=1mm]{B0}
\Vertex[x=1,y=-1.5,size=.05,color=black]{B-1}
\Vertex[x=1,y=-3,size=.05,color=black,label={\tiny$-2$},position=right,distance=1mm]{B-2}
\Edge[lw=0.75pt,bend=-45,position=right](B0)(B2)
\Edge[lw=0.75pt](B1)(B2)
\Edge[lw=0.75pt,bend=45,label={$\Sq^2$},position=left](B-1)(B1)
\Edge[lw=0.75pt,label={$\Sq^1$},position=left](B-2)(B-1)
\Edge[lw=0.75pt,bend=-45](B-2)(B0)
\end{tikzpicture}
\end{equation}
whose grading is arranged so that it is
a self dual $\StA(1)$-module, indeed it
represents the unique element of order~$2$
in the Picard group of~$\StMod_{\StA(1)}$,
$\Pic_{\StA(1)}$.

In $\StMod_{\StA(1)}$ we have the following
presentations as quotients of free modules:
\begin{align*}
\Omega\J
&= \StA(1)/\StA(1)\{\Sq^1,\Sq^2\Sq^1\Sq^2\}[1]
= \StA(1)/\StA(1)\{\Sq^1,\Sq^2\Sq^3\}[1], \\
\Omega^{-1}\J
&= \StA(1)/\StA(1)\{\Sq^2\}[-4], \\
\Omega^2\J
&=
(\StA(1)/\StA(1)\{\Sq^1\}[2] \oplus
\StA(1)/\StA(1)\{\Sq^1,\Sq^2\Sq^1\Sq^2\}[6])
/\StA(1)\{(\Sq^2\Sq^1\Sq^2,\Sq^1)\}, \\
\Omega^{-2}\J
&= \StA(1)/\StA(1)\{\Sq^2\Sq^1\}[-7].
\end{align*}
These modules have the following diagrams
(the first two are often called the
\emph{question mark} and \emph{upside down
question mark} modules).
\begin{equation}\label{eq:Omegas}
\begin{tikzpicture}[scale=.9]
\Text[x=0,y=1]{$\Omega\mathcal{J}$}
\Vertex[x=0,y=0,size=.05,color=black,label={4},position=left,distance=0.5mm]{A4}
\Vertex[x=0,y=-1,size=.05,color=black,label={3},position=left,distance=0.5mm]{A3}
\Vertex[x=0,y=-3,size=.05,color=black,label={1},position=left,distance=0.5mm]{A1}
\Edge[lw=0.75pt](A3)(A4)
\Edge[lw=0.75pt,bend=-45](A1)(A3)
\Text[x=2,y=1]{$\Omega^{-1}\mathcal{J}$}
\Vertex[x=2,y=0,size=.05,color=black,label={-1},position=left,distance=0.5mm]{B-1}
\Vertex[x=2,y=-2,size=.05,color=black,label={-3},position=left,distance=0.5mm]{B-3}
\Vertex[x=2,y=-3,size=.05,color=black,label={-4},position=left,distance=0.5mm]{B-4}
\Edge[lw=0.75pt,bend=-45](B-3)(B-1)
\Edge[lw=0.75pt](B-4)(B-3)
\Text[x=4,y=1]{$\Omega^2\mathcal{J}$}
\Vertex[x=4,y=0,size=.05,color=black,label={7},position=left,distance=0.5mm]{C7}
\Vertex[x=4,y=-1,size=.05,color=black,label={6},position=left,distance=0.5mm]{C6}
\Vertex[x=4,y=-2,size=.05,color=black,label={5},position=left,distance=0.5mm]{C5}
\Vertex[x=4,y=-3,size=.05,color=black,label={4},position=left,distance=0.5mm]{C4}
\Vertex[x=4,y=-5,size=.05,color=black,label={2},position=left,distance=0.5mm]{C2}
\Edge[lw=0.75pt](C6)(C7)
\Edge[lw=0.75pt,bend=-45](C5)(C7)
\Edge[lw=0.75pt](C4)(C5)
\Edge[lw=0.75pt,bend=-45](C2)(C4)
\Text[x=6,y=1]{$\Omega^{-2}\J$}
\Vertex[x=6,y=0,size=.05,color=black,label={-2},position=left,distance=0.5mm]{D-2}
\Vertex[x=6,y=-2,size=.05,color=black,label={-4},position=left,distance=0.5mm]{D-4}
\Vertex[x=6,y=-3,size=.05,color=black,label={-5},position=left,distance=0.5mm]{D-5}
\Vertex[x=6,y=-4,size=.05,color=black,label={-6},position=left,distance=0.5mm]{D-6}
\Vertex[x=6,y=-5,size=.05,color=black,label={-7},position=left,distance=0.5mm]{D-7}
\Edge[lw=0.75pt,bend=-45](D-4)(D-2)
\Edge[lw=0.75pt](D-5)(D-4)
\Edge[lw=0.75pt,bend=-45](D-7)(D-5)
\Edge[lw=0.75pt](D-7)(D-6)
\end{tikzpicture}
\end{equation}

We recall the structure of $\Ext_{\StA(1)}(\F_2,\F_2)$,
the Adams $\mathrm{E}_2$-term for~$\pi_*\kO$,
part of which appears in Figure~\ref{fig:kO}.
We also have
\begin{equation}\label{eq:ExtI(1)}
\Ext_{\StA(1)}^{s,*}(\mathcal{I}(1),\F_2)
\iso
\Ext_{\StA(1)}^{s+1,*}(\F_2,\F_2),
\end{equation}
whose chart is a shifted version of
Figure~\ref{fig:kO} with the~$\bullet$
at~$(0,0)$ removed. For the dual
$\mathrm{D}\mathcal{I}(1)$,
\begin{equation}\label{eq:ExtDI(1)}
\Ext_{\StA(1)}^{s,*}(\mathrm{D}\mathcal{I}(1),\F_2)
\iso
\begin{dcases*}
\F_2 & if $s=0$, \\
\Ext_{\StA(1)}^{s-1,*}(\F_2,\F_2) & if $s>0$.
\end{dcases*}
\end{equation}
Adams \& Priddy \cite{JFA&SBP:BSO}*{tables~3.10 \&~3.11}
give the $\Ext$ groups of the two question
mark modules $\Omega\J$ and $\Omega^{-1}\J$.

Such $\StA(1)$-modules can often be realised
as cohomology of $\kO$-modules, i.e., as
$H_{\kO}^*X$ for some $\kO$-module~$X$.
Indeed, in some cases~$X=\kO\wedge Z$
for an $S$-module~$Z$. For example,
there are $S$-modules
$S^0\cup_\eta e^2\cup_2 e^3$ and
$S^0\cup_2 e^1\cup_\eta e^3$ so that as
$\StA(1)$-modules,
\begin{align*}
H_{\kO}^*(\kO\wedge S^0\cup_\eta e^2\cup_2 e^3)
&\iso H^*(S^0\cup_\eta e^2\cup_2 e^3)
\iso \Omega\mathcal{J}[-1], \\
H_{\kO}^*(\kO\wedge S^0\cup_2 e^1\cup_\eta e^3)
&\iso H^*(S^0\cup_2 e^1\cup_\eta e^3)
\iso \Omega^{-1}\mathcal{J}[4].
\end{align*}
In~\cite{AB:Joker}, the Joker was shown to be
realisable as $H^*J$ for an $S$-module~$J$,
so as $\StA(1)$-modules,
$H_{\kO}^*(\kO\wedge J) \iso H^*J$. The
realisability of the self dual cyclic
$\StA(1)$-module $\StA(1)/\!/\StA(0)$ provides
a more interesting question.

\begin{equation}\label{eq:HZ}
\begin{tikzpicture}
\Text[x=-2.5,y=1.875]{$\StA(1)/\!/\StA(0)=\StA(1)/\StA(1)\{\Sq^1\}$}
\Vertex[x=1,y=3.75,size=.05,color=black,label={$5$},position=right,distance=0.1mm]{B5}
\Vertex[x=1,y=2.25,size=.05,color=black]{B3}
\Vertex[x=1,y=1.5,size=.05,color=black]{B2}
\Vertex[x=1,y=0,size=.05,color=black,label={$0$},position=right,distance=0.1mm]{B0}
\Edge[lw=0.75pt,bend=-45](B3)(B5)
\Edge[lw=0.75pt,bend=-45,label={$\Sq^2$},position=right](B0)(B2)
\Edge[lw=0.75pt,label={$\Sq^1$},position=right](B2)(B3)
\end{tikzpicture}
\end{equation}

In $\mathscr{M}_S$, no CW complex of the
form $S^0\cup_\eta e^2\cup_2e^3\cup_\eta e^5$
can exist since the Toda bracket
$\langle\eta,2,\eta\rangle=\{2\nu,-2\nu\}\subseteq\pi_3S^0$
does not contain~$0$; alternatively, its
existence is contradicted by the relation
$\Sq^1\Sq^2\Sq^1=\Sq^4\Sq^1+\Sq^1\Sq^4$
in the Steenrod algebra $\StA$. However,
taken in~$\pi_*\kO$, the Toda bracket
$\langle\eta,2,\eta\rangle\subseteq\pi_3\kO$
contains the images of $\pm2\nu$ which are
both zero, hence there is a CW $\kO$-module
with this cohomology. Of course~$H\Z$ is a
$\kO$-module whose cohomology agrees with
this $\StA(1)$-module, so we are describing
a minimal CW structure for it. Since~$H\Z$
is a $\kO$ ring spectrum and $H^{\kO}_*H\Z$
satisfies Poincar\'e duality,
Proposition~\ref{prop:PD-HA} applies. Hence
the $\kO$-module~$H\Z$ is stably self dual
with $D_\kO H\Z \sim \Sigma^{-5}H\Z$.

The $\kO$-modules discussed above fit into
the Whitehead tower of~$\kO$ shown
in~\eqref{eq:BottCW} as CW $\kO$-modules,
where the numbers indicate degrees of cells
in a minimal CW structure; we learnt about
the following from Bob Bruner and John Rognes,
see for example
\cites{RRB:PostnikovkOkU,RRB:Picard,RRB&JPCG}.
On applying $H_{\kO}^*(-)$ we obtain the
$\StA(1)$-module extension of~\eqref{eq:BottCW-H^*}
and~\eqref{eq:Bottseq} which represents an
element of $\Ext_{\StA(1)}^{4,12}(\F_2,\F_2)$,
and in turn this represents the Bott periodicity
element in~$\pi_8\kO$ in the Adams spectral
sequence. This periodicity is also visible
in Figure~\ref{fig:kO} which shows part of
$\Ext_{\StA(1)}^{*,*}(\F_2,\F_2)$ where it
is given by Yoneda product with the basis
element of
$\Ext_{\StA(1)}^{8,4}(\F_2,\F_2)$.
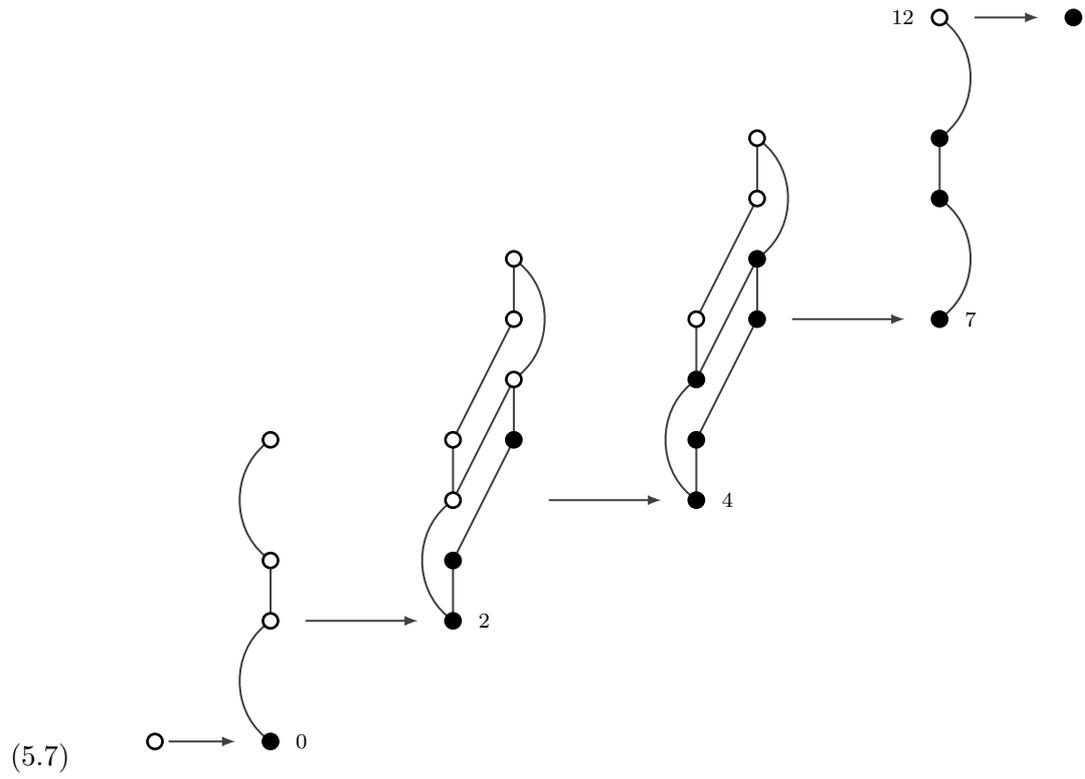
\begin{equation}\label{eq:BottCW}
\begin{tikzpicture}[scale=0.8]
\Vertex[x=0.1,y=0,color=white,size=.05]{F0}
\Vertex[x=0.2,y=0,color=black,size=.05,Pseudo]{F01}
\Vertex[x=2,y=0,color=black,size=.05,label={$0$},position=right,distance=1mm]{A0}
\Vertex[x=2,y=2,color=white,size=.05]{A4}
\Vertex[x=2,y=3,color=white,size=.05]{A6}
%\Vertex[x=2,y=3.5,color=white,size=.05]{A7}
\Vertex[x=2,y=5,color=white,size=.05]{A10}
\Edge[lw=0.75pt,bend=50](A0)(A4)
\Edge[lw=0.75pt](A4)(A6)
\Edge[lw=0.75pt,bend=50](A6)(A10)
\Vertex[x=5,y=2,color=black,size=.05,label={$2$},position=right,distance=1mm]{B4}
\Vertex[x=5,y=3,color=black,size=.05]{B6}
\Vertex[x=5,y=4,color=white,size=.05]{B8}
\Vertex[x=5,y=5,color=white,size=.05]{B10}
\Vertex[x=6,y=5,color=black,size=.05]{B10-2}
\Vertex[x=6,y=6,color=white,size=.05]{B12}
\Vertex[x=6,y=7,color=white,size=.05]{B14}
\Vertex[x=6,y=8,color=white,size=.05]{B16}
\Edge[lw=0.75pt](B4)(B6)
\Edge[lw=0.75pt,bend=50,position=right](B4)(B8)
\Edge[lw=0.75pt](B6)(B10-2)
\Edge[lw=0.75pt](B8)(B10)
\Edge[lw=0.75pt](B8)(B12)
\Edge[lw=0.75pt](B10-2)(B12)
\Edge[lw=0.75pt](B10)(B14)
\Edge[lw=0.75pt](B14)(B16)
\Edge[lw=0.75pt,bend=-50](B12)(B16)
\Vertex[x=9,y=4,color=black,size=.05,label={$4$},position=right,distance=1mm]{C8}
\Vertex[x=9,y=5,color=black,size=.05]{C10}
\Vertex[x=9,y=6,color=black,size=.05]{C12}
\Vertex[x=9,y=7,color=white,size=.05]{C14}
\Vertex[x=10,y=7,color=black,size=.05]{C14-2}
\Vertex[x=10,y=8,color=black,size=.05]{C16}
\Vertex[x=10,y=9,color=white,size=.05]{C18}
\Vertex[x=10,y=10,color=white,size=.05]{C20}
\Edge[lw=0.75pt](C8)(C10)
\Edge[lw=0.75pt,bend=50](C8)(C12)
\Edge[lw=0.75pt](C10)(C14-2)
\Edge[lw=0.75pt](C12)(C14)
\Edge[lw=0.75pt](C12)(C16)
\Edge[lw=0.75pt](C14)(C18)
\Edge[lw=0.75pt](C14-2)(C16)
\Edge[lw=0.75pt,bend=-50](C16)(C20)
\Edge[lw=0.75pt](C18)(C20)
\Vertex[x=13,y=7,color=black,size=.05,label={$7$},position=right,distance=1mm]{D14}
\Vertex[x=13,y=9,color=black,size=.05]{D18}
\Vertex[x=13,y=10,color=black,size=.05]{D20}
\Vertex[x=13,y=12,color=white,size=.05,label={$12$},position=left,distance=1mm]{D24}
\Edge[lw=0.75pt,bend=-50](D14)(D18)
\Edge[lw=0.75pt](D18)(D20)
\Edge[lw=0.75pt,bend=-50](D20)(D24)
\Vertex[x=1.8,y=0,Pseudo]{F02}
\Edge[lw=0.75pt,Direct](F01)(F02)
\Vertex[x=2.2,y=2,Pseudo]{A01}
\Vertex[x=4.8,y=2,Pseudo]{A02}
\Edge[lw=0.75pt,Direct](A01)(A02)
\Vertex[x=6.2,y=4,Pseudo]{B01}
\Vertex[x=8.8,y=4,Pseudo]{B02}
\Edge[lw=0.75pt,Direct](B01)(B02)
\Vertex[x=10.2,y=7,Pseudo]{C01}
\Vertex[x=12.8,y=7,Pseudo]{C02}
\Edge[lw=0.75pt,Direct](C01)(C02)
\Vertex[x=13.2,y=12,Pseudo]{F24-1}
\Vertex[x=15,y=12,Pseudo]{F24-2}
\Edge[lw=0.75pt,Direct](F24-1)(F24-2)
\Vertex[x=15.2,y=12,color=black,size=.05]{F12}
\end{tikzpicture}
\end{equation}

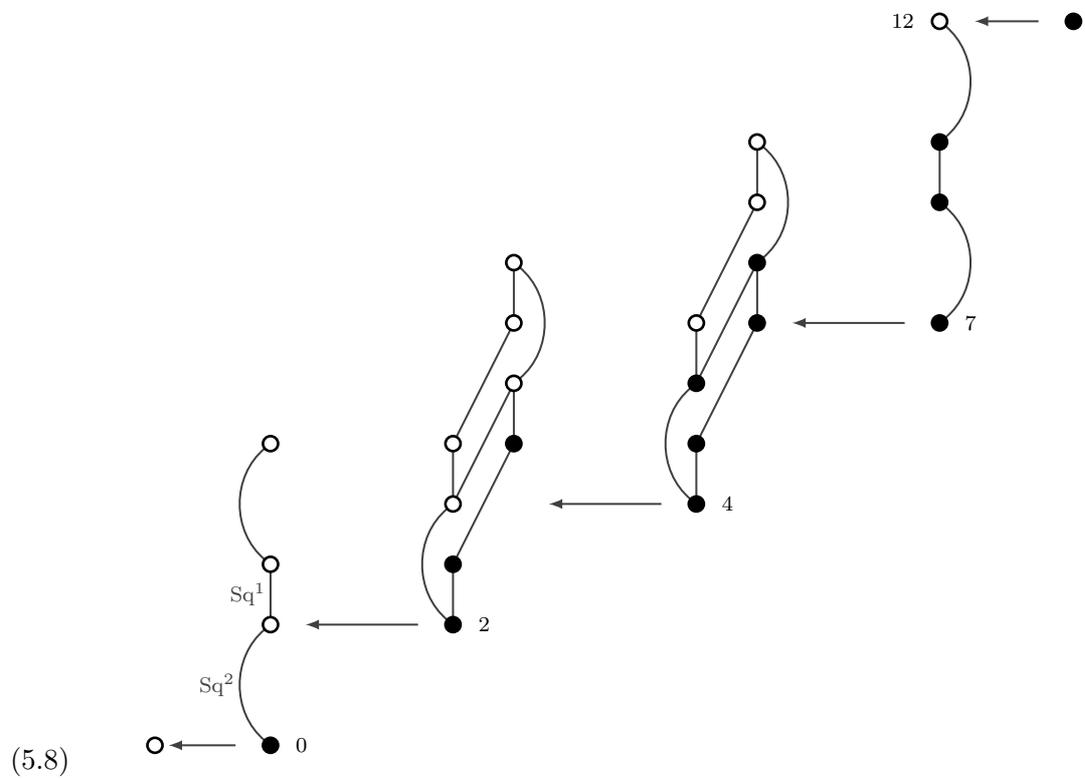
\begin{equation}\label{eq:BottCW-H^*}
\begin{tikzpicture}[scale=0.8]
\Vertex[x=0.1,y=0,color=white,size=.05]{F0}
\Vertex[x=0.2,y=0,color=black,size=.05,Pseudo]{F01}
\Vertex[x=2,y=0,color=black,size=.05,label={$0$},position=right,distance=1mm]{A0}
\Vertex[x=2,y=2,color=white,size=.05]{A4}
\Vertex[x=2,y=3,color=white,size=.05]{A6}
\Vertex[x=2,y=5,color=white,size=.05]{A10}
\Edge[lw=0.75pt,bend=50,label={$\Sq^2$},position=left](A0)(A4)
\Edge[lw=0.75pt,label={$\Sq^1$},position=left](A4)(A6)
\Edge[lw=0.75pt,bend=50](A6)(A10)
\Vertex[x=5,y=2,color=black,size=.05,label={$2$},position=right,distance=1mm]{B4}
\Vertex[x=5,y=3,color=black,size=.05]{B6}
\Vertex[x=5,y=4,color=white,size=.05]{B8}
\Vertex[x=5,y=5,color=white,size=.05]{B10}
\Vertex[x=6,y=5,color=black,size=.05]{B10-2}
\Vertex[x=6,y=6,color=white,size=.05]{B12}
\Vertex[x=6,y=7,color=white,size=.05]{B14}
\Vertex[x=6,y=8,color=white,size=.05]{B16}
\Edge[lw=0.75pt](B4)(B6)
\Edge[lw=0.75pt,bend=50,position=right](B4)(B8)
\Edge[lw=0.75pt](B6)(B10-2)
\Edge[lw=0.75pt](B8)(B10)
\Edge[lw=0.75pt](B8)(B12)
\Edge[lw=0.75pt](B10-2)(B12)
\Edge[lw=0.75pt](B10)(B14)
\Edge[lw=0.75pt](B14)(B16)
\Edge[lw=0.75pt,bend=-50](B12)(B16)
\Vertex[x=9,y=4,color=black,size=.05,label={$4$},position=right,distance=1mm]{C8}
\Vertex[x=9,y=5,color=black,size=.05]{C10}
\Vertex[x=9,y=6,color=black,size=.05]{C12}
\Vertex[x=9,y=7,color=white,size=.05]{C14}
\Vertex[x=10,y=7,color=black,size=.05]{C14-2}
\Vertex[x=10,y=8,color=black,size=.05]{C16}
\Vertex[x=10,y=9,color=white,size=.05]{C18}
\Vertex[x=10,y=10,color=white,size=.05]{C20}
\Edge[lw=0.75pt](C8)(C10)
\Edge[lw=0.75pt,bend=50](C8)(C12)
\Edge[lw=0.75pt](C10)(C14-2)
\Edge[lw=0.75pt](C12)(C14)
\Edge[lw=0.75pt](C12)(C16)
\Edge[lw=0.75pt](C14)(C18)
\Edge[lw=0.75pt](C14-2)(C16)
\Edge[lw=0.75pt,bend=-50](C16)(C20)
\Edge[lw=0.75pt](C18)(C20)
\Vertex[x=13,y=7,color=black,size=.05,label={$7$},position=right,distance=1mm]{D14}
\Vertex[x=13,y=9,color=black,size=.05]{D18}
\Vertex[x=13,y=10,color=black,size=.05]{D20}
\Vertex[x=13,y=12,color=white,size=.05,label={$12$},position=left,distance=1mm]{D24}
\Edge[lw=0.75pt,bend=-50](D14)(D18)
\Edge[lw=0.75pt](D18)(D20)
\Edge[lw=0.75pt,bend=-50](D20)(D24)
\Vertex[x=1.8,y=0,Pseudo]{F02}
\Edge[lw=0.75pt,Direct](F02)(F01)
\Vertex[x=2.2,y=2,Pseudo]{A01}
\Vertex[x=4.8,y=2,Pseudo]{A02}
\Edge[lw=0.75pt,Direct](A02)(A01)
\Vertex[x=6.2,y=4,Pseudo]{B01}
\Vertex[x=8.8,y=4,Pseudo]{B02}
\Edge[lw=0.75pt,Direct](B02)(B01)
\Vertex[x=10.2,y=7,Pseudo]{C01}
\Vertex[x=12.8,y=7,Pseudo]{C02}
\Edge[lw=0.75pt,Direct](C02)(C01)
\Vertex[x=13.2,y=12,Pseudo]{F24-1}
\Vertex[x=15,y=12,Pseudo]{F24-2}
\Edge[lw=0.75pt,Direct](F24-2)(F24-1)
\Vertex[x=15.2,y=12,color=black,size=.05]{F12}
\end{tikzpicture}
\end{equation}
\begin{equation}\label{eq:Bottseq}
0\la \F_2 \la \StA(1)/\StA(1)\{\Sq^1\}
\la \StA(1)[2] \la \StA(1)[4]
\la \StA(1)/\StA(1)\{\Sq^1\}[7]
\la \F_2[12]\la 0
\end{equation}

\begin{figure}[ht]
\begin{tikzpicture}[scale=.70]

\clip (-1.5,-0.5) rectangle ( 20.3, 12.50);
\draw[color=lightgray] (0,0) grid [step=4] (20.3,12.5);

\foreach \n in {0,4,...,20}
{
\def\nn{\n-0};
\node [below] at (\nn,0) {{\tiny$\n$}};
}

\foreach \s in {0,4,...,12}
{
\def\ss{\s-0};
\node [left] at (-0.1,\ss,0) {{\tiny$\s$}};
}

\draw [] ( 0.00, 0.00) circle [radius=0.075];
\node [left] at  ( 0.00, 0.00) {\ph{0}};

\draw [fill] ( 0.00, 1.00) circle [radius=0.05];
\node [left] at  ( 0.00, 1.00) {\ph{0}};

\draw [fill] ( 1.00, 1.00) circle [radius=0.05];
\node [left] at  ( 1.00, 1.00) {\ph{1}};

\draw [fill] ( 0.00, 2.00) circle [radius=0.05];
\node [left] at  ( 0.00, 2.00) {\ph{0}};

\draw [fill] ( 2.00, 2.00) circle [radius=0.05];
\node [left] at  ( 2.00, 2.00) {\ph{1}};

\draw [fill] ( 0.00, 3.00) circle [radius=0.05];
\node [left] at  ( 0.00, 3.00) {\ph{0}};

\draw [fill] ( 4.00, 3.00) circle [radius=0.05];
\node [left] at  ( 4.00, 3.00) {\ph{1}};

\draw [fill] ( 0.00, 4.00) circle [radius=0.05];
\node [left] at  ( 0.00, 4.00) {\ph{0}};

\draw [fill] ( 4.00, 4.00) circle [radius=0.05];
\node [left] at  ( 4.00, 4.00) {\ph{1}};

\draw [fill] ( 8.00, 4.00) circle [radius=0.05];
\node [left] at  ( 8.00, 4.00) {\ph{2}};

\draw [fill] ( 0.00, 5.00) circle [radius=0.05];
\node [left] at  ( 0.00, 5.00) {\ph{0}};

\draw [fill] ( 4.00, 5.00) circle [radius=0.05];
\node [left] at  ( 4.00, 5.00) {\ph{1}};

\draw [fill] ( 8.00, 5.00) circle [radius=0.05];
\node [left] at  ( 8.00, 5.00) {\ph{2}};

\draw [fill] ( 9.00, 5.00) circle [radius=0.05];
\node [left] at  ( 9.00, 5.00) {\ph{3}};

\draw [fill] ( 0.00, 6.00) circle [radius=0.05];
\node [left] at  ( 0.00, 6.00) {\ph{0}};

\draw [fill] ( 4.00, 6.00) circle [radius=0.05];
\node [left] at  ( 4.00, 6.00) {\ph{1}};

\draw [fill] ( 8.00, 6.00) circle [radius=0.05];
\node [left] at  ( 8.00, 6.00) {\ph{2}};

\draw [fill] (10.00, 6.00) circle [radius=0.05];
\node [left] at  (10.00, 6.00) {\ph{3}};

\draw [fill] ( 0.00, 7.00) circle [radius=0.05];
\node [left] at  ( 0.00, 7.00) {\ph{0}};

\draw [fill] ( 4.00, 7.00) circle [radius=0.05];
\node [left] at  ( 4.00, 7.00) {\ph{1}};

\draw [fill] ( 8.00, 7.00) circle [radius=0.05];
\node [left] at  ( 8.00, 7.00) {\ph{2}};

\draw [fill] (12.00, 7.00) circle [radius=0.05];
\node [left] at  (12.00, 7.00) {\ph{3}};

\draw [fill] ( 0.00, 8.00) circle [radius=0.05];
\node [left] at  ( 0.00, 8.00) {\ph{0}};

\draw [fill] ( 4.00, 8.00) circle [radius=0.05];
\node [left] at  ( 4.00, 8.00) {\ph{1}};

\draw [fill] ( 8.00, 8.00) circle [radius=0.05];
\node [left] at  ( 8.00, 8.00) {\ph{2}};

\draw [fill] (12.00, 8.00) circle [radius=0.05];
\node [left] at  (12.00, 8.00) {\ph{3}};

\draw [fill] (16.00, 8.00) circle [radius=0.05];
\node [left] at  (16.00, 8.00) {\ph{4}};

\draw [fill] ( 0.00, 9.00) circle [radius=0.05];
\node [left] at  ( 0.00, 9.00) {\ph{0}};

\draw [fill] ( 4.00, 9.00) circle [radius=0.05];
\node [left] at  ( 4.00, 9.00) {\ph{1}};

\draw [fill] ( 8.00, 9.00) circle [radius=0.05];
\node [left] at  ( 8.00, 9.00) {\ph{2}};

\draw [fill] (12.00, 9.00) circle [radius=0.05];
\node [left] at  (12.00, 9.00) {\ph{3}};

\draw [fill] (16.00, 9.00) circle [radius=0.05];
\node [left] at  (16.00, 9.00) {\ph{4}};

\draw [fill] (17.00, 9.00) circle [radius=0.05];
\node [left] at  (17.00, 9.00) {\ph{5}};

\draw [fill] ( 0.00,10.00) circle [radius=0.05];
\node [left] at  ( 0.00,10.00) {\ph{0}};

\draw [fill] ( 4.00,10.00) circle [radius=0.05];
\node [left] at  ( 4.00,10.00) {\ph{1}};

\draw [fill] ( 8.00,10.00) circle [radius=0.05];
\node [left] at  ( 8.00,10.00) {\ph{2}};

\draw [fill] (12.00,10.00) circle [radius=0.05];
\node [left] at  (12.00,10.00) {\ph{3}};

\draw [fill] (16.00,10.00) circle [radius=0.05];
\node [left] at  (16.00,10.00) {\ph{4}};

\draw [fill] (18.00,10.00) circle [radius=0.05];
\node [left] at  (18.00,10.00) {\ph{5}};

\draw [fill] ( 0.00,11.00) circle [radius=0.05];
\node [left] at  ( 0.00,11.00) {\ph{0}};

\draw [fill] ( 4.00,11.00) circle [radius=0.05];
\node [left] at  ( 4.00,11.00) {\ph{1}};

\draw [fill] ( 8.00,11.00) circle [radius=0.05];
\node [left] at  ( 8.00,11.00) {\ph{2}};

\draw [fill] (12.00,11.00) circle [radius=0.05];
\node [left] at  (12.00,11.00) {\ph{3}};

\draw [fill] (16.00,11.00) circle [radius=0.05];
\node [left] at  (16.00,11.00) {\ph{4}};

\draw [fill] (20.00,11.00) circle [radius=0.05];
\node [left] at  (20.00,11.00) {\ph{5}};

\draw [fill] ( 0.00,12.00) circle [radius=0.05];
\node [left] at  ( 0.00,12.00) {\ph{0}};

\draw [fill] ( 4.00,12.00) circle [radius=0.05];
\node [left] at  ( 4.00,12.00) {\ph{1}};

\draw [fill] ( 8.00,12.00) circle [radius=0.05];
\node [left] at  ( 8.00,12.00) {\ph{2}};

\draw [fill] (12.00,12.00) circle [radius=0.05];
\node [left] at  (12.00,12.00) {\ph{3}};

\draw [fill] (16.00,12.00) circle [radius=0.05];
\node [left] at  (16.00,12.00) {\ph{4}};

\draw [fill] (20.00,12.00) circle [radius=0.05];
\node [left] at  (20.00,12.00) {\ph{5}};

\draw [fill] (24.00,12.00) circle [radius=0.05];
\node [left] at  (24.00,12.00) {\ph{6}};

\draw [fill] ( 0.00,13.00) circle [radius=0.05];
\node [left] at  ( 0.00,13.00) {\ph{0}};

\draw [fill] ( 4.00,13.00) circle [radius=0.05];
\node [left] at  ( 4.00,13.00) {\ph{1}};

\draw [fill] ( 8.00,13.00) circle [radius=0.05];
\node [left] at  ( 8.00,13.00) {\ph{2}};

\draw [fill] (12.00,13.00) circle [radius=0.05];
\node [left] at  (12.00,13.00) {\ph{3}};

\draw [fill] (16.00,13.00) circle [radius=0.05];
\node [left] at  (16.00,13.00) {\ph{4}};

\draw [fill] (20.00,13.00) circle [radius=0.05];
\node [left] at  (20.00,13.00) {\ph{5}};

\draw [fill] (24.00,13.00) circle [radius=0.05];
\node [left] at  (24.00,13.00) {\ph{6}};

\draw [fill] (25.00,13.00) circle [radius=0.05];
\node [left] at  (25.00,13.00) {\ph{7}};

\draw [fill] ( 0.00,14.00) circle [radius=0.05];
\node [left] at  ( 0.00,14.00) {\ph{0}};

\draw [fill] ( 4.00,14.00) circle [radius=0.05];
\node [left] at  ( 4.00,14.00) {\ph{1}};

\draw [fill] ( 8.00,14.00) circle [radius=0.05];
\node [left] at  ( 8.00,14.00) {\ph{2}};

\draw [fill] (12.00,14.00) circle [radius=0.05];
\node [left] at  (12.00,14.00) {\ph{3}};

\draw [fill] (16.00,14.00) circle [radius=0.05];
\node [left] at  (16.00,14.00) {\ph{4}};

\draw [fill] (20.00,14.00) circle [radius=0.05];
\node [left] at  (20.00,14.00) {\ph{5}};

\draw [fill] (24.00,14.00) circle [radius=0.05];
\node [left] at  (24.00,14.00) {\ph{6}};

\draw [fill] (26.00,14.00) circle [radius=0.05];
\node [left] at  (26.00,14.00) {\ph{7}};

\draw [fill] ( 0.00,15.00) circle [radius=0.05];
\node [left] at  ( 0.00,15.00) {\ph{0}};

\draw [fill] ( 4.00,15.00) circle [radius=0.05];
\node [left] at  ( 4.00,15.00) {\ph{1}};

\draw [fill] ( 8.00,15.00) circle [radius=0.05];
\node [left] at  ( 8.00,15.00) {\ph{2}};

\draw [fill] (12.00,15.00) circle [radius=0.05];
\node [left] at  (12.00,15.00) {\ph{3}};

\draw [fill] (16.00,15.00) circle [radius=0.05];
\node [left] at  (16.00,15.00) {\ph{4}};

\draw [fill] (20.00,15.00) circle [radius=0.05];
\node [left] at  (20.00,15.00) {\ph{5}};

\draw [fill] (24.00,15.00) circle [radius=0.05];
\node [left] at  (24.00,15.00) {\ph{6}};

\draw [fill] (28.00,15.00) circle [radius=0.05];
\node [left] at  (28.00,15.00) {\ph{7}};

\draw ( 0.00, 1.00) --( 0.00, 0.00);
\draw ( 1.00, 1.00) --( 0.00, 0.00);
\draw ( 0.00, 2.00) --( 0.00, 1.00);
\draw ( 2.00, 2.00) --( 1.00, 1.00);
\draw ( 0.00, 3.00) --( 0.00, 2.00);
\draw ( 0.00, 4.00) --( 0.00, 3.00);
\draw ( 4.00, 4.00) --( 4.00, 3.00);
\draw ( 0.00, 5.00) --( 0.00, 4.00);
\draw ( 4.00, 5.00) --( 4.00, 4.00);
\draw ( 8.00, 5.00) --( 8.00, 4.00);
\draw ( 9.00, 5.00) --( 8.00, 4.00);
\draw ( 0.00, 6.00) --( 0.00, 5.00);
\draw ( 4.00, 6.00) --( 4.00, 5.00);
\draw ( 8.00, 6.00) --( 8.00, 5.00);
\draw (10.00, 6.00) --( 9.00, 5.00);
\draw ( 0.00, 7.00) --( 0.00, 6.00);
\draw ( 4.00, 7.00) --( 4.00, 6.00);
\draw ( 8.00, 7.00) --( 8.00, 6.00);
\draw ( 0.00, 8.00) --( 0.00, 7.00);
\draw ( 4.00, 8.00) --( 4.00, 7.00);
\draw ( 8.00, 8.00) --( 8.00, 7.00);
\draw (12.00, 8.00) --(12.00, 7.00);
\draw ( 0.00, 9.00) --( 0.00, 8.00);
\draw ( 4.00, 9.00) --( 4.00, 8.00);
\draw ( 8.00, 9.00) --( 8.00, 8.00);
\draw (12.00, 9.00) --(12.00, 8.00);
\draw (16.00, 9.00) --(16.00, 8.00);
\draw (17.00, 9.00) --(16.00, 8.00);
\draw ( 0.00,10.00) --( 0.00, 9.00);
\draw ( 4.00,10.00) --( 4.00, 9.00);
\draw ( 8.00,10.00) --( 8.00, 9.00);
\draw (12.00,10.00) --(12.00, 9.00);
\draw (16.00,10.00) --(16.00, 9.00);
\draw (18.00,10.00) --(17.00, 9.00);
\draw ( 0.00,11.00) --( 0.00,10.00);
\draw ( 4.00,11.00) --( 4.00,10.00);
\draw ( 8.00,11.00) --( 8.00,10.00);
\draw (12.00,11.00) --(12.00,10.00);
\draw (16.00,11.00) --(16.00,10.00);
\draw ( 0.00,12.00) --( 0.00,11.00);
\draw ( 4.00,12.00) --( 4.00,11.00);
\draw ( 8.00,12.00) --( 8.00,11.00);
\draw (12.00,12.00) --(12.00,11.00);
\draw (16.00,12.00) --(16.00,11.00);
\draw (20.00,12.00) --(20.00,11.00);
\draw ( 0.00,13.00) --( 0.00,12.00);
\draw ( 4.00,13.00) --( 4.00,12.00);
\draw ( 8.00,13.00) --( 8.00,12.00);
\draw (12.00,13.00) --(12.00,12.00);
\draw (16.00,13.00) --(16.00,12.00);
\draw (20.00,13.00) --(20.00,12.00);
\draw (24.00,13.00) --(24.00,12.00);
\draw (25.00,13.00) --(24.00,12.00);
\draw ( 0.00,14.00) --( 0.00,13.00);
\draw ( 4.00,14.00) --( 4.00,13.00);
\draw ( 8.00,14.00) --( 8.00,13.00);
\draw (12.00,14.00) --(12.00,13.00);
\draw (16.00,14.00) --(16.00,13.00);
\draw (20.00,14.00) --(20.00,13.00);
\draw (24.00,14.00) --(24.00,13.00);
\draw (26.00,14.00) --(25.00,13.00);
\draw ( 0.00,15.00) --( 0.00,14.00);
\draw ( 4.00,15.00) --( 4.00,14.00);
\draw ( 8.00,15.00) --( 8.00,14.00);
\draw (12.00,15.00) --(12.00,14.00);
\draw (16.00,15.00) --(16.00,14.00);
\draw (20.00,15.00) --(20.00,14.00);
\draw (24.00,15.00) --(24.00,14.00);
\draw ( 0.00,16.00) --( 0.00,15.00);
\draw ( 4.00,16.00) --( 4.00,15.00);
\draw ( 8.00,16.00) --( 8.00,15.00);
\draw (12.00,16.00) --(12.00,15.00);
\draw (16.00,16.00) --(16.00,15.00);
\draw (20.00,16.00) --(20.00,15.00);
\draw (24.00,16.00) --(24.00,15.00);
\draw (28.00,16.00) --(28.00,15.00);

\node [right] at (17,-0.35,0) {{$t-s\to$}};
\node [left] at (-0.1,10,0) {{$s\uparrow$}};
\end{tikzpicture}
\caption{$\Ext_{\StA(1)}^{s,t}(\F_2,\F_2)$:
for $0 \leq t-s \leq 20$, $0 \leq s \leq 12$.
Removing the generator at $(0,0)$ and shifting
down and left gives $\Ext_{\StA(1)}(\mathcal{I}(1),\F_2)$.}
\label{fig:kO}
\end{figure}

The cofibres of the maps in~\eqref{eq:BottCW}
have the following cell structures
\begin{equation}\label{eq:Bottcofibres}
\begin{tikzpicture}
\Vertex[x=-3,y=0,size=.05,color=white]{A0}
\Vertex[x=-3,y=1,size=.05,color=white]{A1}
\Vertex[x=-3,y=3,size=.05,color=white]{A3}
\Edge[lw=0.75pt,label={$2$},position=left](A0)(A1)
\Edge[lw=0.75pt,bend=50,label={$\eta$},position=left](A1)(A3)
\Vertex[x=0,y=0,size=.05,color=white]{B0}
\Vertex[x=0,y=1,size=.05,color=white]{B1}
\Vertex[x=0,y=2,size=.05,color=white]{B2}
\Vertex[x=0,y=3,size=.05,color=white]{B3}
\Vertex[x=0,y=4,size=.05,color=white]{B4}
\Edge[lw=0.75pt](B0)(B1)
\Edge[lw=0.75pt](B3)(B4)
\Edge[lw=0.75pt,bend=-50](B0)(B2)
\Edge[lw=0.75pt,bend=50](B1)(B3)
\Edge[lw=0.75pt,bend=-50](B2)(B4)
\Vertex[x=3,y=0,size=.05,color=white]{C0}
\Vertex[x=3,y=2,size=.05,color=white]{C2}
\Vertex[x=3,y=3,size=.05,color=white]{C3}
\Edge[lw=0.75pt,label={$2$},position=right](C2)(C3)
\Edge[lw=0.75pt,bend=-50,label={$\eta$},position=right](C0)(C2)
\end{tikzpicture}
\end{equation}
and their cohomologies are the following
cyclic $\StA(1)$-modules.
\begin{equation}\label{eq:Bottcofibres-H^*}
\begin{tikzpicture}
\Text[x=-9.2,y=2]{$\StA(1)/\StA(1)\{\Sq^2\}$}
\Vertex[x=-7,y=0,size=.05,color=black]{A0}
\Vertex[x=-7,y=1,size=.05,color=black]{A1}
\Vertex[x=-7,y=3,size=.05,color=black]{A3}
\Edge[lw=0.75pt](A0)(A1)
\Edge[lw=0.75pt,bend=50](A1)(A3)
\Text[x=-4.5,y=2]{$\StA(1)/\StA(1)\{\Sq^1\Sq^2\}$}
\Vertex[x=-2,y=0,size=.05,color=black]{B0}
\Vertex[x=-2,y=1,size=.05,color=black]{B1}
\Vertex[x=-2,y=2,size=.05,color=black]{B2}
\Vertex[x=-2,y=3,size=.05,color=black]{B3}
\Vertex[x=-2,y=4,size=.05,color=black]{B4}
\Edge[lw=0.75pt,label={$\Sq^1$},position=left](B0)(B1)
\Edge[lw=0.75pt](B3)(B4)
\Edge[lw=0.75pt,bend=-50,label={$\Sq^2$},position=right](B0)(B2)
\Edge[lw=0.75pt,bend=50](B1)(B3)
\Edge[lw=0.75pt,bend=-50](B2)(B4)
\Text[x=1.5,y=2]{$\StA(1)/\StA(1)\{\Sq^1,\Sq^2\Sq^1\Sq^2\}$}
\Vertex[x=4.3,y=0,size=.05,color=black]{C0}
\Vertex[x=4.3,y=2,size=.05,color=black]{C2}
\Vertex[x=4.3,y=3,size=.05,color=black]{C3}
\Edge[lw=0.75pt](C2)(C3)
\Edge[lw=0.75pt,bend=-50](C0)(C2)
\end{tikzpicture}
\end{equation}
In fact the whole sequence~\eqref{eq:BottCW}
can be made to be Spanier-Whitehead self dual
in the sense that applying $\Sigma^{12}D_{\kO}$
gives an equivalent sequence.

\subsection*{Realisation of cyclic $\StA(1)$-modules}
There are of course other cyclic quotients
of~$\StA(1)$ which can be realised as cohomology
of $\kO$-modules. These fall into two groups,
those containing a submodule isomorphic to
the quotient $\StA(1)/\!/\StA(0)$ and those
which do not. See Example~\ref{examp:H->kO}
for the first one, the others involving
$\StA(1)/\!/\StA(0)$ are easily realised.
The others can all be constructed as
$\kO$-modules of the form $\kO\wedge Z$.

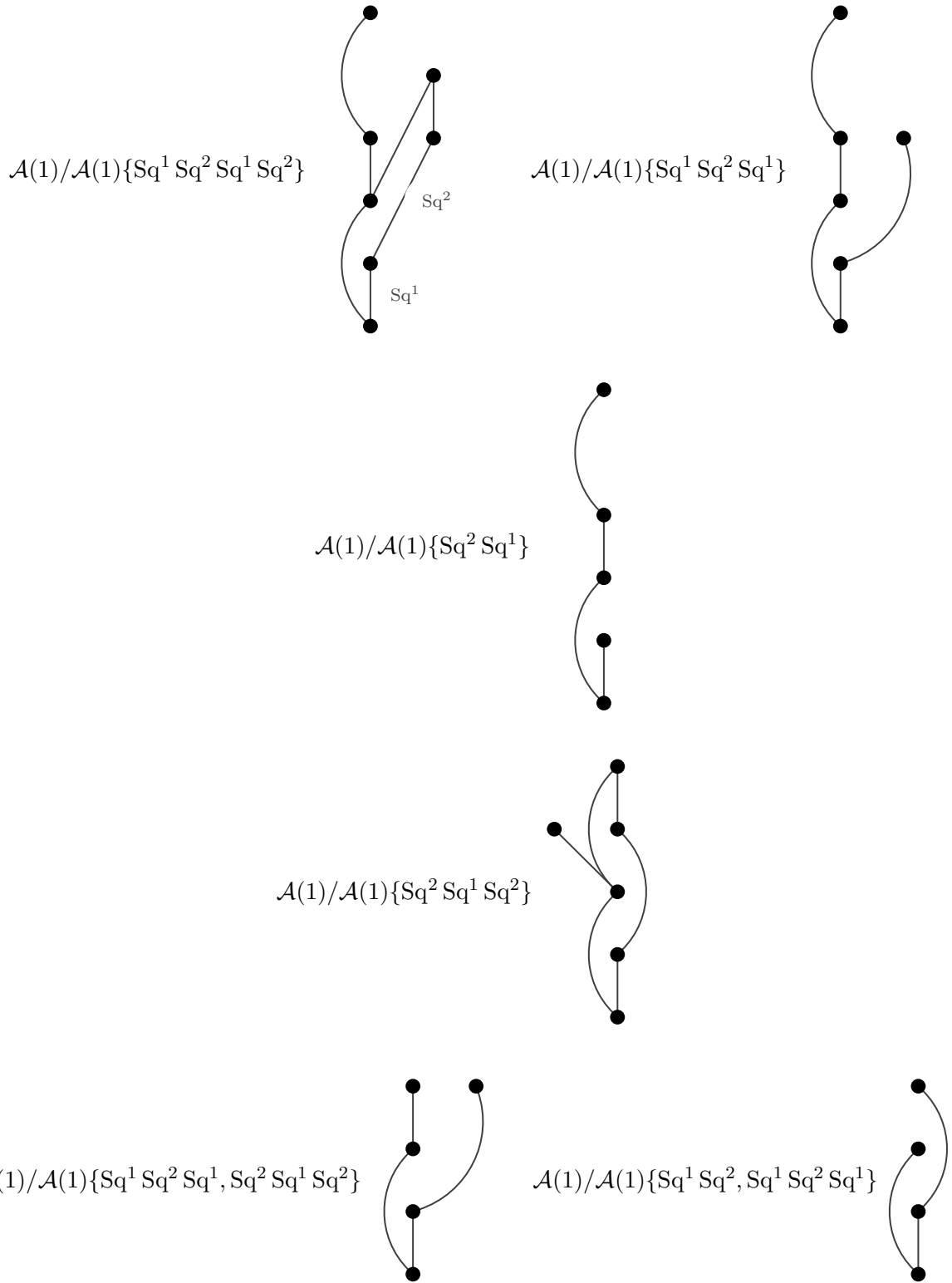
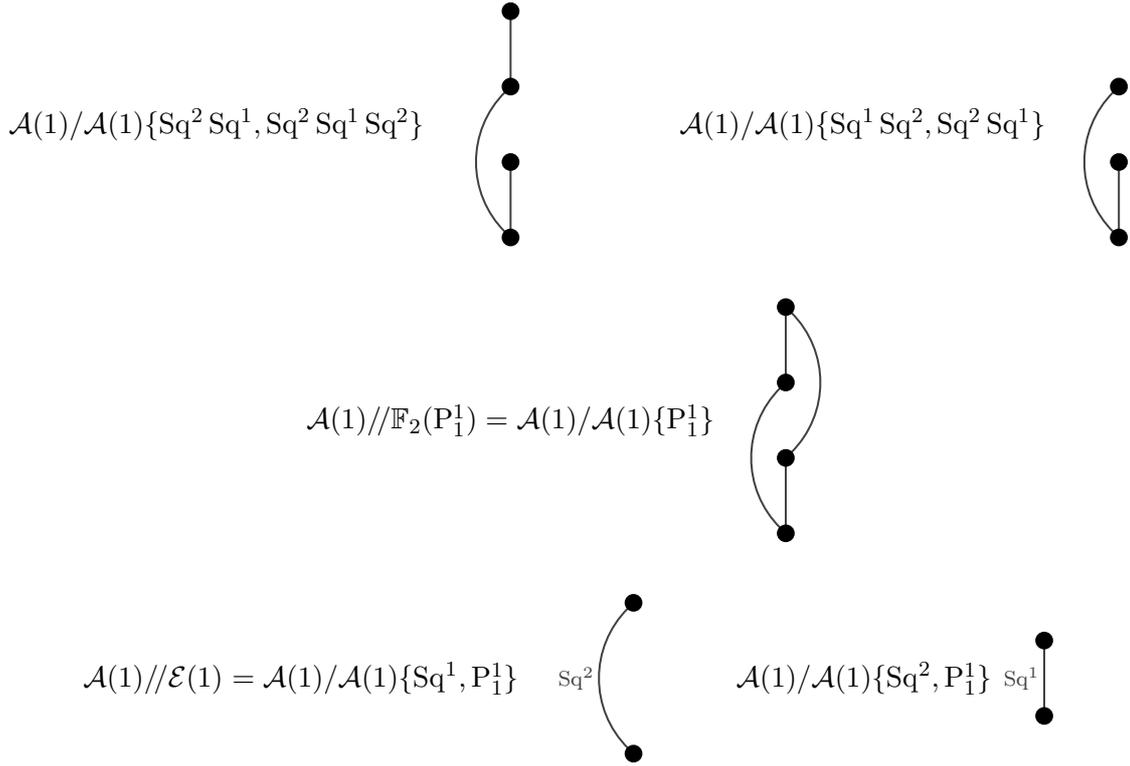
\bigskip
\begin{center}
\begin{tikzpicture}
\Text[x=-3.5,y=2.5]{$\StA(1)/\StA(1)\{\Sq^1\Sq^2\Sq^1\Sq^2\}$}
\Vertex[y=5,color=black,size=.1]{A5}
\Vertex[x=1,y=4,size=.1,color=black]{A4}
\Vertex[y=3,size=.1,color=black,label={3}]{A31}
\Vertex[x=1,y=3,size=.1,color=black]{A32}
\Vertex[y=2,size=.1,color=black]{A2}
\Vertex[y=1,size=.1,color=black]{A1}
\Vertex[y=0,size=.1,color=black]{A0}
\Edge[lw=0.75pt](A32)(A4)
\Edge[lw=0.75pt,bend=45](A31)(A5)
\Edge[lw=0.75pt](A2)(A31)
\Edge[lw=0.75pt](A2)(A4)
\Edge[lw=0.75pt,label={$\Sq^2$},position=right](A1)(A32)
\Edge[lw=0.75pt,label={$\Sq^1$},position=right](A0)(A1)
\Edge[lw=0.75pt,bend=45](A0)(A2)
\end{tikzpicture}
\qquad
\begin{tikzpicture}
\Text[x=-3.0,y=2.5]{$\StA(1)/\StA(1)\{\Sq^1\Sq^2\Sq^1\}$}
\Vertex[y=5,color=black,size=.1]{A5}
\Vertex[y=3,size=.1,color=black,label={3}]{A31}
\Vertex[x=1,y=3,size=.1,color=black]{A32}
\Vertex[y=2,size=.1,color=black]{A2}
\Vertex[y=1,size=.1,color=black]{A1}
\Vertex[y=0,size=.1,color=black]{A0}
\Edge[lw=0.75pt,bend=45](A31)(A5)
\Edge[lw=0.75pt](A2)(A31)
\Edge[lw=0.75pt,bend=-45](A1)(A32)
\Edge[lw=0.75pt](A0)(A1)
\Edge[lw=0.75pt,bend=45](A0)(A2)
\end{tikzpicture}
\end{center}

\bigskip
\begin{center}
\begin{tikzpicture}
\Text[x=-3.0,y=2.5]{$\StA(1)/\StA(1)\{\Sq^2\Sq^1\}$}
\Vertex[y=5,color=black,size=.1]{A5}
\Vertex[y=3,size=.1,color=black,label={3}]{A31}
\Vertex[y=2,size=.1,color=black]{A2}
\Vertex[y=1,size=.1,color=black]{A1}
\Vertex[y=0,size=.1,color=black]{A0}
\Edge[lw=0.75pt,bend=45](A31)(A5)
\Edge[lw=0.75pt](A2)(A31)
\Edge[lw=0.75pt](A0)(A1)
\Edge[lw=0.75pt,bend=45](A0)(A2)
\end{tikzpicture}
\end{center}

\bigskip
\begin{center}
\begin{tikzpicture}
\Text[x=-3.5,y=2]{$\StA(1)/\StA(1)\{\Sq^2\Sq^1\Sq^2\}$}
\Vertex[x=0,y=4,size=.1,color=black]{A4}
\Vertex[x=-1,y=3,size=.1,color=black]{A31}
\Vertex[x=0,y=3,size=.1,color=black]{A32}
\Vertex[y=2,size=.1,color=black]{A2}
\Vertex[y=1,size=.1,color=black]{A1}
\Vertex[y=0,size=.1,color=black]{A0}
\Edge[lw=0.75pt](A32)(A4)
\Edge[lw=0.75pt](A2)(A31)
\Edge[lw=0.75pt,bend=45](A2)(A4)
\Edge[lw=0.75pt,bend=-45](A1)(A32)
\Edge[lw=0.75pt](A0)(A1)
\Edge[lw=0.75pt,bend=45](A0)(A2)
\end{tikzpicture}
\end{center}
\bigskip
\begin{center}
\begin{tikzpicture}
\Text[x=-8.0,y=1.5]{$\StA(1)/\StA(1)\{\Sq^1\Sq^2\Sq^1,\Sq^2\Sq^1\Sq^2\}$}
\Vertex[x=-4,y=3,size=.1,color=black,label={3}]{A31}
\Vertex[x=-3,y=3,size=.1,color=black]{A32}
\Vertex[x=-4,y=2,size=.1,color=black]{A2}
\Vertex[x=-4,y=1,size=.1,color=black]{A1}
\Vertex[x=-4,y=0,size=.1,color=black]{A0}
\Edge[lw=0.75pt](A2)(A31)
\Edge[lw=0.75pt,bend=-45](A1)(A32)
\Edge[lw=0.75pt](A0)(A1)
\Edge[lw=0.75pt,bend=45](A0)(A2)
%\end{tikzpicture}
%\qquad
%\begin{tikzpicture}
\Text[x=0.5,y=1.5]{$\StA(1)/\StA(1)\{\Sq^1\Sq^2,\Sq^1\Sq^2\Sq^1\}$}
\Vertex[x=4,y=3,size=.1,color=black,label={3}]{B31}
\Vertex[x=4,y=2,size=.1,color=black]{B2}
\Vertex[x=4,y=1,size=.1,color=black]{B1}
\Vertex[x=4,y=0,size=.1,color=black]{B0}
\Edge[lw=0.75pt,bend=-45](B1)(B31)
\Edge[lw=0.75pt](B0)(B1)
\Edge[lw=0.75pt,bend=45](B0)(B2)
\end{tikzpicture}
\end{center}

\bigskip
\begin{center}
\begin{tikzpicture}
\Text[x=-8.0,y=1.5]{$\StA(1)/\StA(1)\{\Sq^2\Sq^1,\Sq^2\Sq^1\Sq^2\}$}
\Vertex[x=-4,y=3,size=.1,color=black,label={3}]{A31}
\Vertex[x=-4,y=2,size=.1,color=black]{A2}
\Vertex[x=-4,y=1,size=.1,color=black]{A1}
\Vertex[x=-4,y=0,size=.1,color=black]{A0}
\Edge[lw=0.75pt](A2)(A31)
\Edge[lw=0.75pt](A0)(A1)
\Edge[lw=0.75pt,bend=45](A0)(A2)
\Text[x=0.5,y=1.5]{$\StA(1)/\StA(1)\{\Sq^1\Sq^2,\Sq^2\Sq^1\}$}
\Vertex[x=4,y=2,size=.1,color=black]{B2}
\Vertex[x=4,y=1,size=.1,color=black]{B1}
\Vertex[x=4,y=0,size=.1,color=black]{B0}
\Edge[lw=0.75pt](B0)(B1)
\Edge[lw=0.75pt,bend=45](B0)(B2)
\end{tikzpicture}
\end{center}
\bigskip
\begin{center}
\begin{tikzpicture}
\Text[x=-3.75,y=2]{$\StA(1)/\!/\F_2(\mathrm{P}_1^1)
=\StA(1)/\StA(1)\{\mathrm{P}_1^1\}$}
\Vertex[x=0,y=3.5,size=.1,color=black]{A3}
\Vertex[y=2.5,size=.1,color=black]{A2}
\Vertex[y=1.5,size=.1,color=black]{A1}
\Vertex[y=0.5,size=.1,color=black]{A0}
\Edge[lw=0.75pt](A2)(A3)
\Edge[lw=0.75pt,bend=-45](A1)(A3)
\Edge[lw=0.75pt](A0)(A1)
\Edge[lw=0.75pt,bend=45](A0)(A2)
\end{tikzpicture}
\end{center}

\bigskip
\begin{center}
\begin{tikzpicture}
\Text[x=-4.5,y=1.0]{$\StA(1)/\!/\StE(1)
=\StA(1)/\StA(1)\{\Sq^1,\mathrm{P}_1^1\}$}
\Vertex[y=2.0,size=.1,color=black]{A2}
\Vertex[y=0,size=.1,color=black]{A0}
\Edge[lw=0.75pt,label={$\Sq^2$},position=left,bend=45](A0)(A2)
\end{tikzpicture}
\qquad
\begin{tikzpicture}
\Text[x=-2.5,y=1]{$\StA(1)/\StA(1)\{\Sq^2,\mathrm{P}_1^1\}$}
\Vertex[y=1.5,size=.1,color=black]{A1}
\Vertex[y=0.5,size=.1,color=black]{A0}
\Vertex[y=0,size=.1,Pseudo]{U}
\Edge[lw=0.75pt,label={$\Sq^1$},position=left](A0)(A1)
\end{tikzpicture}
\end{center}

Each of the last three has the form
$\StA(1)\otimes_{\mathcal{B}}\F_2$ for
some subalgebra $\mathcal{B}\subseteq\StA(1)$.
For each of the first two, $\mathcal{B}$
is a subHopf algebra which explains the
self duality. In the last case,
\[
\mathcal{B}=\F_2(\Sq^2,\mathrm{P}_1^1)
          =\F_2(\Sq^2,\Sq^1\Sq^2+\Sq^2\Sq^1)
\]
is a $6$-dimensional commutative subalgebra,
but even so $\StA(1)\otimes_{\mathcal{B}}\F_2$
is a stably self dual $\StA(1)$-module. Here
$\StA(1)$ is not a free right $\mathcal{B}$-module
since for example
\[
\Sq^1\mathrm{P}_1^1 = 1(\Sq^2\Sq^2).
\]

The self dual cyclic module
$\StA(1)/\!/\F_2(\mathrm{P}_1^1)
= \StA(1)/\StA(1)\{\mathrm{P}_1^1\}$ is a quotient
Hopf algebra since $\mathrm{P}_1^1$ is central
in~$\StA(1)$. This module can be realised, see
Example~\ref{examp:A1modP1} for details.

\subsection*{Constructions of some more
$\kO$-modules}

Collapsing $H\Z$ onto its top cell gives
a map $H\Z\to S_\kO^5$ and composing with
a map of degree~$2$ to the bottom cell of
$\Sigma^5H\Z$ gives a $\kO$-module map
$H\Z\to\Sigma^5H\Z$ whose mapping cone
has cohomology as shown.
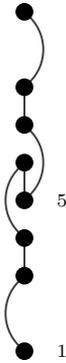
\begin{center}
\begin{tikzpicture}
\Vertex[x=1,y=5,size=.05,color=black]{B10}
\Vertex[x=1,y=4,size=.05,color=black]{B8}
\Vertex[x=1,y=3.5,size=.05,color=black]{B7}
\Vertex[x=1,y=3,size=.05,color=black]{B6}
\Vertex[x=1,y=2.5,size=.05,color=black]{B5}
\Vertex[x=1,y=2,size=.05,color=black]{B4}
\Vertex[x=1,y=1.5,size=.05,color=black]{B3}
\Vertex[x=1,y=0.5,size=.05,color=black]{B1}
\Text[x=1,y=2.5,position=right,distance=1mm]{\tiny$5$}
\Text[x=1,y=0.5,position=right,distance=1mm]{\tiny$1$}
\Edge[lw=0.75pt,bend=-45](B8)(B10)
\Edge[lw=0.75pt](B7)(B8)
\Edge[lw=0.75pt,bend=-45](B5)(B7)
\Edge[lw=0.75pt](B5)(B6)
\Edge[lw=0.75pt,bend=45](B4)(B6)
\Edge[lw=0.75pt](B3)(B4)
\Edge[lw=0.75pt,bend=45](B1)(B3)
\end{tikzpicture}
\end{center}

We can repeat this using a suitable
map from $\Sigma^{-4}H\Z$ to this by
using a degree~$2$ map to the bottom
cell and so on. Repeatedly using such
constructions and their Spanier-Whitehead
duals leads to a `daisy chain' of copies
of suspensions of $\StA(1)/\!/\StA(0)$
glued together by actions of~$\Sq^1$.
Such an $\StA(1)$-module is realised
as the cohomology of a $\kO$-module.
This process can either be stopped after
finitely many iterations or continued
indefinitely in either positive or
negative directions. Of course these
$\StA(1)$-modules are well known.

In fact $H\Z$ can itself be realised
in a similar way. The $\kO$-module
$\kU$ has cohomology
\begin{equation}\label{eq:H*kU}
\begin{tikzpicture}
\Text[x=-4.6,y=0.5]{$H_\kO^*\kU=\StA(1)/\!/\StE(1)
=\StA(1)/\StA(1)\{\Sq^1,\Sq^1\Sq^2\}$}
\Vertex[x=1,y=1,size=.05,color=black]{B2}
\Vertex[x=1,y=0,size=.05,color=black]{B0}
\Edge[lw=0.75pt,bend=45,label={$\Sq^2$},position=left](B0)(B2)
\end{tikzpicture}
\end{equation}
so $D_\kO\kU\sim\Sigma^{-2}\kU$. A similiar
construction to the above yields a complex
whose cohomology has the following form.
\begin{center}
\begin{tikzpicture}
\Vertex[x=1,y=1.5,size=.05,color=black]{B3}
\Vertex[x=1,y=1,size=.05,color=black]{B2}
\Vertex[x=1,y=0.5,size=.05,color=black]{B1}
\Vertex[x=1,y=0,size=.05,color=black]{B0}
\Edge[lw=0.75pt,bend=-45](B1)(B3)
\Edge[lw=0.75pt](B1)(B2)
\Edge[lw=0.75pt,bend=45](B0)(B2)
\end{tikzpicture}
\end{center}
By iterating we can obtain familiar $\StA(1)$-modules
such as
\begin{center}
\begin{tikzpicture}
\Vertex[x=1,y=3.5,size=.05,color=black]{B7}
\Vertex[x=1,y=3,size=.05,color=black]{B6}
\Vertex[x=1,y=2.5,size=.05,color=black]{B5}
\Vertex[x=1,y=2,size=.05,color=black]{B4}
\Vertex[x=1,y=1.5,size=.05,color=black]{B3}
\Vertex[x=1,y=1,size=.05,color=black]{B2}
\Vertex[x=1,y=0.5,size=.05,color=black]{B1}
\Vertex[x=1,y=0,size=.05,color=black]{B0}
\Edge[lw=0.75pt,bend=-45](B5)(B7)
\Edge[lw=0.75pt](B5)(B6)
\Edge[lw=0.75pt,bend=45](B4)(B6)
\Edge[lw=0.75pt](B3)(B4)
\Edge[lw=0.75pt,bend=-45](B1)(B3)
\Edge[lw=0.75pt](B1)(B2)
\Edge[lw=0.75pt,bend=45](B0)(B2)
\end{tikzpicture}
\end{center}
which can also be extended both upwards
and downwards.

\subsection*{Some sample calculations and examples}

\begin{examp}\label{examp:H->HZ}
Consider the Adams spectral sequence
\[
\mathrm{E}_2^{s,t}
= \Ext_{\StA(1)}^{s,t}(H^*_\kO H\Z,H^*_\kO H)
\Lra \mathscr{D}_\kO(H,H\Z)^{s-t}.
\]
Since
\[
H^*_\kO H \iso \StA(1) \iso \mathrm{D}(\StA(1))[6]
\]
we have
\begin{align*}
\mathrm{E}_2^{*,*}
&\iso
\Ext_{\StA(1)}^{*,*}(H^*_\kO H\Z,\mathrm{D}(\StA(1))[6]) \\
&\iso
\Ext_{\StA(1)}^{*,*}(\StA(1)\otimes H^*_\kO H\Z,\F_2[6]) \\
&\iso
\Hom^*(H^*_\kO H\Z,\F_2[6]) \\
&\iso
\Hom^6(H^*_\kO H\Z,\F_2[6])
\iso \Hom(H^*_\kO H\Z,\F_2)
\iso \F_2
\end{align*}
where the generator is the dual map to
the non-zero element in $H_\kO^0H\Z$.
So $\mathrm{D}_\kO(H,H\Z)^*\iso\F_2[0]$
and the generator corresponds to the
map $H\to H\Z$ which collapses~$H$ onto
its top cell composed with inclusion of
the bottom cell of~$\Sigma^6H\Z$.

By Spanier-Whitehead duality,
\[
\mathscr{D}_\kO(H\Z,H)^*
\iso \mathscr{D}_\kO(\Sigma^{-6}H,\Sigma^{-5}H\Z)^*
\iso \mathscr{D}_\kO(H,\Sigma H\Z)^*
\iso \F_2[5].
\]
This time the generator involves collapse
of $H\Z$ onto its top cell composed with
inclusion of the bottom cell of~$\Sigma^5H$,
i.e., a composition
\[
H\Z \to S^5_\kO\to\Sigma^5H.
\]

The reader may like to compare and
contrast this calculation with that
using the classical Adams spectral
sequence
\[
\mathrm{E}_2^{s,t}
= \Ext_{\StA}^{s,t}(H^*H\Z,H^*H)
\Lra \mathscr{D}_S(H,H\Z).
\]
%This time we have
%\[
%\Ext_{\StA}^{*,*}(H^*H\Z,H^*H)
%= \Ext_{\StA}^{*,*}(H^*H\Z,\StA)
%= \Hom^*(\StA\otimes_{\StA(0)}\F_2,\F_2).
%\]
%This time the lowest degree comes
%from the dual of $\Sq^4$ but there
%is also the dual of $\Sq^5$ which
%corresponds to the $\kO$-module
%generator.
\end{examp}

\begin{examp}\label{examp:H->kO}
Consider $\mathscr{D}_\kO(H,\kO)^*$.
By Spanier-Whitehead duality this
is isomorphic to
\[
\mathscr{D}_\kO(\kO,\Sigma^{-6}H)^*
\iso
\mathscr{D}_\kO(\kO,H)^{*-6}.
\]
The Adams spectral sequence
\[
\mathrm{E}_2^{s,t}
= \Ext_{\StA(1)}^{s,t}(H^*_\kO H,\F_2)
\Lra \mathscr{D}_\kO(H,\kO)^{s-t}
\]
has
\[
\Ext_{\StA(1)}^{*,*}(H^*_\kO H,\F_2)
\iso
\Ext_{\StA(1)}^{*,*}(\StA(1),\F_2)
= \Hom(\F_2,\F_2) = \F_2
\]
whose generator detects the inclusion
of the bottom cell of~$H$ whose
Spanier-Whitehead dual is collapse onto
the top cell of~$H$, i.e., a map
$H\to S_{\kO}^6\sim\Sigma^6\kO$. The
fibre of this realises the cyclic
module
\[
\mathrm{D}\mathcal{I}(1)[5]
= \StA(1)/\StA(1)\{\Sq^1\Sq^2\Sq^1\Sq^2\},
\]
a suspension of the dual of the counit,
see~\eqref{eq:I(1)&DI(1)}.
\end{examp}

\begin{examp}\label{examp:Joker->Joker}
Recall the Joker module $\mathcal{J}$
of~\eqref{eq:Joker}. In~\cite{AB:Joker}
we showed that there is a spectrum~$J$
for which the $\kO$-module $\kO\wedge J$
has cohomology is
\[
H_\kO^*(\kO\wedge J) \iso H^*(J)
\iso \mathcal{J}
\]
as an $\StA(1)$-module. Actually there
are two inequivalent such spectra which
are Spanier-Whitehead dual and their
cohomology realises the $\StA$-modules
with the two possible~$\Sq^4$ actions.
For our present purposes we may choose~$J'$
and~$J''$ to be either of them.

First we will determine
$\mathscr{D}_\kO(\kO\wedge J',\kO\wedge J'')$.
Using duality we have
\begin{align*}
\mathscr{D}_\kO(\kO\wedge J',\kO\wedge J'')^*
&\iso
\mathscr{D}_\kO(\kO,(\kO\wedge DJ')\wedge_\kO(\kO\wedge J''))^*  \\
&\iso
\mathscr{D}_\kO(\kO,\kO\wedge DJ'\wedge J'')^*.
\end{align*}
Now $\mathcal{J}$ is stably self dual
and in fact as $\StA(1)$-modules,
\begin{equation}\label{eq:J*J}
\mathcal{J}\otimes\mathcal{J} \iso
\F_2[0] \oplus \StA(1)[-4]\oplus \StA(1)[-3]\oplus \StA(1)[-2].
\end{equation}
There is an Adams spectral sequence
\[
\mathrm{E}_2^{s,t}
= \Ext_{\StA(1)}(H^*_\kO(\kO\wedge DJ'\wedge J''),\F_2)
\Lra \mathscr{D}_\kO(\kO,\kO\wedge DJ'\wedge J'')^*
\]
and its $\mathrm{E}_2$-term is given
by
 \begin{align*}
\mathrm{E}_2^{*,*} &\iso
\Ext_{\StA(1)}^{*,*}(\F_2,\F_2)
\oplus \Hom^*(\F_2[-4],\F_2)
\oplus \Hom^*(\F_2[-3],\F_2)
\oplus \Hom^*(\F_2[-2],\F_2) \\
&\iso
\Ext_{\StA(1)}^{*,*}(\F_2,\F_2)
\oplus\F_2[4]\oplus\F_2[3]\oplus\F_2[2].
\end{align*}
The reader is invited to describe
the stable maps $\kO\wedge J'\to\kO\wedge J''$
corresponding to $\F_2$-summands
in cellular terms. The generator of
$\Ext_{\StA(1)}^{0,0}(\F_2,\F_2)$
is an infinite cycle in the spectral
sequence showing there is indeed a
weak equivalence of $\kO$-modules
$\kO\wedge J'\to\kO\wedge J''$, hence
the choices of spectra $J'$ and $J''$
do not affect the $\kO$-module up to
homotopy equivalence and from now on
we just write $\kO\wedge J$ for any
such $\kO$-module.

Notice that
$\Ext_{\StA(1)}^{1,2}(\F_2,\F_2)\iso\F_2$,
and this corresponds to the map
$\kO\wedge\Sigma J \to \kO\wedge J$
induced by multiplication by $\eta$.
This shows that the cofibre sequence
\[
\kO\wedge J
\to \kO\wedge C_\eta\wedge J
\to \kO\wedge\Sigma^2 J
\]
cannot split since the short exact
sequence
\[
\xymatrix{
0 & \ar[l]H_\kO^*(\kO\wedge J)\ar@{<->}[d]^\iso
& \ar[l]H_\kO^*(\kO\wedge C_\eta\wedge J)\ar@{<->}[d]^\iso
& \ar[l]H_\kO^*(\kO\wedge\Sigma^2 J)\ar@{<->}[d]^\iso
& \ar[l]0   \\
& H^*(J)
& H^*(C_\eta\wedge J)
& H^*(\Sigma^2 J)
&
}
\]
represents the corresponding element of
$\Ext_{\StA(1)}^{1,2}(\mathcal{J},\mathcal{J})
\iso
\Ext_{\StA(1)}^{1,2}(\mathcal{J}\otimes\mathcal{J},\F_2)$.
Of course this also implies the well-known
fact that the cofibre sequence of spectra
\[
J\to C_\eta\wedge J\to \Sigma^2 J
\]
does not split.
\end{examp}

\begin{examp}\label{examp:HZHZ}
Consider the $\kO$-module $H\Z\wedge_{\kO}H\Z$.
The cohomology of this is
\[
H_\kO^*(H\Z\wedge_{\kO}H\Z)
\iso H_\kO^*H\Z\odot H_\kO^*H\Z
\]
where the factors are given in~\eqref{eq:HZ}.
Since~$H\Z$ is a unital $\kO$-algebra, $H\Z$
is a retract.

Notice that $H\Z$ and hence $H\Z\wedge_{\kO}H\Z$
are Spanier-Whitehead stably self dual with
\[
D_\kO H\Z \sim \Sigma^{-5}H\Z,
\quad
D_\kO(H\Z\wedge_{\kO}H\Z)
      \sim \Sigma^{-10}H\Z\wedge_{\kO}H\Z.
\]
Moreover, $H_\kO^*H\Z\odot H_\kO^*H\Z$
must have $\StA(1)$-module summands $H_\kO^*H\Z$
and $H_\kO^*H\Z[5]$; in fact a routine
calculation shows that
\[
H_\kO^*H\Z\odot H_\kO^*H\Z
\iso
H_\kO^*H\Z \oplus \StA(1)[2]
           \oplus H_\kO^*H\Z[5].
\]
Now it is straightforward to show that
\[
H\Z\wedge_{\kO}H\Z \sim
 H\Z\vee \Sigma^2H\vee \Sigma^5H\Z.
\]
In particular it follows that $H\Z^{\kO}_*H\Z$
is not a free $\Z_{(2)}$-module; it is
known that $H\Z_*H\Z$ also has simple
$2$-torsion, see~\cite{SOK:HZ*HZ}. Of
course the $\Z_{(2)}$-algebra structures
of $H\Z^{\kO}_*H\Z$ and $H\Z_{\kO}^*H\Z$
are both trivial for degree reasons.
\end{examp}

In the planned sequel we will consider
the $\tmf$-algebras $H\Z\wedge_\tmf H\Z$
and $\kO\wedge_\tmf\kO$.

\begin{examp}\label{examp:A1modP1}
There is a CW complex~$Z$ such that
\[
H_\kO^*(\kO\wedge Z)\iso H^*Z
\iso \StA(1)/\StA(1)\{\mathrm{P}_1^1\},
\]
here is a bare handed cellular construction
illustrating some of the tools available.

Consider the complex~$Z'$ for which
$\pi_2(Z')\iso\Z\oplus\Z/2$.
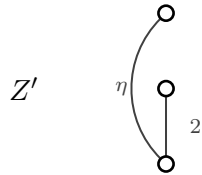
\begin{center}
\begin{tikzpicture}
\Text[x=-2.0,y=1.0]{$Z'$}
\Vertex[y=2,size=.1,color=white]{A2}
\Vertex[y=1,size=.1,color=white]{A1}
\Vertex[y=0,size=.1,color=white]{A0}
\Edge[lw=0.75pt,label={$2$},position=right](A0)(A1)
\Edge[lw=0.75pt,label={$\eta$},position=left,bend=45](A0)(A2)
\end{tikzpicture}
\end{center}
Then $Z$ is obtained by attaching a $3$-cell
to~$Z'$ using the sum of the generators
of the group $\pi_2(Z')\iso\Z\oplus\Z/2$
(see Figure~\ref{fig:Ext(Z')}).

There is a map $Z\to H$ which extends to
a $\kO$-module morphism $\kO\wedge Z\to H$
inducing an epimorphism of $\StA(1)$-modules
making the following diagram commute.
\[
\xymatrix{
H_\kO^*H\ar@{->>}[r]\ar@{<->}[d]_{\iso}
& H_\kO^*(\kO\wedge Z)\ar@{<->}[d]^{\iso}\ar@{<->}[r]^(.65){\iso}
& H^*Z \\
\StA(1)\ar@{.>>}[r]^(.3){\mathrm{quo}}
& \StA(1)/\StA(1)\{\mathrm{P}_1^1\} &
}
\]
The kernel of the quotient homomorphism is
isomorphic to $\StA(1)/\StA(1)\{\mathrm{P}_1^1\}[3]$
and there is a cofibre sequence of $\kO$-modules
\[
\kO\wedge Z \to H \to \kO\wedge\Sigma^3Z.
\]
We can also splice together infinitely
many copies of the short exact sequence
\[
\xymatrix{
0 & \ar[l] H^*Z & \ar[l]\StA(1) & \ar[l]\StA(1)/\StA(1)\{\mathrm{P}_1^1\}[3] & \ar[l]0
}
\]
to obtain
\[
\xymatrix{
0 & \ar[l] H^*Z & \ar[l]\StA(1)
& \ar[l]\StA(1)/\StA(1)[3]
& \ar[l]\StA(1)/\StA(1)[6]
& \ar[l]\cdots
}
\]
which is a periodic resolution of $H^*Z$ by
cyclic free $\StA(1)$-modules. This can be
used to determine the Adams $\mathrm{E}_2$-term
for computing $\pi_*(\kO\wedge Z)$, see
Figure~\ref{fig:A1modP1}. The homotopy groups
of $\kO\wedge Z$ are given by
\[
\pi_*(\kO\wedge Z) = \F_2[u_2]
\]
where $u_2\in\pi_2(\kO\wedge Z)$ is represented
in the Adams spectral sequence by the element
of $\Ext_{\StA(1)}^{1,3}(H^*Z,\F_2)$ corresponding
to the algebraic extension
\[
0\to \F_2[3] \to \StA(1) \to H^*Z \to 0
\]
with non-trivial $\mathrm{P}_1^1$ action.

\begin{figure}[ht]
\begin{tikzpicture}[scale=0.8]

\clip (-1.5,-0.5) rectangle (12.50, 8.50);
\draw[color=lightgray] (0,0) grid [step=4] (12.5,8.5);

\foreach \n in {0,4,...,20}
{
\def\nn{\n-0};
\node [below] at (\nn,0) {{\tiny$\n$}};
}

\foreach \s in {0,4,...,12}
{
\def\ss{\s-0};
\node [left] at (-0.1,\ss,0) {{\tiny$\s$}};
}

\draw [fill] ( 0.00, 0.00) circle [radius=0.05];
\node [left] at  ( 0.00, 0.00) {\ph{0}};

\draw [fill] ( 2.07, 0.93) circle [radius=0.05];
\node [left] at  ( 2.07, 0.93) {\ph{0}};

\draw [fill] ( 1.93, 1.07) circle [radius=0.05];
\node [left] at  ( 1.93, 1.07) {\ph{1}};

\draw [fill] ( 3.00, 1.00) circle [radius=0.05];
\node [left] at  ( 3.00, 1.00) {\ph{2}};

\draw [fill] ( 5.00, 1.00) circle [radius=0.05];
\node [left] at  ( 5.00, 1.00) {\ph{3}};

\draw [fill] ( 7.00, 1.00) circle [radius=0.05];
\node [left] at  ( 7.00, 1.00) {\ph{4}};

\draw [fill] (15.00, 1.00) circle [radius=0.05];
\node [left] at  (15.00, 1.00) {\ph{5}};

\draw [fill] ( 2.00, 2.00) circle [radius=0.05];
\node [left] at  ( 2.00, 2.00) {\ph{0}};

\draw [fill] ( 3.00, 2.00) circle [radius=0.05];
\node [left] at  ( 3.00, 2.00) {\ph{1}};

\draw [fill] ( 4.00, 2.00) circle [radius=0.05];
\node [left] at  ( 4.00, 2.00) {\ph{2}};

\draw [fill] ( 5.00, 2.00) circle [radius=0.05];
\node [left] at  ( 5.00, 2.00) {\ph{3}};

\draw [fill] ( 6.00, 2.00) circle [radius=0.05];
\node [left] at  ( 6.00, 2.00) {\ph{4}};

\draw [fill] ( 7.00, 2.00) circle [radius=0.05];
\node [left] at  ( 7.00, 2.00) {\ph{5}};

\draw [fill] ( 8.00, 2.00) circle [radius=0.05];
\node [left] at  ( 8.00, 2.00) {\ph{6}};

\draw [fill] ( 9.07, 1.93) circle [radius=0.05];
\node [left] at  ( 9.07, 1.93) {\ph{7}};

\draw [fill] ( 8.93, 2.07) circle [radius=0.05];
\node [left] at  ( 8.93, 2.07) {\ph{8}};

\draw [fill] (14.00, 2.00) circle [radius=0.05];
\node [left] at  (14.00, 2.00) {\ph{9}};

\draw [fill] (16.00, 2.00) circle [radius=0.05];
\node [left] at  (16.00, 2.00) {\ph{10}};

\draw [fill] (17.07, 1.93) circle [radius=0.05];
\node [left] at  (17.07, 1.93) {\ph{11}};

\draw [fill] (16.93, 2.07) circle [radius=0.05];
\node [left] at  (16.93, 2.07) {\ph{12}};

\draw [fill] (18.00, 2.00) circle [radius=0.05];
\node [left] at  (18.00, 2.00) {\ph{13}};

\draw [fill] (20.00, 2.00) circle [radius=0.05];
\node [left] at  (20.00, 2.00) {\ph{14}};

\draw [fill] (30.00, 2.00) circle [radius=0.05];
\node [left] at  (30.00, 2.00) {\ph{15}};

\draw [fill] ( 2.00, 3.00) circle [radius=0.05];
\node [left] at  ( 2.00, 3.00) {\ph{0}};

\draw [fill] ( 4.00, 3.00) circle [radius=0.05];
\node [left] at  ( 4.00, 3.00) {\ph{1}};

\draw [fill] ( 5.00, 3.00) circle [radius=0.05];
\node [left] at  ( 5.00, 3.00) {\ph{2}};

\draw [fill] ( 8.00, 3.00) circle [radius=0.05];
\node [left] at  ( 8.00, 3.00) {\ph{3}};

\draw [fill] ( 9.07, 2.93) circle [radius=0.05];
\node [left] at  ( 9.07, 2.93) {\ph{4}};

\draw [fill] ( 8.93, 3.07) circle [radius=0.05];
\node [left] at  ( 8.93, 3.07) {\ph{5}};

\draw [fill] (10.00, 3.00) circle [radius=0.05];
\node [left] at  (10.00, 3.00) {\ph{6}};

\draw [fill] (11.00, 3.00) circle [radius=0.05];
\node [left] at  (11.00, 3.00) {\ph{7}};

\draw [fill] (15.00, 3.00) circle [radius=0.05];
\node [left] at  (15.00, 3.00) {\ph{8}};

\draw [fill] (16.00, 3.00) circle [radius=0.05];
\node [left] at  (16.00, 3.00) {\ph{9}};

\draw [fill] (17.00, 3.00) circle [radius=0.05];
\node [left] at  (17.00, 3.00) {\ph{10}};

\draw [fill] (18.00, 3.00) circle [radius=0.05];
\node [left] at  (18.00, 3.00) {\ph{11}};

\draw [fill] (19.07, 2.93) circle [radius=0.05];
\node [left] at  (19.07, 2.93) {\ph{12}};

\draw [fill] (18.93, 3.07) circle [radius=0.05];
\node [left] at  (18.93, 3.07) {\ph{13}};

\draw [fill] (20.07, 2.93) circle [radius=0.05];
\node [left] at  (20.07, 2.93) {\ph{14}};

\draw [fill] (19.93, 3.07) circle [radius=0.05];
\node [left] at  (19.93, 3.07) {\ph{15}};

\draw [fill] (21.07, 2.93) circle [radius=0.05];
\node [left] at  (21.07, 2.93) {\ph{16}};

\draw [fill] (20.93, 3.07) circle [radius=0.05];
\node [left] at  (20.93, 3.07) {\ph{17}};

\draw [fill] (22.00, 3.00) circle [radius=0.05];
\node [left] at  (22.00, 3.00) {\ph{18}};

\draw [fill] (23.00, 3.00) circle [radius=0.05];
\node [left] at  (23.00, 3.00) {\ph{19}};

\draw [fill] ( 2.00, 4.00) circle [radius=0.05];
\node [left] at  ( 2.00, 4.00) {\ph{0}};

\draw [fill] ( 8.00, 4.00) circle [radius=0.05];
\node [left] at  ( 8.00, 4.00) {\ph{1}};

\draw [fill] ( 9.00, 4.00) circle [radius=0.05];
\node [left] at  ( 9.00, 4.00) {\ph{2}};

\draw [fill] (10.00, 4.00) circle [radius=0.05];
\node [left] at  (10.00, 4.00) {\ph{3}};

\draw [fill] (11.00, 4.00) circle [radius=0.05];
\node [left] at  (11.00, 4.00) {\ph{4}};

\draw [fill] (14.00, 4.00) circle [radius=0.05];
\node [left] at  (14.00, 4.00) {\ph{5}};

\draw [fill] (17.07, 3.93) circle [radius=0.05];
\node [left] at  (17.07, 3.93) {\ph{6}};

\draw [fill] (16.93, 4.07) circle [radius=0.05];
\node [left] at  (16.93, 4.07) {\ph{7}};

\draw [fill] (18.00, 4.00) circle [radius=0.05];
\node [left] at  (18.00, 4.00) {\ph{8}};

\draw [fill] (19.07, 3.93) circle [radius=0.05];
\node [left] at  (19.07, 3.93) {\ph{9}};

\draw [fill] (18.93, 4.07) circle [radius=0.05];
\node [left] at  (18.93, 4.07) {\ph{10}};

\draw [fill] (20.14, 3.86) circle [radius=0.05];
\node [left] at  (20.14, 3.86) {\ph{11}};

\draw [fill] (20.00, 4.00) circle [radius=0.05];
\node [left] at  (20.00, 4.00) {\ph{12}};

\draw [fill] (19.86, 4.14) circle [radius=0.05];
\node [left] at  (19.86, 4.14) {\ph{13}};

\draw [fill] (22.00, 4.00) circle [radius=0.05];
\node [left] at  (22.00, 4.00) {\ph{14}};

\draw [fill] (23.07, 3.93) circle [radius=0.05];
\node [left] at  (23.07, 3.93) {\ph{15}};

\draw [fill] (22.93, 4.07) circle [radius=0.05];
\node [left] at  (22.93, 4.07) {\ph{16}};

\draw [fill] (24.07, 3.93) circle [radius=0.05];
\node [left] at  (24.07, 3.93) {\ph{17}};

\draw [fill] (23.93, 4.07) circle [radius=0.05];
\node [left] at  (23.93, 4.07) {\ph{18}};

\draw [fill] ( 2.00, 5.00) circle [radius=0.05];
\node [left] at  ( 2.00, 5.00) {\ph{0}};

\draw [fill] ( 9.00, 5.00) circle [radius=0.05];
\node [left] at  ( 9.00, 5.00) {\ph{1}};

\draw [fill] (10.00, 5.00) circle [radius=0.05];
\node [left] at  (10.00, 5.00) {\ph{2}};

\draw [fill] (11.00, 5.00) circle [radius=0.05];
\node [left] at  (11.00, 5.00) {\ph{3}};

\draw [fill] (13.00, 5.00) circle [radius=0.05];
\node [left] at  (13.00, 5.00) {\ph{4}};

\draw [fill] (16.07, 4.93) circle [radius=0.05];
\node [left] at  (16.07, 4.93) {\ph{5}};

\draw [fill] (15.93, 5.07) circle [radius=0.05];
\node [left] at  (15.93, 5.07) {\ph{6}};

\draw [fill] (17.00, 5.00) circle [radius=0.05];
\node [left] at  (17.00, 5.00) {\ph{7}};

\draw [fill] (19.07, 4.93) circle [radius=0.05];
\node [left] at  (19.07, 4.93) {\ph{8}};

\draw [fill] (18.93, 5.07) circle [radius=0.05];
\node [left] at  (18.93, 5.07) {\ph{9}};

\draw [fill] (20.00, 5.00) circle [radius=0.05];
\node [left] at  (20.00, 5.00) {\ph{10}};

\draw [fill] (22.07, 4.93) circle [radius=0.05];
\node [left] at  (22.07, 4.93) {\ph{11}};

\draw [fill] (21.93, 5.07) circle [radius=0.05];
\node [left] at  (21.93, 5.07) {\ph{12}};

\draw [fill] (23.07, 4.93) circle [radius=0.05];
\node [left] at  (23.07, 4.93) {\ph{13}};

\draw [fill] (22.93, 5.07) circle [radius=0.05];
\node [left] at  (22.93, 5.07) {\ph{14}};

\draw [fill] (25.07, 4.93) circle [radius=0.05];
\node [left] at  (25.07, 4.93) {\ph{15}};

\draw [fill] (24.93, 5.07) circle [radius=0.05];
\node [left] at  (24.93, 5.07) {\ph{16}};

\draw [fill] (26.00, 5.00) circle [radius=0.05];
\node [left] at  (26.00, 5.00) {\ph{17}};

\draw [fill] ( 2.00, 6.00) circle [radius=0.05];
\node [left] at  ( 2.00, 6.00) {\ph{0}};

\draw [fill] (11.00, 6.00) circle [radius=0.05];
\node [left] at  (11.00, 6.00) {\ph{1}};

\draw [fill] (12.00, 6.00) circle [radius=0.05];
\node [left] at  (12.00, 6.00) {\ph{2}};

\draw [fill] (13.00, 6.00) circle [radius=0.05];
\node [left] at  (13.00, 6.00) {\ph{3}};

\draw [fill] (15.00, 6.00) circle [radius=0.05];
\node [left] at  (15.00, 6.00) {\ph{4}};

\draw [fill] (16.00, 6.00) circle [radius=0.05];
\node [left] at  (16.00, 6.00) {\ph{5}};

\draw [fill] (17.07, 5.93) circle [radius=0.05];
\node [left] at  (17.07, 5.93) {\ph{6}};

\draw [fill] (16.93, 6.07) circle [radius=0.05];
\node [left] at  (16.93, 6.07) {\ph{7}};

\draw [fill] (18.00, 6.00) circle [radius=0.05];
\node [left] at  (18.00, 6.00) {\ph{8}};

\draw [fill] (19.00, 6.00) circle [radius=0.05];
\node [left] at  (19.00, 6.00) {\ph{9}};

\draw [fill] (21.00, 6.00) circle [radius=0.05];
\node [left] at  (21.00, 6.00) {\ph{10}};

\draw [fill] (22.00, 6.00) circle [radius=0.05];
\node [left] at  (22.00, 6.00) {\ph{11}};

\draw [fill] (24.00, 6.00) circle [radius=0.05];
\node [left] at  (24.00, 6.00) {\ph{12}};

\draw [fill] (25.00, 6.00) circle [radius=0.05];
\node [left] at  (25.00, 6.00) {\ph{13}};

\draw [fill] (26.00, 6.00) circle [radius=0.05];
\node [left] at  (26.00, 6.00) {\ph{14}};

\draw [fill] (27.00, 6.00) circle [radius=0.05];
\node [left] at  (27.00, 6.00) {\ph{15}};

\draw [fill] (28.00, 6.00) circle [radius=0.05];
\node [left] at  (28.00, 6.00) {\ph{16}};

\draw [fill] (30.00, 6.00) circle [radius=0.05];
\node [left] at  (30.00, 6.00) {\ph{17}};

\draw [fill] ( 2.00, 7.00) circle [radius=0.05];
\node [left] at  ( 2.00, 7.00) {\ph{0}};

\draw [fill] (12.00, 7.00) circle [radius=0.05];
\node [left] at  (12.00, 7.00) {\ph{1}};

\draw [fill] (13.00, 7.00) circle [radius=0.05];
\node [left] at  (13.00, 7.00) {\ph{2}};

\draw [fill] (16.00, 7.00) circle [radius=0.05];
\node [left] at  (16.00, 7.00) {\ph{3}};

\draw [fill] (17.07, 6.93) circle [radius=0.05];
\node [left] at  (17.07, 6.93) {\ph{4}};

\draw [fill] (16.93, 7.07) circle [radius=0.05];
\node [left] at  (16.93, 7.07) {\ph{5}};

\draw [fill] (18.00, 7.00) circle [radius=0.05];
\node [left] at  (18.00, 7.00) {\ph{6}};

\draw [fill] (19.00, 7.00) circle [radius=0.05];
\node [left] at  (19.00, 7.00) {\ph{7}};

\draw [fill] (23.00, 7.00) circle [radius=0.05];
\node [left] at  (23.00, 7.00) {\ph{8}};

\draw [fill] (25.00, 7.00) circle [radius=0.05];
\node [left] at  (25.00, 7.00) {\ph{9}};

\draw [fill] (26.00, 7.00) circle [radius=0.05];
\node [left] at  (26.00, 7.00) {\ph{10}};

\draw [fill] (28.00, 7.00) circle [radius=0.05];
\node [left] at  (28.00, 7.00) {\ph{11}};

\draw [fill] (29.00, 7.00) circle [radius=0.05];
\node [left] at  (29.00, 7.00) {\ph{12}};

\draw [fill] ( 2.00, 8.00) circle [radius=0.05];
\node [left] at  ( 2.00, 8.00) {\ph{0}};

\draw [fill] (16.00, 8.00) circle [radius=0.05];
\node [left] at  (16.00, 8.00) {\ph{1}};

\draw [fill] (17.00, 8.00) circle [radius=0.05];
\node [left] at  (17.00, 8.00) {\ph{2}};

\draw [fill] (18.00, 8.00) circle [radius=0.05];
\node [left] at  (18.00, 8.00) {\ph{3}};

\draw [fill] (19.00, 8.00) circle [radius=0.05];
\node [left] at  (19.00, 8.00) {\ph{4}};

\draw [fill] (22.00, 8.00) circle [radius=0.05];
\node [left] at  (22.00, 8.00) {\ph{5}};

\draw [fill] (25.07, 7.93) circle [radius=0.05];
\node [left] at  (25.07, 7.93) {\ph{6}};

\draw [fill] (24.93, 8.07) circle [radius=0.05];
\node [left] at  (24.93, 8.07) {\ph{7}};

\draw [fill] (27.00, 8.00) circle [radius=0.05];
\node [left] at  (27.00, 8.00) {\ph{8}};

\draw [fill] (28.07, 7.93) circle [radius=0.05];
\node [left] at  (28.07, 7.93) {\ph{9}};

\draw [fill] (27.93, 8.07) circle [radius=0.05];
\node [left] at  (27.93, 8.07) {\ph{10}};

\draw [fill] (30.00, 8.00) circle [radius=0.05];
\node [left] at  (30.00, 8.00) {\ph{11}};

\draw [fill] ( 2.00, 9.00) circle [radius=0.05];
\node [left] at  ( 2.00, 9.00) {\ph{0}};

\draw [fill] (17.00, 9.00) circle [radius=0.05];
\node [left] at  (17.00, 9.00) {\ph{1}};

\draw [fill] (18.00, 9.00) circle [radius=0.05];
\node [left] at  (18.00, 9.00) {\ph{2}};

\draw [fill] (19.00, 9.00) circle [radius=0.05];
\node [left] at  (19.00, 9.00) {\ph{3}};

\draw [fill] (21.00, 9.00) circle [radius=0.05];
\node [left] at  (21.00, 9.00) {\ph{4}};

\draw [fill] (24.07, 8.93) circle [radius=0.05];
\node [left] at  (24.07, 8.93) {\ph{5}};

\draw [fill] (23.93, 9.07) circle [radius=0.05];
\node [left] at  (23.93, 9.07) {\ph{6}};

\draw [fill] (25.00, 9.00) circle [radius=0.05];
\node [left] at  (25.00, 9.00) {\ph{7}};

\draw [fill] (27.07, 8.93) circle [radius=0.05];
\node [left] at  (27.07, 8.93) {\ph{8}};

\draw [fill] (26.93, 9.07) circle [radius=0.05];
\node [left] at  (26.93, 9.07) {\ph{9}};

\draw [fill] (28.00, 9.00) circle [radius=0.05];
\node [left] at  (28.00, 9.00) {\ph{10}};

\draw [fill] (30.07, 8.93) circle [radius=0.05];
\node [left] at  (30.07, 8.93) {\ph{11}};

\draw [fill] (29.93, 9.07) circle [radius=0.05];
\node [left] at  (29.93, 9.07) {\ph{12}};

\draw [fill] ( 2.00,10.00) circle [radius=0.05];
\node [left] at  ( 2.00,10.00) {\ph{0}};

\draw [fill] (19.00,10.00) circle [radius=0.05];
\node [left] at  (19.00,10.00) {\ph{1}};

\draw [fill] (20.00,10.00) circle [radius=0.05];
\node [left] at  (20.00,10.00) {\ph{2}};

\draw [fill] (21.00,10.00) circle [radius=0.05];
\node [left] at  (21.00,10.00) {\ph{3}};

\draw [fill] (23.00,10.00) circle [radius=0.05];
\node [left] at  (23.00,10.00) {\ph{4}};

\draw [fill] (24.00,10.00) circle [radius=0.05];
\node [left] at  (24.00,10.00) {\ph{5}};

\draw [fill] (25.07, 9.93) circle [radius=0.05];
\node [left] at  (25.07, 9.93) {\ph{6}};

\draw [fill] (24.93,10.07) circle [radius=0.05];
\node [left] at  (24.93,10.07) {\ph{7}};

\draw [fill] (26.00,10.00) circle [radius=0.05];
\node [left] at  (26.00,10.00) {\ph{8}};

\draw [fill] (27.00,10.00) circle [radius=0.05];
\node [left] at  (27.00,10.00) {\ph{9}};

\draw [fill] (29.00,10.00) circle [radius=0.05];
\node [left] at  (29.00,10.00) {\ph{10}};

\draw [fill] (30.00,10.00) circle [radius=0.05];
\node [left] at  (30.00,10.00) {\ph{11}};

\draw [fill] ( 2.00,11.00) circle [radius=0.05];
\node [left] at  ( 2.00,11.00) {\ph{0}};

\draw [fill] (20.00,11.00) circle [radius=0.05];
\node [left] at  (20.00,11.00) {\ph{1}};

\draw [fill] (21.00,11.00) circle [radius=0.05];
\node [left] at  (21.00,11.00) {\ph{2}};

\draw [fill] (24.00,11.00) circle [radius=0.05];
\node [left] at  (24.00,11.00) {\ph{3}};

\draw [fill] (25.07,10.93) circle [radius=0.05];
\node [left] at  (25.07,10.93) {\ph{4}};

\draw [fill] (24.93,11.07) circle [radius=0.05];
\node [left] at  (24.93,11.07) {\ph{5}};

\draw [fill] (26.00,11.00) circle [radius=0.05];
\node [left] at  (26.00,11.00) {\ph{6}};

\draw [fill] (27.00,11.00) circle [radius=0.05];
\node [left] at  (27.00,11.00) {\ph{7}};

\draw [fill] ( 2.00,12.00) circle [radius=0.05];
\node [left] at  ( 2.00,12.00) {\ph{0}};

\draw [fill] (24.00,12.00) circle [radius=0.05];
\node [left] at  (24.00,12.00) {\ph{1}};

\draw [fill] (25.00,12.00) circle [radius=0.05];
\node [left] at  (25.00,12.00) {\ph{2}};

\draw [fill] (26.00,12.00) circle [radius=0.05];
\node [left] at  (26.00,12.00) {\ph{3}};

\draw [fill] (27.00,12.00) circle [radius=0.05];
\node [left] at  (27.00,12.00) {\ph{4}};

\draw [fill] (30.00,12.00) circle [radius=0.05];
\node [left] at  (30.00,12.00) {\ph{5}};

\draw [fill] ( 2.00,13.00) circle [radius=0.05];
\node [left] at  ( 2.00,13.00) {\ph{0}};

\draw [fill] (25.00,13.00) circle [radius=0.05];
\node [left] at  (25.00,13.00) {\ph{1}};

\draw [fill] (26.00,13.00) circle [radius=0.05];
\node [left] at  (26.00,13.00) {\ph{2}};

\draw [fill] (27.00,13.00) circle [radius=0.05];
\node [left] at  (27.00,13.00) {\ph{3}};

\draw [fill] (29.00,13.00) circle [radius=0.05];
\node [left] at  (29.00,13.00) {\ph{4}};

\draw [fill] ( 2.00,14.00) circle [radius=0.05];
\node [left] at  ( 2.00,14.00) {\ph{0}};

\draw [fill] (27.00,14.00) circle [radius=0.05];
\node [left] at  (27.00,14.00) {\ph{1}};

\draw [fill] (28.00,14.00) circle [radius=0.05];
\node [left] at  (28.00,14.00) {\ph{2}};

\draw [fill] (29.00,14.00) circle [radius=0.05];
\node [left] at  (29.00,14.00) {\ph{3}};

\draw [fill] ( 2.00,15.00) circle [radius=0.05];
\node [left] at  ( 2.00,15.00) {\ph{0}};

\draw [fill] (28.00,15.00) circle [radius=0.05];
\node [left] at  (28.00,15.00) {\ph{1}};

\draw [fill] (29.00,15.00) circle [radius=0.05];
\node [left] at  (29.00,15.00) {\ph{2}};

\draw [fill] ( 2.00,16.00) circle [radius=0.05];
\node [left] at  ( 2.00,16.00) {\ph{0}};

\draw [fill] ( 2.00,17.00) circle [radius=0.05];
\node [left] at  ( 2.00,17.00) {\ph{0}};

\draw [fill] ( 2.00,18.00) circle [radius=0.05];
\node [left] at  ( 2.00,18.00) {\ph{0}};

\draw [fill] ( 2.00,19.00) circle [radius=0.05];
\node [left] at  ( 2.00,19.00) {\ph{0}};

\draw [fill] ( 2.00,20.00) circle [radius=0.05];
\node [left] at  ( 2.00,20.00) {\ph{0}};

\draw [dashed]  ( 3.00, 1.00) --( 0.00, 0.00);
\draw ( 2.00, 2.00) --( 2.07, 0.93);
\draw ( 3.00, 2.00) --( 2.07, 0.93);
\draw ( 3.00, 2.00) --( 1.93, 1.07);
\draw [dashed]  ( 5.00, 2.00) --( 2.07, 0.93);
\draw ( 5.00, 2.00) --( 5.00, 1.00);
\draw [dashed]  ( 6.00, 2.00) --( 3.00, 1.00);
\draw ( 6.00, 2.00) --( 5.00, 1.00);
\draw [dashed]  ( 8.00, 2.00) --( 5.00, 1.00);
\draw [dashed]  (18.00, 2.00) --(15.00, 1.00);
\draw ( 2.00, 3.00) --( 2.00, 2.00);
\draw ( 4.00, 3.00) --( 3.00, 2.00);
\draw ( 4.00, 3.00) --( 4.00, 2.00);
\draw [dashed]  ( 5.00, 3.00) --( 2.00, 2.00);
\draw ( 5.00, 3.00) --( 4.00, 2.00);
\draw ( 5.00, 3.00) --( 5.00, 2.00);
\draw [dashed]  ( 8.00, 3.00) --( 5.00, 2.00);
\draw ( 8.00, 3.00) --( 7.00, 2.00);
\draw ( 8.00, 3.00) --( 8.00, 2.00);
\draw ( 8.93, 3.07) --( 9.07, 1.93);
\draw [dashed]  (10.00, 3.00) --( 7.00, 2.00);
\draw (10.00, 3.00) --( 9.07, 1.93);
\draw (10.00, 3.00) --( 8.93, 2.07);
\draw [dashed]  (11.00, 3.00) --( 8.00, 2.00);
\draw (16.00, 3.00) --(16.00, 2.00);
\draw (17.00, 3.00) --(17.07, 1.93);
\draw (18.00, 3.00) --(17.07, 1.93);
\draw (18.00, 3.00) --(16.93, 2.07);
\draw [dashed]  (19.07, 2.93) --(16.00, 2.00);
\draw [dashed]  (19.93, 3.07) --(17.07, 1.93);
\draw (19.93, 3.07) --(20.00, 2.00);
\draw [dashed]  (20.93, 3.07) --(18.00, 2.00);
\draw (20.93, 3.07) --(20.00, 2.00);
\draw [dashed]  (23.00, 3.00) --(20.00, 2.00);
\draw ( 2.00, 4.00) --( 2.00, 3.00);
\draw ( 9.00, 4.00) --( 8.93, 3.07);
\draw (10.00, 4.00) --( 9.07, 2.93);
\draw (16.93, 4.07) --(17.00, 3.00);
\draw [dashed]  (18.00, 4.00) --(15.00, 3.00);
\draw (18.93, 4.07) --(18.00, 3.00);
\draw (18.93, 4.07) --(18.93, 3.07);
\draw [dashed]  (19.86, 4.14) --(17.00, 3.00);
\draw (19.86, 4.14) --(18.93, 3.07);
\draw (19.86, 4.14) --(19.93, 3.07);
\draw [dashed]  (22.00, 4.00) --(19.07, 2.93);
\draw (22.00, 4.00) --(21.07, 2.93);
\draw [dashed]  (23.07, 3.93) --(20.07, 2.93);
\draw [dashed]  (22.93, 4.07) --(20.07, 2.93);
\draw [dashed]  (22.93, 4.07) --(19.93, 3.07);
\draw (22.93, 4.07) --(22.00, 3.00);
\draw (22.93, 4.07) --(23.00, 3.00);
\draw [dashed]  (24.07, 3.93) --(21.07, 2.93);
\draw ( 2.00, 5.00) --( 2.00, 4.00);
\draw ( 9.00, 5.00) --( 8.00, 4.00);
\draw ( 9.00, 5.00) --( 9.00, 4.00);
\draw [dashed]  (11.00, 5.00) --( 8.00, 4.00);
\draw (11.00, 5.00) --(10.00, 4.00);
\draw (11.00, 5.00) --(11.00, 4.00);
\draw (17.00, 5.00) --(16.93, 4.07);
\draw (18.93, 5.07) --(19.07, 3.93);
\draw (20.00, 5.00) --(19.07, 3.93);
\draw (20.00, 5.00) --(20.00, 4.00);
\draw [dashed]  (22.07, 4.93) --(19.07, 3.93);
\draw [dashed]  (23.07, 4.93) --(20.00, 4.00);
\draw [dashed]  (22.93, 5.07) --(20.14, 3.86);
\draw (24.93, 5.07) --(23.93, 4.07);
\draw ( 2.00, 6.00) --( 2.00, 5.00);
\draw (11.00, 6.00) --(10.00, 5.00);
\draw (13.00, 6.00) --(13.00, 5.00);
\draw [dashed]  (16.00, 6.00) --(13.00, 5.00);
\draw (16.00, 6.00) --(16.07, 4.93);
\draw (17.07, 5.93) --(16.07, 4.93);
\draw (17.07, 5.93) --(15.93, 5.07);
\draw (16.93, 6.07) --(17.00, 5.00);
\draw [dashed]  (19.00, 6.00) --(16.07, 4.93);
\draw (19.00, 6.00) --(18.93, 5.07);
\draw [dashed]  (22.00, 6.00) --(18.93, 5.07);
\draw (22.00, 6.00) --(22.07, 4.93);
\draw [dashed]  (25.00, 6.00) --(22.07, 4.93);
\draw (25.00, 6.00) --(25.07, 4.93);
\draw [dashed]  (26.00, 6.00) --(22.93, 5.07);
\draw (26.00, 6.00) --(25.07, 4.93);
\draw (26.00, 6.00) --(24.93, 5.07);
\draw (26.00, 6.00) --(26.00, 5.00);
\draw [dashed]  (28.00, 6.00) --(25.07, 4.93);
\draw ( 2.00, 7.00) --( 2.00, 6.00);
\draw (12.00, 7.00) --(11.00, 6.00);
\draw (12.00, 7.00) --(12.00, 6.00);
\draw (13.00, 7.00) --(12.00, 6.00);
\draw (13.00, 7.00) --(13.00, 6.00);
\draw [dashed]  (16.00, 7.00) --(13.00, 6.00);
\draw (16.00, 7.00) --(15.00, 6.00);
\draw (16.00, 7.00) --(16.00, 6.00);
\draw (16.93, 7.07) --(16.93, 6.07);
\draw [dashed]  (18.00, 7.00) --(15.00, 6.00);
\draw (18.00, 7.00) --(17.07, 5.93);
\draw (18.00, 7.00) --(18.00, 6.00);
\draw [dashed]  (19.00, 7.00) --(16.00, 6.00);
\draw (19.00, 7.00) --(18.00, 6.00);
\draw (19.00, 7.00) --(19.00, 6.00);
\draw ( 2.00, 8.00) --( 2.00, 7.00);
\draw (17.00, 8.00) --(16.93, 7.07);
\draw (18.00, 8.00) --(17.07, 6.93);
\draw (24.93, 8.07) --(25.00, 7.00);
\draw [dashed]  (27.93, 8.07) --(25.00, 7.00);
\draw (27.93, 8.07) --(28.00, 7.00);
\draw [dashed]  (30.93, 8.07) --(28.00, 7.00);
\draw ( 2.00, 9.00) --( 2.00, 8.00);
\draw (17.00, 9.00) --(16.00, 8.00);
\draw (17.00, 9.00) --(17.00, 8.00);
\draw [dashed]  (19.00, 9.00) --(16.00, 8.00);
\draw (19.00, 9.00) --(18.00, 8.00);
\draw (19.00, 9.00) --(19.00, 8.00);
\draw (25.00, 9.00) --(24.93, 8.07);
\draw (26.93, 9.07) --(27.00, 8.00);
\draw [dashed]  (28.00, 9.00) --(24.93, 8.07);
\draw (28.00, 9.00) --(27.00, 8.00);
\draw (28.00, 9.00) --(27.93, 8.07);
\draw [dashed]  (29.93, 9.07) --(27.00, 8.00);
\draw (29.93, 9.07) --(30.00, 8.00);
\draw [dashed]  (31.00, 9.00) --(27.93, 8.07);
\draw (31.00, 9.00) --(30.00, 8.00);
\draw [dashed]  (33.00, 9.00) --(30.00, 8.00);
\draw ( 2.00,10.00) --( 2.00, 9.00);
\draw (19.00,10.00) --(18.00, 9.00);
\draw (21.00,10.00) --(21.00, 9.00);
\draw [dashed]  (24.00,10.00) --(21.00, 9.00);
\draw (24.00,10.00) --(24.07, 8.93);
\draw (25.07, 9.93) --(24.07, 8.93);
\draw (25.07, 9.93) --(23.93, 9.07);
\draw (24.93,10.07) --(25.00, 9.00);
\draw [dashed]  (27.00,10.00) --(24.07, 8.93);
\draw (27.00,10.00) --(26.93, 9.07);
\draw [dashed]  (30.00,10.00) --(26.93, 9.07);
\draw (30.00,10.00) --(29.93, 9.07);
\draw [dashed]  (33.07, 9.93) --(29.93, 9.07);
\draw ( 2.00,11.00) --( 2.00,10.00);
\draw (20.00,11.00) --(19.00,10.00);
\draw (20.00,11.00) --(20.00,10.00);
\draw (21.00,11.00) --(20.00,10.00);
\draw (21.00,11.00) --(21.00,10.00);
\draw [dashed]  (24.00,11.00) --(21.00,10.00);
\draw (24.00,11.00) --(23.00,10.00);
\draw (24.00,11.00) --(24.00,10.00);
\draw (24.93,11.07) --(24.93,10.07);
\draw [dashed]  (26.00,11.00) --(23.00,10.00);
\draw (26.00,11.00) --(25.07, 9.93);
\draw (26.00,11.00) --(26.00,10.00);
\draw [dashed]  (27.00,11.00) --(24.00,10.00);
\draw (27.00,11.00) --(26.00,10.00);
\draw (27.00,11.00) --(27.00,10.00);
\draw ( 2.00,12.00) --( 2.00,11.00);
\draw (25.00,12.00) --(24.93,11.07);
\draw (26.00,12.00) --(25.07,10.93);
\draw ( 2.00,13.00) --( 2.00,12.00);
\draw (25.00,13.00) --(24.00,12.00);
\draw (25.00,13.00) --(25.00,12.00);
\draw [dashed]  (27.00,13.00) --(24.00,12.00);
\draw (27.00,13.00) --(26.00,12.00);
\draw (27.00,13.00) --(27.00,12.00);
\draw ( 2.00,14.00) --( 2.00,13.00);
\draw (27.00,14.00) --(26.00,13.00);
\draw (29.00,14.00) --(29.00,13.00);
\draw [dashed]  (32.00,14.00) --(29.00,13.00);
\draw ( 2.00,15.00) --( 2.00,14.00);
\draw (28.00,15.00) --(27.00,14.00);
\draw (28.00,15.00) --(28.00,14.00);
\draw (29.00,15.00) --(28.00,14.00);
\draw (29.00,15.00) --(29.00,14.00);
\draw [dashed]  (32.00,15.00) --(29.00,14.00);
\draw ( 2.00,16.00) --( 2.00,15.00);
\draw ( 2.00,17.00) --( 2.00,16.00);
\draw ( 2.00,18.00) --( 2.00,17.00);
\draw ( 2.00,19.00) --( 2.00,18.00);
\draw ( 2.00,20.00) --( 2.00,19.00);
\draw ( 2.00,21.00) --( 2.00,20.00);

\end{tikzpicture}
\caption{$\Ext_{\StA}^{s,t}(H^*Z',\F_2)$ for
$0\leqslant s\leqslant 8$,
$0 \leqslant t-s \leqslant 12$}
\label{fig:Ext(Z')}
\end{figure}

\begin{figure}[ht]
\begin{tikzpicture}[scale=0.8]

\clip (-1.5,-0.5) rectangle (16.50, 8.50);
\draw[color=lightgray] (0,0) grid [step=4] (16,8);

\foreach \n in {0,4,...,16}
{
\def\nn{\n-0};
\node [below] at (\nn,0) {{\tiny$\n$}};
}

\foreach \s in {0,4,...,8}
{
\def\ss{\s-0};
\node [left] at (-0.1,\ss,0) {{\tiny$\s$}};
}

\draw [fill] ( 0.00, 0.00) circle [radius=0.05];
\node [left] at  ( 0.00, 0.00) {\ph{0}};

\draw [fill] ( 2.00, 1.00) circle [radius=0.05];
\node [left] at  ( 2.00, 1.00) {\ph{0}};

\draw [fill] ( 4.00, 2.00) circle [radius=0.05];
\node [left] at  ( 4.00, 2.00) {\ph{0}};

\draw [fill] ( 6.00, 3.00) circle [radius=0.05];
\node [left] at  ( 6.00, 3.00) {\ph{0}};

\draw [fill] ( 8.00, 4.00) circle [radius=0.05];
\node [left] at  ( 8.00, 4.00) {\ph{0}};

\draw [fill] (10.00, 5.00) circle [radius=0.05];
\node [left] at  (10.00, 5.00) {\ph{0}};

\draw [fill] (12.00, 6.00) circle [radius=0.05];
\node [left] at  (12.00, 6.00) {\ph{0}};

\draw [fill] (14.00, 7.00) circle [radius=0.05];
\node [left] at  (14.00, 7.00) {\ph{0}};

\draw [fill] (16.00, 8.00) circle [radius=0.05];
\node [left] at  (16.00, 8.00) {\ph{0}};

\draw [fill] (18.00, 9.00) circle [radius=0.05];
\node [left] at  (18.00, 9.00) {\ph{0}};

\draw [fill] (20.00,10.00) circle [radius=0.05];
\node [left] at  (20.00,10.00) {\ph{0}};

\draw [fill] (22.00,11.00) circle [radius=0.05];
\node [left] at  (22.00,11.00) {\ph{0}};

\draw [fill] (24.00,12.00) circle [radius=0.05];
\node [left] at  (24.00,12.00) {\ph{0}};

\draw [fill] (26.00,13.00) circle [radius=0.05];
\node [left] at  (26.00,13.00) {\ph{0}};

\draw [fill] (28.00,14.00) circle [radius=0.05];
\node [left] at  (28.00,14.00) {\ph{0}};

\draw [fill] (30.00,15.00) circle [radius=0.05];
\node [left] at  (30.00,15.00) {\ph{0}};

\end{tikzpicture}
\caption{$\Ext_{\StA(1)}^{s,t}(H^*Z,\F_2)$
for $0\leqslant s\leqslant 12$,
$0 \leqslant t-s \leqslant 20$}
\label{fig:A1modP1}
\end{figure}

Given two realisations of
$\StA(1)/\!/\F_2(\mathrm{P}_1^1)$
by $\kO$-modules $W_1,W_2$, consider
the Adams spectral sequence
\[
\mathrm{E}_2^{s,t} =
\Ext_{\StA(1)}^{s,t}(\StA(1)/\!/\F_2(\mathrm{P}_1^1),\StA(1)/\!/\F_2(\mathrm{P}_1^1))
\Lra \mathscr{D}_{\kO}(W_1,W_2).
\]
Then by Proposition~\ref{prop:ExtH//K},
%\begin{align*}
%\mathrm{E}_2^{s,t}
%&\iso
%\Ext_{\StA(1)}^{s,t}(\StA(1)/\!/\F_2(\mathrm{P}_1^1)\odot\StA(1)/\!/\F_2(\mathrm{P}_1^1)[-3],\F_2) \\
%&\iso
%\Ext_{\StA(1)}^{s,t}(\StA(1)/\!/\F_2(\mathrm{P}_1^1)\odot\StA(1)/\!/\F_2(\mathrm{P}_1^1),\F_2[3]) \\
%&\iso
%\Ext_{\StA(1)}^{s,t}(\StA(1)\otimes_{\F_2(\mathrm{P}_1^1)}\StA(1)/\!/\F_2(\mathrm{P}_1^1),\F_2[3]) \\
%&\iso
%\Ext_{\F_2(\mathrm{P}_1^1)}^{s,t}(\StA(1)/\!/\F_2(\mathrm{P}_1^1),\F_2[3]).
%\end{align*}
\[
\mathrm{E}_2^{s,t}
\iso
\Ext_{\F_2(\mathrm{P}_1^1)}^{s,t}(\StA(1)/\!/\F_2(\mathrm{P}_1^1),\F_2[3]).
\]
The $\F_2(\mathrm{P}_1^1)$-module structure
of $\StA(1)/\!/\F_2(\mathrm{P}_1^1)$ has a
non-trivial $\mathrm{P}_1^1$-action linking
the generators in degrees~$0$ and~$3$.
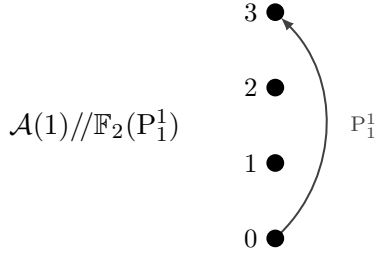
\begin{center}
\begin{tikzpicture}
\Text[x=-2.5,y=2]{$\StA(1)/\!/\F_2(\mathrm{P}_1^1)$}
\Vertex[x=0,y=3.5,size=.1,color=black,label={\small$3$},position=left]{A3}
\Vertex[y=2.5,size=.1,color=black,label={\small$2$},position=left]{A2}
\Vertex[y=1.5,size=.1,color=black,label={\small$1$},position=left]{A1}
\Vertex[y=0.5,size=.1,color=black,label={\small$0$},position=left]{A0}
\Edge[lw=0.75pt,label={$\mathrm{P}_1^1$},position=right,bend=-45,Direct](A0)(A3)
\end{tikzpicture}
\end{center}
Therefore
\[
\StA(1)/\!/\F_2(\mathrm{P}_1^1)
\iso
\F_2(\mathrm{P}_1^1)\oplus\F_2[1]\oplus\F_2[2]
\]
and since
\[
\Ext_{\F_2(\mathrm{P}_1^1)}^{*,*}(\F_2,\F_2)
= \F_2[w]
\]
with $w$ in bidegree $(1,3)$, we have
\[
\mathrm{E}_2^{*,*} \iso
\F_2[-3] \oplus \F_2[w][-2] \oplus \F_2[w][-1].
\]
There can be no non-trivial Adams differentials,
in particular, the generator of $\F_2[-3]$ which
corresponds to the identity homomorphism can be
realised by a weak equivalence $W_1\to W_2$ of
$\kO$-modules. This shows that this $\kO$-module
is well defined up to weak equivalence and also
stably self dual.

The $\StA(1)$-module obtained by inducing up the
$\F_2(\mathrm{P}_1^1)$-module above has the form
\[
\StA(1)\otimes_{\F_2(\mathrm{P}_1^1)}\StA(1)/\!/\F_2(\mathrm{P}_1^1)
\iso
\StA(1)
\oplus \StA(1)/\!/\F_2(\mathrm{P}_1^1)[1]
\oplus \StA(1)/\!/\F_2(\mathrm{P}_1^1)[2],
\]
and this is isomorphic to $H^*(Z\wedge Z)$.

\noindent
Remark: The reader may have spotted that
the smash product $C_2\wedge C_\eta$ is
also a model for~$Z$ which is stably self
dual since $C_2$ and $C_\eta$ are. However,
we feel it worthwhile building it explicitly
to demonstrate the techniques available.
Of course we know that the actual construction
of~$Z$ is irrelevant since it is well defined
up to weak equivalence.
\end{examp}

\section{Examples based on homogeneous spaces}\label{sec:HomogenSpces}
In this section we consider some examples
built using homogeneous spaces, our goal
is to identify $\kO$-orientable manifolds
so we discuss how to determine when they
admit $\Spin$ structures. The reader will
find many interesting examples in the paper
of Douglas, Henriques and
Hill~\cite{CLD-AGH-MAH:HomObstrStrOtns},
and we will give some others. We refer
the reader to the work of Singhof \&
Wemmer~\cites{WS:ParallelHomogSpces-I,WS&WD:ParallelHomogSpces-II,WS&WD:ParallelHomogSpces-IIErratum}
for detailed results on stable frameability
of homogeneous spaces.

We were inspired to consider the question of
orientability for homogeneous spaces by work
of Atiyah \&
Smith~\cite{MFA-LS:CptLieGpsStHtpySph}*{lemma~3.1}
which gave a `local' consequence of the existence
of tangential~$\Spin$ structures for homogeneous
spaces~$G/T$ where~$T\leq G$ is a maximal torus;
this is based on a condition on the sum of the
positive roots.

\begin{prop}\label{prop:HomoGenSpce-TgtBdle}
Let $G$ be a compact connected Lie group and
$H\leq G$ a closed subgroup. Then the tangent
bundle of~$G/H$ is\/
$\mathrm{T}G/H=G\times_H\mathfrak{g}/\mathfrak{h}$
where~$H$ acts on $\mathfrak{g}/\mathfrak{h}$
by the adjoint representation.
\end{prop}
\begin{proof}
This is well-known, for example see Brockett
\& Sussman~\cite{RWB&HJS:TgtBdlesHomogSpces}
or Singhof \& Wemmer
\cites{WS:ParallelHomogSpces-I,WS&WD:ParallelHomogSpces-II}.
\end{proof}

The vector bundle $G\times_H\mathfrak{g}\to G/H$
admits a trivialisation
\begin{equation}\label{eq:AdjTriv}
\xymatrix{
G\times_H\mathfrak{g}\ar[rr]^\iso\ar[dr]
&&
G/H\times\mathfrak{g}\ar[dl] \\
& G/H &  \\
[g,x]_H \ar@{<->}[rr] && (gH,\Adj(g)(x))
}
\end{equation}
which is $G$-equivariant where $G/H\times\mathfrak{g}$
is given the diagonal left $G$-action.

By choosing an $H$-invariant inner product
we can find a splitting of the adjoint
representation of~$H$ on $\mathfrak{g}$,
\[
\mathfrak{g}
\iso
\mathfrak{g}/\mathfrak{h}\oplus\mathfrak{h},
\]
and this induces a $G$-vector bundle isomorphism
\[
\Adj_{G/H}\mathfrak{g}
\iso
\Adj_{G/H}\mathfrak{g}/\mathfrak{h}\oplus\Adj_{G/H}\mathfrak{h},
\]
where
\[
\Adj_{G/H}(-) = G\times_H (-)\to G/H
\]
denotes the associated vector bundle construction.
Now combining this with Proposition~\ref{prop:HomoGenSpce-TgtBdle}
and~\eqref{eq:AdjTriv}, we obtain some useful
results.

\begin{prop}\label{prop:HomoGenSpce-TgtBdle-2}
Let $G$ be a compact connected Lie group and
$H\leq G$ a closed subgroup. Then there is
an isomorphism of $G$-vector bundles
\[
\mathrm{T}G/H \oplus \mathrm{Adj}_{G/H}\mathfrak{h}
\iso G/H\times\mathfrak{g}.
\]
Working non-equivariantly, the right hand
bundle is trivial, so $\mathrm{T}G/H$ and
$\mathrm{Adj}_{G/H}\mathfrak{h}$ are
stably inverse bundles.
\end{prop}
\begin{cor}\label{cor:HomoGenSpce-TgtBdle}
For a commutative ring spectrum~$E$,~$G/H$
is $E$-orientable if and only if\/
$\mathrm{Adj}_{G/H}\mathfrak{h}$ is $E$-orientable.
\end{cor}

Thus the $E$-orientability of $G/H$ can be
investigated either by considering the adjoint
action of~$H$ on $\mathfrak{g}/\mathfrak{h}$
in terms of the action of~$H$ on roots of~$G$
outside of $\mathfrak{h}$, or the adjoint
representation of~$H$ on $\mathfrak{h}$.
In particular, for $\kO$-orientability
this reduces to checking whether for
a chosen maximal torus of~$H$, the Weyl
vector
\[
\rho_H =
\frac{1}{2}(\text{sum of positive roots of~$H$})
\]
is a weight (this is always true when~$H$
is simply connected). It is known that
$\rho_H$ is the sum of the fundamental
weights of~$H$.

Before giving some examples, we recall that
for a compact simply connected Lie group~$G$
with chosen maximal torus~$T$, and finite
centre $\mathrm{Z}(G)$, there are some important
relationships between the root lattice
$\Lambda_{\mathrm{rt}}(G)$, the weight lattice
$\Lambda_{\mathrm{wt}}(G)$, the centre and
the fundamental group~$\pi_1(G/\mathrm{Z}(G))$:
\begin{equation}\label{eq:relationships-1}
\Lambda_{\mathrm{wt}}(G)/\Lambda_{\mathrm{rt}}(G)
\iso
\mathrm{Z}(G)
\iso
\pi_1(G/\mathrm{Z}(G));
\end{equation}
furthermore,
\begin{equation}\label{eq:relationships-2}
\Lambda_{\mathrm{wt}}(G/\mathrm{Z}(G))
=
\Lambda_{\mathrm{rt}}(G/\mathrm{Z}(G)).
\end{equation}
So for example when $n\geq3$, this gives
\[
\Lambda_{\mathrm{wt}}(\Spin(n))/\Lambda_{\mathrm{rt}}(\Spin(n))
\iso \Z/2,
\quad
\Lambda_{\mathrm{wt}}(\SO(n))
= \Lambda_{\mathrm{rt}}(\SO(n)).
\]
It is also known that in general,
\begin{equation}\label{eq:relationships-3}
|\pi_1(G/\mathrm{Z}(G))|
=
|\mathrm{Z}(G)|
=
|\Lambda_{\mathrm{wt}}(G):\Lambda_{\mathrm{rt}}(G)|
= \text{determinant of the Cartan matrix of $G$}.
\end{equation}

Finally we recall that if $T\leq H\leq G$
where~$T$ is a maximal torus of~$G$ then the
cohomology of~$G/H$ is concentrated in even
degrees so all odd degree elements of the
Steenrod algebra act trivially.

Now we discuss some examples where we can
determine whether $\Spin$ structures exist
on the adjoint representations of some
standard Lie groups.

\begin{examp}\label{examp:U(n)}
Let $n\geq1$. Then the adjoint representation
of $\mathrm{U}(n)$ on $\mathfrak{u}(n)$ (the
space of $n\times n$ skew-Hermitian matrices)
lifts to a $\Spin$ representation if and only
if~$n$ is odd.

To see this, recall that the diagonal matrices
$\diag(z_1,\ldots,z_n)\in\mathrm{U}(n)$ form
a maximal torus $\mathrm{T}^n\leq\mathrm{U}(n)$.
Let $E^{rs}$ be the matrix with all entries~$0$
except~$1$ in the $(r,s)$ place. For each pair
$r,s$ with $1\leq r<s\leq n-1$, consider the
adjoint action on the matrix~$E^{rs}$:
\[
\diag(z_1,\ldots,z_n)E^{rs}\diag(z_1,\ldots,z_n)^{-1}
= z_rE^{rs}z_s^{-1}
\]
It follows that the roots are the
$\omega_r-\omega_s$ where $\omega_k$ is the
fundamental weight given by the projection
of~$T^n$ onto the $k$-th factor. The sum of
the positive roots is
\begin{align*}
2\rho_{\mathrm{U}(n)} =
\sum_{1\leq r<s\leq n}(\omega_r-\omega_s)
&=
\sum_{1\leq r\leq n-1} (n-r)\omega_r
-\sum_{2\leq r\leq n}(r-1)\omega_r \\
&=
-(n-1)\omega_n +
\sum_{1\leq r\leq n-1} (n-2r+1)\omega_r,
\end{align*}
which is divisible by~$2$ in the weight
lattice if and only if~$n$ is odd. It
follows that if~$G$ is a compact
connected Lie group which contains
a subgroup isomorphic to $\mathrm{U}(n)$,
~$G/\mathrm{U}(n)$ admits a~$\Spin$
structure if and only if~$n$ is odd.
\end{examp}

\begin{examp}\label{examp:G_2}
The root system for the rank~$2$ group~$\mathrm{G}_2$
is shown in Figure~\ref{fig:G2}.
\begin{figure}[ht]
\[
\begin{xy}
(0,0)="O",
(28.2842712474,0)="A1",
(-28.2842712474,0)="-A1",
(-42.4264068711,24.4948974278)="A2",
(42.4264068711,-24.4948974278)="-A2",
(-14.1421356237,24.4948974278)="A1+A2",
(14.1421356237,-24.4948974278)="-A1-A2",
(14.1421356237,24.4948974278)="2A1+A2",
(-14.1421356237,-24.4948974278)="-2A1-A2",
(42.4264068711,24.4948974278)="3A1+A2",
(-42.4264068711,-24.4948974278)="-3A1-A2",
(0,48.9897948556)="3A1+2A2",
(0,-48.9897948556)="-3A1-2A2",
\ar^>>{\ds\alpha_1} "O";"A1"
\ar@{^{*}-->}_>>{\ds-\alpha_1} "O";"-A1"
\ar_>>{\ds\alpha_2} "O";"A2"
\ar@{-->}^>>{\ds-\alpha_2} "O";"-A2"
\ar@{-->}^>>{\ds\alpha_1+\alpha_2} "O";"A1+A2"
\ar@{-->}^>>{\ds-\alpha_1-\alpha_2} "O";"-A1-A2"
\ar@{-->}_>>{\ds2\alpha_1+\alpha_2} "O";"2A1+A2"
\ar@{-->}_>>{\ds-2\alpha_1-\alpha_2} "O";"-2A1-A2"
\ar@{-->}_>>{\ds\quad3\alpha_1+\alpha_2} "O";"3A1+A2"
\ar@{-->}_>>{\ds-3\alpha_1-\alpha_2\quad} "O";"-3A1-A2"
\ar@{-->}_>>{\ds3\alpha_1+2\alpha_2} "O";"3A1+2A2"
\ar@{-->}^>>{\ds-3\alpha_1-2\alpha_2} "O";"-3A1-2A2"
\end{xy}
\]
\caption{The roots of $\mathrm{G}_2$}
\label{fig:G2}
\end{figure}
The Cartan matrix has determinant~$1$, so the
weight and root lattices agree and the centre
is trivial. One half of the sum of the positive
roots is
\[
\rho_{\mathrm{G}_2}
= \frac{1}{2}\bigl(10\alpha_1+6\alpha_2\bigr)
=  5\alpha_1+3\alpha_2,
\]
which is a weight and therefore the adjoint
representation on the Lie algebra $\mathfrak{g}_2$
lifts to~$\Spin$. So for any compact connected
Lie group~$G$ which contains a subgroup isomorphic
to $\mathrm{G}_2$,~$G/\mathrm{G}_2$ admits
a~$\Spin$ structure. One example of this is
$\mathrm{F}_4$, so the homogeneous space
$\mathrm{F}_4/\mathrm{G}_2$ of dimension~$36$
has a $\Spin$ structure. We also know that
for any maximal torus $\mathrm{T}^2\leq\mathrm{G}_2$,
$\mathrm{G}_2/\mathrm{T}^2$ has a~$\Spin$
structure.
See~\cite{CLD-AGH-MAH:HomObstrStrOtns}*{figure~4}
for the cohomology of these examples as
$\StA$-modules.
\end{examp}

\begin{examp}\label{examp:A_n}
For $n\geq1$ the Lie group $\SU(n+1)$
corresponds to the root system~$\mathrm{A}_n$
and $\mathrm{Z}(\SU(n+1))$ is cyclic of
order~$n+1$.

The positive roots have the form
$\alpha_r+\alpha_{r+1}+\cdots+\alpha_s$
($1\leq r\leq s\leq n$) and half the sum
of these is
\begin{equation}\label{eq:A_n-posrootsum}
\rho_{\SU(n+1)} =
\sum_{1\leq k\leq n}\frac{k(n+1-k)}{2}\alpha_k.
\end{equation}
When $n$ is even, $k$ and $n+1-k$ have opposite
parity, so the coefficients are all integers
so this is in the weight lattice
of~$\SU(n+1)/\mathrm{Z}(\SU(n+1))$. When~$n$
is odd, for each odd value of~$k$ the coefficient
of $\alpha_k$ is odd, so this is not in the
weight lattice of~$\SU(n+1)/\mathrm{Z}(\SU(n+1))$.
So the adjoint representation of
$\SU(n+1)/\mathrm{Z}(\SU(n+1))$ admits
a~$\Spin$ structure if and only if~$n$
is even.
\end{examp}

\begin{examp}\label{examp:B_n}
For $n\geq2$, $\mathrm{B}_n$ is the root
system of $\Spin(2n+1)$ and the centre
has order~$2$ where
$\Spin(2n+1)/\mathrm{Z}(\Spin(n+1))\iso\SO(2n+1)$.
It turns out that in the root lattice the
sum of the positive roots satisfies
\[
\rho_{\Spin(2n+1)} \equiv
\sum_{1\leq k\leq n}k\alpha_k
\mod{2},
\]
which implies that the adjoint representation
of~$\SO(2n+1)$ does not have a~$\Spin$
structure.
\end{examp}
\begin{examp}\label{examp:CnDn}
For $n\geq2$, $\mathrm{C}_n$ is the root
system of $\Sp(n)$ and the centre has
order~$2$. The adjoint representation
of~$\Sp(n)$ has a~$\Spin$ structure
if and only if $n\equiv0,2\bmod{4}$.

For $n\geq4$, $\mathrm{D}_n$ is the root
system of $\Spin(2n)$ and the centre has
order~$4$. The adjoint representation
of~$\Spin(2n)$ has a~$\Spin$ structure
if and only if $n\equiv0,1\bmod{4}$.
\end{examp}
%\begin{examp}\label{examp:C_n}
%For $n\geq2$, $\mathrm{C}_n$ is the root
%system of $\Sp(n)$ and the centre has
%order~$2$.
%
%?????????????????
%It turns out that in the root lattice
%the sum of the positive roots is
%congruent to
%\[
%\sum_{1\leq k\leq n}k\alpha_k
%\bmod{2},
%\]
%which implies that the adjoint representation
%of~$\SO(2n+1)$ does not have a~$\Spin$
%structure.
%\end{examp}
%\begin{examp}\label{examp:D_n}
%For $n\geq4$, $\mathrm{D}_n$ is the root
%system of $\Spin(2n)$ and the centre has
%order~$4$.
%
%?????????????????
%It turns out that in the root lattice
%the sum of the positive roots is
%congruent to
%\[
%\sum_{1\leq k\leq n}k\alpha_k
%\bmod{2},
%\]
%which implies that the adjoint representation
%of~$\SO(2n+1)$ does not have a~$\Spin$
%structure.
%\end{examp}

\begin{examp}\label{examp:F_4}
The exceptional Lie group $\mathrm{F}_4$ has rank~$4$
and its Cartan matrix has determinant~$1$ so the root
and weight lattices coincide. One half of the sum of
the positive roots is the weight
\[
\rho_{\mathrm{F}_4} =
8\alpha_{1} + 15\alpha_{2} + 21\alpha_{3} + 11\alpha_{4},
\]
so the adjoint representation on the Lie algebra
$\mathfrak{f}_4$ has a~$\Spin$ structure.
\end{examp}

\begin{examp}\label{examp:E_?}
For the exceptional Lie groups $\mathrm{E}_6$,
$\mathrm{E}_7$ and $\mathrm{E}_8$, the Cartan
matrices have determinants~$3$,~$2$ and~$1$
respectively and these are the orders of their
centres. Here are the formulae for half the
sums of the positive roots:
\begin{align*}
\rho_{\mathrm{E}_6} &=
8\alpha_{1} + 11\alpha_{2} + 15\alpha_{3}
+ 21\alpha_{4} + 15\alpha_{5} + 8\alpha_{6}, \\
\rho_{\mathrm{E}_7} &=
17\alpha_{1} + \frac{49}{2}\alpha_{2}
+ 33\alpha_{3} + 48\alpha_{4}
+ \frac{75}{2}\alpha_{5} + 26\alpha_{6}
+ \frac{27}{2}\alpha_{7}, \\
\rho_{\mathrm{E}_8} &=
46\alpha_{1} + 68\alpha_{2} + 91\alpha_{3}
+ 135\alpha_{4} + 110\alpha_{5}
+ 84\alpha_{6} + 57\alpha_{7} + 29\alpha_{8}.
\end{align*}
For $\mathrm{E}_6/\mathrm{Z}(\mathrm{E}_6)$
and $\mathrm{E}_8/\mathrm{Z}(\mathrm{E}_8)$,
the adjoint representations admit~$\Spin$
structures, but
for~$\mathrm{E}_7/\mathrm{Z}(\mathrm{E}_7)$
the expression is not a weight so the adjoint
representation does not admit a~$\Spin$
structure.
\end{examp}

\begin{examp}\label{examp:SO(8)/Sp(2)}
Consider the $18$-dimensional homogeneous
space~$\SO(8)/\Sp(2)$. The cohomology is
easily determined using for example the
Eilenberg-Moore spectral sequence
\[
E^2_{s,t} = \Tor_{s,t}^{H^*(B\SO(8))}(H^*(B\Sp(4))),\F_2)
\Lra H^{t-s}(\SO(8)/\Sp(2)).
\]
Here the inclusion $\Sp(2)\hookrightarrow\SO(8)$
induces the algebra homomorphism
\[
\xymatrix{
H^*(B\SO(8))\ar[rr]\ar@{=}[d] && H^*(B\Sp(4))\ar@{=}[d] \\
\F_2[w_2,w_3,w_4,w_5,w_6,w_7,w_8] &w_{4k}\mapsto\wp_{k}& \F_2[\wp_1,\wp_2] \\
%w_{4k}\ar@{|->}[rr]&& \wp_{k}
}
\]
where $\wp_k\in H^{4k}(B\Sp(4))$ is the $k$-th
symplectic Pontrjagin class. The $E^2$-term
is the exterior algebra
\[
\mathrm{E}^2_{*,*} = \Lambda(U_1,U_2,U_4,U_5,U_6)
\]
where $U_k$ has bidegree $(1,k+1)$ and is
represented in the bar construction by
the class~$[w_{k+1}]$. This spectral
sequence has Steenrod operations and
using the action on the Stiefel-Whitney
classes we find that
\[
\Sq^1U_1=U_2,\;\Sq^2U_2=U_4.
\]
The spectral sequence collapses and so
$H^*(\SO(8)/\Sp(2))$ is generated as
an algebra by the elements $u_k$
represented by~$U_k$, where
\[
u_2=u_1^2,\;u_3=u_1^3,\;u_4=u_1^4.
\]
Therefore it has a basis consisting of
the monomials
\[
u_1^au_5^bu_6^c
\quad
(0\leq a\leq 7,\;b,c\in\{0,1\})
\]
and the self dual $\StA(1)$-module structure
shown below. The $\kO$-module
$\kO\wedge(\SO(8)/\Sp(2)_+)$ is Spanier-Whitehead
stably self dual and
$H_{\kO}^*(\kO\wedge(\SO(8)/\Sp(2)_+))$ agrees
with this~$\StA(1)$-module.
\begin{center}
\begin{tikzpicture}
\Vertex[x=0,y=3.5,size=.05,color=black,position=left,distance=2mm,label={\tiny$u_1^7$}]{A7}
\Vertex[x=0,y=3.0,size=.05,color=black,position=left,distance=2mm,label={\tiny$u_1^6$}]{A6}
\Vertex[x=0,y=2.5,size=.05,color=black,position=left,distance=2mm,label={\tiny$u_1^5$}]{A5}
\Vertex[x=0,y=2.0,size=.05,color=black,position=left,distance=2mm,label={\tiny$u_1^4$}]{A4}
\Vertex[x=0,y=1.5,size=.05,color=black,position=left,distance=2mm,label={\tiny$u_1^3$}]{A3}
\Vertex[x=0,y=1.0,size=.05,color=black,position=left,distance=2mm,label={\tiny$u_1^2$}]{A2}
\Vertex[x=0,y=0.5,size=.05,color=black,position=left,distance=2mm,label={\tiny$u_1$}]{A1}
\Vertex[x=0,y=0.0,size=.05,color=black,position=left,distance=2mm,label={\tiny$1$}]{A0}
\Edge[lw=0.75pt](A5)(A6)
\Edge[lw=0.75pt,bend=45](A3)(A5)
\Edge[lw=0.75pt](A3)(A4)
\Edge[lw=0.75pt,bend=-45](A2)(A4)
\Edge[lw=0.75pt](A1)(A2)
\Vertex[x=2.5,y=6.0,size=.05,color=black,position=left,distance=2mm,label={\tiny$u_1^7u_5$}]{B7}
\Vertex[x=2.5,y=5.5,size=.05,color=black,position=left,distance=2mm,label={\tiny$u_1^6u_5$}]{B6}
\Vertex[x=2.5,y=5.0,size=.05,color=black,position=left,distance=2mm,label={\tiny$u_1^5u_5$}]{B5}
\Vertex[x=2.5,y=4.5,size=.05,color=black,position=left,distance=2mm,label={\tiny$u_1^4u_5$}]{B4}
\Vertex[x=2.5,y=4.0,size=.05,color=black,position=left,distance=2mm,label={\tiny$u_1^3u_5$}]{B3}
\Vertex[x=2.5,y=3.5,size=.05,color=black,position=left,distance=2mm,label={\tiny$u_1^2u_5$}]{B2}
\Vertex[x=2.5,y=3.0,size=.05,color=black,position=left,distance=2mm,label={\tiny$u_1u_5$}]{B1}
\Vertex[x=2.5,y=2.5,size=.05,color=black,position=left,distance=2mm,label={\tiny$u_5$}]{B0}
\Edge[lw=0.75pt](B5)(B6)
\Edge[lw=0.75pt,bend=45](B3)(B5)
\Edge[lw=0.75pt](B3)(B4)
\Edge[lw=0.75pt,bend=-45](B2)(B4)
\Edge[lw=0.75pt](B1)(B2)
\Vertex[x=5,y=6.5,size=.05,color=black,position=left,distance=2mm,label={\tiny$u_1^7u_6$}]{C7}
\Vertex[x=5,y=6.0,size=.05,color=black,position=left,distance=2mm,label={\tiny$u_1^6u_6$}]{C6}
\Vertex[x=5,y=5.5,size=.05,color=black,position=left,distance=2mm,label={\tiny$u_1^5u_6$}]{C5}
\Vertex[x=5,y=5.0,size=.05,color=black,position=left,distance=2mm,label={\tiny$u_1^4u_6$}]{C4}
\Vertex[x=5,y=4.5,size=.05,color=black,position=left,distance=2mm,label={\tiny$u_1^3u_6$}]{C3}
\Vertex[x=5,y=4.0,size=.05,color=black,position=left,distance=2mm,label={\tiny$u_1^2u_6$}]{C2}
\Vertex[x=5,y=3.5,size=.05,color=black,position=left,distance=2mm,label={\tiny$u_1u_6$}]{C1}
\Vertex[x=5,y=3.0,size=.05,color=black,position=left,distance=2mm,label={\tiny$u_6$}]{C0}
\Edge[lw=0.75pt](C5)(C6)
\Edge[lw=0.75pt,bend=45](C3)(C5)
\Edge[lw=0.75pt](C3)(C4)
\Edge[lw=0.75pt,bend=-45](C2)(C4)
\Edge[lw=0.75pt](C1)(C2)
\Vertex[x=7.5,y=9,size=.05,color=black,position=left,distance=2mm,label={\tiny$u_1^7u_5u_6$}]{D7}
\Vertex[x=7.5,y=8.5,size=.05,color=black,position=left,distance=2mm,label={\tiny$u_1^6u_5u_6$}]{D6}
\Vertex[x=7.5,y=8,size=.05,color=black,position=left,distance=2mm,label={\tiny$u_1^5u_5u_6$}]{D5}
\Vertex[x=7.5,y=7.5,size=.05,color=black,position=left,distance=2mm,label={\tiny$u_1^4u_5u_6$}]{D4}
\Vertex[x=7.5,y=7,size=.05,color=black,position=left,distance=2mm,label={\tiny$u_1^3u_5u_6$}]{D3}
\Vertex[x=7.5,y=6.5,size=.05,color=black,position=left,distance=2mm,label={\tiny$u_1^2u_5u_6$}]{D2}
\Vertex[x=7.5,y=6,size=.05,color=black,position=left,distance=2mm,label={\tiny$u_1u_5u_6$}]{D1}
\Vertex[x=7.5,y=5.5,size=.05,color=black,position=left,distance=2mm,label={\tiny$u_5u_6$}]{D0}
\Edge[lw=0.75pt](D5)(D6)
\Edge[lw=0.75pt,bend=45](D3)(D5)
\Edge[lw=0.75pt](D3)(D4)
\Edge[lw=0.75pt,bend=-45](D2)(D4)
\Edge[lw=0.75pt](D1)(D2)
\end{tikzpicture}
\end{center}
\end{examp}

\end{document}